\documentclass[preprint,11pt]{imsart}
\pdfoutput=1
\PassOptionsToPackage{numbers}{natbib}

\RequirePackage[OT1]{fontenc}
\RequirePackage{amsthm,amsmath}
\RequirePackage[numbers]{natbib}
\RequirePackage[colorlinks=true, citecolor=blue, filecolor=black,  urlcolor=blue]{hyperref}
\usepackage{hyperref}

\usepackage[utf8]{inputenc} %
\usepackage[T1]{fontenc}    %
\usepackage{url}            %
\usepackage{booktabs}       %
\usepackage{amsfonts}       %
\usepackage{amssymb}
\usepackage{nicefrac}       %
\usepackage{microtype}      %
\usepackage{graphicx}
\usepackage{epstopdf}
\usepackage{algorithm, algorithmic}
\usepackage{bm}
\usepackage{bbm}
\usepackage{xspace}
\usepackage{tikz}
\usepackage{standalone}
\usepackage{mathtools}
\usepackage{subcaption}
\usepackage[textsize=tiny,textwidth=0.8in]{todonotes}
\usepackage[normalem]{ulem}
\usepackage{amsopn}
\usepackage{cleveref}
\usepackage{listings}
\usepackage{xcolor}

\usepackage[centering,letterpaper,margin=1in]{geometry}
\ifpdf
  \DeclareGraphicsExtensions{.eps,.pdf,.png,.jpg}
\else
  \DeclareGraphicsExtensions{.eps}
\fi

\usepackage{enumitem}
\setlist[enumerate]{leftmargin=.5in}
\setlist[itemize]{leftmargin=.5in}

\newtheorem{theorem}{Theorem}[section]

\newtheorem{definition}[theorem]{Definition}
\newtheorem{dataassumption}[theorem]{Data Assumption}
\newtheorem{definitiono}[theorem]{Optimization Problem}
\newtheorem{lemma}[theorem]{Lemma}

\newtheorem{example}[theorem]{Example} %
\newtheorem{remark}[theorem]{Remark} %
\numberwithin{equation}{section}

\DeclareRobustCommand{\textvtt}[1]{%
  \begingroup
  \ttfamily
  \hyphenchar\font=`\- %
  \setlength{\spaceskip}{0.8em plus 0.4em minus 0.2em}%
  \setlength{\xspaceskip}{1em plus 0.4em minus 0.2em}%
  #1%
  \endgroup
}

\makeatletter
\DeclareRobustCommand\vttfamily{%
  \not@math@alphabet\vttfamily\relax
  \fontfamily{cmvtt} %
  \selectfont
}
\DeclareTextFontCommand{\textvtt}{\vttfamily}
\makeatother
\setlist[itemize]{align=parleft,left=0pt..1em}

\definecolor{darkcolor}{RGB}{127,0,85}
\colorlet{numb}{magenta!60!black}
\lstdefinelanguage{json}{
    basicstyle=\fontsize{8.6}{11}\vttfamily,
    commentstyle=\color{black},
    stringstyle=\color{darkcolor},
    showstringspaces=false,
    breaklines=true,
    breakatwhitespace=true,
    frame=lines,
    string=[s]{"}{"},
    comment=[l]{:\ "},
    morecomment=[l]{:"},
    literate=
        *{0}{{{\color{numb}0}}}{1}
         {1}{{{\color{numb}1}}}{1}
         {2}{{{\color{numb}2}}}{1}
         {3}{{{\color{numb}3}}}{1}
         {4}{{{\color{numb}4}}}{1}
         {5}{{{\color{numb}5}}}{1}
         {6}{{{\color{numb}6}}}{1}
         {7}{{{\color{numb}7}}}{1}
         {8}{{{\color{numb}8}}}{1}
         {9}{{{\color{numb}9}}}{1}
         {é}{{\'e}}{1}
         {è}{{\`e}}{1}
         {ê}{{\^e}}{1}
         {à}{{\`a}}{1}
}

\newcommand{\rev}[1]{\textcolor{black}{#1}}

\newcommand{\R}{\mathbb{R}}
\newcommand{\de}{{d_{\textrm{e}}}}
\newcommand{\pa}{{p_{\textrm{a}}}}
\newcommand{\pe}{{p_{\textrm{e}}}}
\newcommand{\Loss}{\textsf{L}}
\newcommand{\J}{\textsf{J}}
\newcommand{\dv}{d_{\textrm{v}}}

\newcommand{\T}{\mathsf{T}}
\newcommand{\A}{\mathfrak{A}}
\newcommand{\dkl}{\mathsf{D}_{\mathsf{KL}}}
\newcommand{\dvg}{\mathsf{D}}

\newcommand{\Asf}{\mathcal{A}}
\newcommand{\Ac}{\mathsf{A}}
\newcommand{\ATT}{\mathsf{ATT}}
\newcommand{\SSM}{\mathsf{SSM}}
\newcommand{\thetas}{\theta^{\star}}

\newcommand{\mf}{\mathsf{f}}
\newcommand{\mv}{\mathsf{v}}
\newcommand{\mg}{\mathsf{g}}
\newcommand{\varthetaa}{\vartheta}
\newcommand{\varthetae}{\varphi}

\newcommand{\fp}{f^{\mathsf{prompt}}}
\newcommand{\ellp}{\ell_{\mathsf{prompt}}}
\newcommand{\FLp}{\F^{L,\mathsf{prompt}}}
\newcommand{\Bl}{\mathsf{B}}
\newcommand{\Em}{\mathsf{E}}
\newcommand{\rr}{r}
\newcommand{\re}{\mathfrak{r}}
\newcommand{\ii}{\mathsf{i}}
\newcommand{\io}{\ii_{\mathsf{original}}}
\newcommand{\ic}{\ii_{\mathsf{convolved}}}
\newcommand{\is}{\ii_{\mathsf{sequence}}}
\newcommand{\Ls}{L_{\mathsf{stop}}}

\newcommand{\D}{\mathsf{D}}
\newcommand{\Se}{\mathsf{S}}
\newcommand{\F}{\mathsf{F}}
\newcommand{\G}{\mathsf{G}}
\newcommand{\E}{\mathsf{E}}
\newcommand{\He}{\mathsf{H}}
\newcommand{\Voc}{\mathsf{V}}
\newcommand{\pl}{\mathsf{p}_{\ell}}
\newcommand{\pem}{\mathsf{p}_{m}}
\newcommand{\qem}{\mathsf{q}_{m}}
\newcommand{\ql}{\mathsf{q}_{\ell}}
\newcommand{\pip}{\pi^{\mathsf{prompt}}}
\newcommand{\pil}{\pi_{\ell}}
\newcommand{\pim}{\pi_m}

\newcommand{\pioL}{\pi_{1:L}}
\newcommand{\pol}{\mathsf{p}_{1:\ell}}
\newcommand{\poL}{\mathsf{p}_{1:L}}
\newcommand{\polm}{\mathsf{p}_{1:\ell-1}}

\newcommand{\qol}{\mathsf{q}_{1:\ell}}
\newcommand{\prob}{\mathcal{P}}
\newcommand{\cL}{\mathcal{L}}
\newcommand{\fol}{f_{1:\ell}}
\newcommand{\foL}{f_{1:L}}

\newcommand{\NN}{\mathfrak{NN}}
\newcommand{\NO}{\mathfrak{N}}

\newcommand{\NNF}{\mathsf{NN}}
\newcommand{\NOF}{\mathsf{N}}

\newcommand{\fil}{f_{\ell}}

\newcommand{\folm}{f_{1:\ell-1}}
\newcommand{\fomm}{f_{1:m-1}}
\newcommand{\im}{\mathsf{im}}
\newcommand{\fdata}{\mathsf{f}}

\DeclareMathOperator*{\argmax}{arg\,max}
\DeclareMathOperator*{\argmin}{arg\,min}

\begin{document}

\begin{frontmatter}

\title{Large Language Models:\\ A Mathematical Formulation}

\runtitle{Large Language Models}

\begin{aug}
\author[a1,a4]{\fnms{Ricardo} \snm{Baptista}%
\ead[label=e1]{r.baptista@utoronto.ca}},
\author[a2,a3]{\fnms{Andrew} \snm{Stuart}%
\ead[label=e2]{andrxstu@amazon.com}},
\and
\author[a2]{\fnms{Son} \snm{Tran}%
\ead[label=e3]{sontran@amazon.com}}
\runauthor{Baptista, Stuart, Tran}
\address[a1]{Statistical Sciences, University of Toronto, Toronto M5G 1X6, Canada\\}
\address[a4]{Vector Institute, Toronto, M5G 0C6, Canada\\}
\address[a2]{Amazon Search, Palo Alto, CA 94301, USA\\}
\address[a3]{Computing and Mathematical Sciences, California Institute of Technology, Pasadena CA 91125, USA\\
\vspace{0.1in}
\printead*{e1}, \printead*{e2}, \printead*{e3}}
\end{aug}

\begin{abstract}
Large language models (LLMs) process and predict sequences containing text to answer questions, and address tasks including document summarization, providing recommendations, writing software and solving quantitative problems. We provide a mathematical framework for LLMs by describing the encoding of text sequences into sequences of tokens, defining the architecture for next-token prediction models, explaining how these models are learned from data, and demonstrating how they are deployed to address a variety of tasks. The mathematical sophistication
required to understand this material is not high, and 
relies on straightforward
ideas from information theory, probability and optimization. Nonetheless, the
combination of ideas resting on these different components from the
mathematical sciences yields a complex algorithmic structure; and
this algorithmic structure has demonstrated remarkable empirical successes. 
The mathematical framework established here
provides a platform from which it is possible to formulate and address 
questions concerning the accuracy, efficiency and robustness of the algorithms 
that constitute LLMs. The framework also suggests directions for development 
of modified and new methodologies.
\end{abstract}



\end{frontmatter}

\section{Introduction} \label{sec:introduction}

\subsection{Goal and Organization of Paper} \label{ssec:GOP}

Large language models (LLMs) have, in only a few years, revolutionized the field of artificial intelligence (AI). Their impact is being
felt across society: they can help with tasks as diverse as the writing of computer code and the planning of a vacation itinerary; they
can hold conversations; they can process medical information and ask follow-up questions to gather more information and provide potential diagnoses; they %
can retrieve and summarize information from different sources including legal documents and scientific papers;  %
and they can interact with other sources of information such as images, audio and video. Their
potential to permanently change our way of life is clear. 

Yet the methodology, however impressive, is prone to errors, is not equipped with any rigorous logical or inferential constraints, comes with limited guarantees and consumes enormous resources, both to create and to deploy. This situation presents, then, as an opportunity for mathematicians to help to address these shortcomings by developing, and then studying, utilizing and evolving, a fundamental picture of how LLMs work. It is, however, not easy to find a 
complete presentation of how LLMs work, aimed at a mathematically inclined 
audience. This is 
because published work in the AI community relies primarily on
communication through a combination of high-level visual sketches of AI algorithms, made precise through supporting code
which is publicly available through, for example, Github. This admirable approach 
\rev{to open-source the software that underpins published}
results has been central to the success of AI \cite{donoho2024data}; it has also revolutionized other fields, such as computational science and engineering, which now strive to meet the same standards of reproducibility as the AI community. However, the reliance on published code to precisely define the methodology underlying LLMs makes it difficult to formulate and address conceptual questions whose study might help to address the aforementioned shortcomings of LLMs. The goal of this paper is to present a precise mathematical formulation of LLMs; in doing so we aim to make the subject of LLMs accessible to mathematicians, and to provide a framework for subsequent computational and theoretical studies of LLMs by mathematicians.

The objects that are manipulated in LLMs are sequences; Section \ref{sec:seq} sets up the notation that we employ for sequences, and describes different representations of language, as sequences of \emph{tokens} and sequences in Euclidean space, that are relevant to the formulation of LLMs. In Section \ref{sec:ntp} we describe the core methodology by which the \rev{majority of currently deployed LLMs work:} \emph{autoregressive factorization} of the probabilistic model for sequences, and \emph{next-token prediction}; this includes
a high-level description of the parameterization of models used for next-token prediction. Section~\ref{sec:train} then describes
the probabilistic objective function used in the training of these models, and the use of large corpora of data taken from the internet to define
this objective. In Section~\ref{sec:arch} we detail the %
parameterization of models for next-token prediction, introducing the concepts of \emph{positional encoding}, the \emph{attention mechanism} and the \emph{set transformer}; we also discuss the computational complexity of these
parameterizations, and describe modifications, such as using 
\emph{state-space models} in place of the attention mechanism, which have proven efficient in certain specific scenarios. Section~\ref{sec:seqgen} discusses different ways in which the 
next-token prediction model is used for \emph{sequence generation}. In Section \ref{sec:applications} we study various applications of these LLMs based on autoregressive factorization. Section~\ref{sec:evaluation} presents a framework to evaluate LLM's sequence generation capabilities on various tasks. In Section~\ref{sec:mask} we describe an alternative emerging approach to autoregressive  LLMs: generalization of diffusion-based generative modeling to discrete sequence space. Section~\ref{sec:SAC} contains our concluding summary and discussion of challenges.

The Appendices contain mathematical details that
are of interest, but which would break the flow if contained in the
main body of the paper.
Section~\ref{sec:deriv} provides details of derivatives of the model, with respect to parameters to be learned; these derivatives are required to employ gradient descent-based methods in the training of LLMs.
Section~\ref{sec:arch_probabilistic} provides a probabilistic perspective on 
the attention architecture; this perspective is of value when considering
the use of the attention mechanism beyond the confines of LLMs.
In Section~\ref{sec:ssmcts} we describe a continuous time perspective
on state-space models. This perspective underpins the original
development of the methodology as a replacement for the attention mechanism;
it is also useful when considering applications beyond the confines of LLMs. In Section~\ref{app:discrete_reversibility} we present a proof on the reversibility of discrete diffusions; this underlies the sequence generation methodology
described in Section~\ref{sec:mask}. Section~\ref{sec:LLMresponses} contains additional details and responses from LLMs for the applications presented in Section~\ref{sec:applications}.

\rev{We are describing a fast-moving field. In doing so we have made certain choices
in order to deliver a crisp mathematical presentation. In particular we focus primarily on autoregressive models and next token prediction; and, within that class of models, we have
focused on attention-based parameterizations of next-token prediction, and hence
transformers. This pair of choices leads to description of the LLMs that dominate the landscape today, and have done so for the last half-decade in which LLMs have
flourished. But it should be understood by the reader that we are not making any claims for the supremacy of autoregressive, attention-based LLMs. To the contrary we 
expect significant evolution over the coming decade and more. For this reason we
have included, in Section \ref{sec:arch}, discussion of emerging competitors to
the attention-based parameterization of next token prediction; and in Section
\ref{sec:mask} we have described diffusion-based language models which eschew autoregressive factorization, and work instead on entire sequences of text at the same time. We firmly believe that the mathematization of LLMs, and indeed
of many other algorithms arising in AI, will have significant value, and can play an
important role, for many years to come. Despite the specific choices we have made to align our presentation of LLMs with current practice, it is our belief that the approach we adopt here to formalizing AI algorithms will have lasting value beyond
the specifics of the current state-of-the-art in LLMs.}

\subsection{Literature Review} \label{ssec:literature}

Recent developments in LLMs are based on both theoretical foundations, and algorithmic innovations. It is argued persuasively in~\cite{donoho2024data}, that the impressive empirical success of LLMs is founded on their access to vast amounts of data, benchmarks agreed across the community and the sharing of computer code aimed at testing against these benchmarks. The resulting empirical studies have driven
innovation resulting in popular commercial models such as LLaMa~\cite{touvron2023llama}, GPT~\cite{achiam2023gpt}, Gemini~\cite{team2023gemini}, and Claude~\cite{claude4_2025}; we refer to a recent survey of this field in~\cite{zhao2023survey}. LLMs as a tool for mathematicians have been reviewed in~\cite{frieder2023large}.

\rev{A core ingredient of LLMs is the use of \emph{tokenization}. This} has two key components: the introduction of a larger vocabulary of \emph{tokens} in the service of 
shortening sequence length; and the identification of elements in this
vocabulary of tokens with points in Euclidean space. The first component
is typically achieved through \emph{byte-pair encoding}, which was introduced as a method of file compression in~\cite{gage1994new}, and appeared as an algorithm in the context of sequences in~\cite{ziv2003universal}.  The second component, identifying each token, which is an element of a large, but finite vocabulary, with a point in Euclidean space, \rev{is undertaken as part of the training of the LLM architecture~\cite{vaswani2017attention, devlin2019bert}.} Byte-pair encoding assigns a new vocabulary element to frequently occurring 
adjacent pairs of vocabulary elements, and then applies this idea recursively. The method is modified for language modeling in~\cite{sennrich2015neural} by adapting the algorithm to merge characters and produce a finite vocabulary of subword, and word, units, known as \emph{tokens}. Machine translation techniques and transformers process these tokens to model language. 
Extensions to continuous token spaces with an infinite vocabulary have also been proposed~\cite{zhang2025flexible}. Much of the text used in training LLMs is longer than the context window length used to take text as input. Decomposition of such larger text sequences, and merging them, so that the new units of text fit into the context window, can be done in multiple ways; optimizing for issues like memory allocation and information content is important in the decomposition process~\cite{ding2024fewer}. 

\rev{A second key ingredient is the autoregressive factorization of the learnable
family of models for sequence generation in token space~\cite{bengio2003neural, radford2018improving}. We will concentrate
on this factorization for much of the paper. However, as explained at the end of the previous subsection, it is important to appreciate that other competitive approaches
are starting to emerge.} In addition to sequential processing of sequences based on autoregressive decoding, other architectures have been proposed for text understanding and generation. These include BERT~\cite{devlin2019bert,wei2022chain}, which learns to unmask an arbitrary subset of elements of a sequence based on the surrounding context, which is not necessarily contiguous in the sequence. Discrete diffusion models are an alternative procedure to generate an entire block of text using a time-dependent denoising process rather than sequentially generating each subsequent element of a sequence. Discrete diffusion models for finite state-space sequence generation were presented in discrete time in~\cite{hoogeboom2021argmax, austin2021structured} and extended to a continuous-time setting in~\cite{campbell2022continuous}. Several choices for the noising process and objective functions have been proposed to learn the denoiser based on a discrete analog of score matching~\cite{lou2024discrete, shi2024simplified, sahoo2024simple} and variational inference~\cite{sun2024generative}. These methods currently achieve competitive performance to GPT-2 scale autoregressive models, but do not yet outperform more recent versions of GPT and similar commercial models. 

\rev{The final key ingredient underlying many currently deployed LLMs is the \emph{transformer} architecture, used within autoregressive factorizations.}
This method for processing sequential data was first
introduced in \cite{vaswani2017attention}. Transformers are studied from a mathematical perspective in~\cite{turner2023introduction}; and properties of
the transformer architecture are analyzed in detail in~\cite{geshkovski2025mathematical, chen2025quantitative, geshkovski2024measure, castin2025unified}. The transformer is based on the \emph{attention mechanism}, a mapping that takes sequences into sequences. When the parameters defining the mechanism are suitably tuned to the data, the mapping can be used to learn about correlations between  different elements in the input sequence. Attention is formulated mathematically in~\cite{calvello2024continuum} and analyzed in detail in~\cite{geshkovski2023emergence, geshkovski2024dynamic, karagodin2024clustering, calvello2024continuum}. In addition to their use in natural language processing, transformer architectures have been applied in other contexts including computer vision~\cite{dosovitskiy2020image}, time-series analysis and forecasting~\cite{ansari2024chronos}, nonlinear filtering~\cite{bach2025learning}, protein folding~\cite{jumper2021highly} and sequential decision making~\cite{chen2021decision}. LLMs have also been tuned to improve their reasoning abilities to solve mathematical problems in~\cite{shao2024deepseekmath}. We refer to~\cite{phuong2022formal} for the construction of many formal algorithms based on the transformer architecture. Several algorithmic variants on transformers have also been introduced to improve their efficiency for long sequences; these include approaches to reduce the quadratic complexity of standard self-attention~\cite{choromanski2020rethinking, dao2022flashattention, child2019generating, team2025kimi}. Other components of transformers, such as normalization for training stability---typically achieved using a process known as layer normalization~\cite{ba2016layer, zhang2019root}---have also been analyzed and explored through various methodological approaches~\cite{xiong2020layer, zhu2025transformers}. 
\rev{We also re-emphasize the point stated at the end of the previous subsection that, although our presentation is focused primarily on the widely-used, attention-based transformer architecture, new and potentially competitive architectures are
emerging. We highlight in particular the state-space model Mamba~\cite{gu2023mamba},
which reduces the quadratic (in sequence length) complexity of attention, and
demanding memory requirements, to linear scaling and more manageable memory requirements. Mamba is starting to have significant impact in some domains. Furthermore the RWKV approach in~\cite{peng2023rwkv} which, like Mamba, builds
on a recurrent structure to obtain linear complexity in sequence length and
reduce memory demands, is also showing promise. In approaches built on
recurrent structure there is a tension between computational/memory demands
and expressivity of the resulting architecture. Recent work attempts to balance these by using a hybrid combination of transformer and state-space approaches \cite{blakeman2025nemotron,lahoti2026mamba}.}

After training, the aim of LLMs is to generate text responses given starting prompts or questions, typically outside of the sequences seen during training. This process is performed using a deterministic or stochastic \emph{decoding} procedure, that sequentially generates the tokens, from first to last, conditioned on preceding tokens. This concept of building a model to predict the next most likely word given the context of a sequence of words first originated in work by Shannon~\cite{shannon1948mathematical}.
Several variants have been proposed to improve the correctness of the responses and avoid \emph{hallucination}, whereby the model produces factually incorrect or inconsistent content, despite being readable. To improve the coherence and variety of generated text, top-k sampling was introduced in~\cite{holtzmancurious}. In the context of LLMs, \emph{chain-of-thought reasoning} refers
to models which generate responses explicitly outlining the reasoning process behind them. Recently, eliciting chain-of-thought reasoning by adapting the prompt~\cite{wei2022chain} or changing the sampling strategy~\cite{wang2024chain, wangself} has been found to yield more factual responses in some settings.

\rev{Lastly we mention that the autoregressive} approach to text sequence generation has been extended to processing and generation of other data types, often referred to as multimodal LLMs. This data includes audio, images, and video, and can be continuous, semi-continuous or discrete. Common approaches~\cite{mizrahi20234m, wu2024next, zhan2024anygpt} embed all data types into the embedding space for a language model and use autoregressive decoding and prediction to process and generate text or data from other modalities.

\section{Sequences}
\label{sec:seq}

This section is foundational in relation to the question of how language is represented in, and manipulated by, algorithms rooted in
(flotating point approximation of) Euclidean space. We start in Subsection \ref{ssec:notseq}  by establishing  notation for finite sequences, over two \emph{vocabularies} (finite sets), denoted $\Voc_0$ and $\Voc$, and over the Euclidean space $\R^{\de}.$ The discrete set $\Voc_0$ defines the initial vocabulary used to describe inputs. The discrete set $\Voc$ defines a
second vocabulary, defined from $\Voc_0$ by a process known as byte-pair encoding, described in Subsection \ref{ssec:bpe}; this process
results in a set $\Voc$ satisfying $|\Voc| \gg |\Voc_0|.$
There is a bijection between sequences taking values in $\Voc_0$ and in $\Voc.$ With respect to the pairing implied by this bijection 
the sequence taking values in $\Voc$ will be much shorter than the sequence taking values in $\Voc_0,$ reflecting a natural trade-off between 
size of the alphabet and length of the sequence. In Subsection~\ref{ssec:mle} we describe how each element in $\Voc$ is identified with a point in
$\R^{\de}.$ This identification is a crucial step as it allows 
for \emph{interpolation} in the space of text; the interpolation is performed
in the Euclidean representation of the text sequence. 

\subsection{Notation for Sequences} \label{ssec:notseq} 

A central component in our framing of LLMs is the manipulation of sequences defined over the finite integer set $\Se^M = \{1,\dots,M\}$ for various $M.$ 
We will consider three types of sequences: the first where each element  takes value in finite discrete set $\Voc_0$, the second 
where each element  takes value in finite discrete set $\Voc$, and the third  where each element takes value $\R^{\de}$. We write 
\begin{equation*}
   \F_0^M := \{f: \Se^M \rightarrow \Voc_0\}, \quad \F^M = \{f: \Se^M \rightarrow \Voc\}, \quad \rev{\F_{e}^M = \{f: \Se^M \rightarrow \R^{d_{e}}\}}.
\end{equation*}
\noindent 
For $f \in \F^M$ we sometimes write $f_{1:M}$ to emphasize the domain of the sequence; and we write $f_m \in\Voc$ for the $m^{th}$ entry of
the sequence. Similar notation extends naturally to $\F^M_0$ and $\F^M_e.$ %
We adopt the convention that $f_{1:0}=\varnothing.$ A commonly occurring real-valued sequence will be an element of $\F^M_e,$ where $\R^{\de}$ is an \emph{embedding space.}\footnote{\rev{Note that later sections will consider other embedding dimensions with letter $e$ replaced by $l$
in $\R^{d_l}$
(when discussing images) and replaced by $c$ (when denoting channels) in $\R^{d_c}$, with analogous embedding spaces $\F^M_l$ and $\F^M_c$, respectively.}}
Finally we define $\Se^{M,M}:=\Se^M \times \Se^M.$

\subsection{Tokenization and Byte-Pair Encoding} \label{ssec:bpe}

Consider an initial vocabulary $\Voc_0$  used to define input text as an element of $\F_0^{M_0}.$ The goal of \emph{compression} is to reduce the length of an input sequence by identifying a new, much larger, vocabulary $\Voc$, containing the initial vocabulary, and representing the text as an element in $\F^M$ for some $M$, typically much smaller than $M_0.$ Compression can be achieved by creating a vocabulary $\Voc$ from $\Voc_0$ which is based on frequently-occurring strings of characters in $\Voc_0$. The vocabulary elements that result from this process are known as \emph{tokens}. Currently \emph{byte-pair encoding} \cite{gage1994new} is widely used for tokenization in LLMs~\cite{rajaraman2024toward}; we explain this approach to encoding by means of examples. 

\begin{example} \label{ex:21} Consider a simple example of compression, starting with initial alphabet $\Voc_0 = \{a,b,c\}$. Given the input sequence $f^{(0)} = \{abcababc\}$, the first iteration counts the occurrences of each pair of neighboring elements. The most frequently occurring pair is identified, in this case $ab,$ and then
labeled $d = ab;$ it is then added to the vocabulary and $\Voc_1 \coloneqq \Voc_0 \cup {d}$ is defined. In this new vocabulary, the input sequence is represented as $f^{(1)} = \{dcddc\}$. The idea of creating new vocabulary elements based on frequently occurring pairs is now recursed on. The second iteration identifies the element $dc$ to be the most frequently occurring pair, defines the new element $e = dc$ and adds it to the vocabulary, defining $\Voc_2 \coloneqq \Voc_1 \cup {e}$. The sequence is now represented as $f^{(2)} = \{ede\}$. This sequence cannot be compressed further using the notion of most frequently occurring neighbouring pairs. We thus define the final new vocabulary to be $\Voc := \Voc_2=\{a,b,c,d,e\}.$ In practice, this iterative  procedure can also be stopped early after a pre-defined number of levels of recursion; 
in this simple example the algorithm might be stopped after one level and then choose $\Voc := \Voc_1=\{a,b,c,d\}$, for instance, rather than after the two levels which lead to termination.
\end{example}

\begin{example} \label{ex:22}
Consider application of byte-pair encoding to an initial vocabulary $\Voc_0$ defined by the English language; this will contain the standard
lower case alphabet $\{a,b,\cdots, y,z\}$, its upper-case version, as well as elements representing punctuation, spaces, new paragraph, start of text, end of text and the like;  
it may also contain symbols from computer code
and include other alphabets, such as the Greek alphabet $\{\alpha, \beta, \cdots,  \omega\}$. Then $|\Voc_0| = \mathcal{O}(10^2)$.
Use of byte-pair encoding, based on studying frequencies in large text corpuses from the internet, leads to $|\Voc| = \mathcal{O}(10^4)$~\cite{radford2019language}. 
English language text of length several paragraphs can be represented as a sequence of length $M=\mathcal{O}(10^2)$ in the encoded vocabulary~\cite{cohere_tokens_2025}.
\end{example}

From now on we set $\Voc:=\{1, \cdots, |\Voc|\}$ noting that we may use this set of integer labels for the elements of $\Voc$ without
loss of  generality. Thus $\Voc \equiv \Se^{|\Voc|}.$

\begin{example}[Example \ref{ex:21} continued]
We have $\Voc=\{a,b,c,d,e\}.$ However without loss of generality we may relabel these elements as $\Voc=\{1,2,3,4,5\}.$
\end{example}

\begin{remark} \label{rem:gabriele}
When byte-pair encoding is used and recursed to significant depth.
most commonly occurring words end up being represented by a single token.
\end{remark}

\begin{remark} \label{rem:byte}
Whilst byte-pair encoding has proven to be successful empirically, it is conceivable that other encodings might be desirable in some
settings. In particular the trade-off between sequence length and vocabulary size may play out differently, in different tasks or on
different architectures, in terms of associated computational cost for training and evaluation.
\end{remark}

\subsection{Making Language Euclidean} \label{ssec:mle}
We now discuss how byte-pair encoded sequences in $\F^M$ can be represented as sequences 
in Euclidean space, that is as elements in $\F^M_e.$ This is done by assigning a unique element in $\R^{\de}$ to each point in $\Voc;$ and indeed this assignation is typically be learned as part of the training of the LLM from data. In this subsection the goal is to set-up a notation for this assignation process. The basic building block is the following definition of function $\phi: \Voc \to \R^{\de}$. Recall that $\Voc=\{1, \cdots, |\Voc|\}=\Se^{|\Voc|}.$

\begin{definition} \label{def:thetam}
Let $\varthetae \in \R^{\de \times |\Voc|}$ be a matrix, where $\varthetae =(\varthetae_1, \cdots, \varthetae_{|\Voc|});$ here $\varthetae_v \in \R^\de\,\, \forall\, v \in \Voc.$
Then, define $\phi: \Voc \to \R^{\de}$ by $\phi(v)=\varthetae_v\, \forall\, v \in \Voc.$
\end{definition}

Given $f \in \F^M$ we extend $\phi$ to act on sequences:

\begin{definition} \label{def:phis}
Define $\phi: \F^M \to \F^M_e$ by $(\phi \circ f)_m=\phi(f_m).$
\end{definition}

\begin{remark}
The construction of $\phi$ to act on sequences in $\F^M$ is parameterized by $\varthetae$. The parameterization is independent of sequence length $M.$ This is because $\phi$ is defined to act elementwise on the sequence. We will see later that the other learnable parameters in LLMs are also defined in a manner which is independent of sequence length. This is an important
feature of the parameterization of LLMs as it enables a common model to be
deployed to manipulate sequences of different lengths.
\end{remark}

\begin{remark} \label{rem:close}
It is natural to ask whether $\varthetae$ could be defined, or learned, independently of LLM training. In particular there is nothing in the current learning of the embedding which ensures that, for example, two words which are synonymous with one another are close when embedded in Euclidean space; indeed two words which are synonymous may even be comprised of different numbers of tokens. For example, if one such word is used infrequently and is not represented as a single token, whilst the other is frequently used and is represented as a single token, then the two words are not even embedded in the same Euclidean space. Addressing this issue might involve moving beyond byte-pair encoding; see Remark \ref{rem:byte}.
\end{remark}

\subsection{Encoding Other Data Modalities} \label{sec:encoding_other_modalities}

Here we discuss the representation of input modalities, other than text, 
as a sequence. This is the basis for \emph{multimodal LLMs} which
facilitate communication between different data modalities.
We start by focusing on images. 
Given a bounded domain $D \subset \mathbb{Z}^2$, an image is given by a mapping $f \colon D \rightarrow \R^{d_c}$ where $d_c$ is a channel dimension; this is taken to be $d_c = 1$ for grayscale images and $d_c = 3$ for RGB images. The 
domain $D$ can be flattened into a set with $M = |D|$ elements that is indexed by a single parameter so that the image $f$ can be represented in the space of sequences $\F_{c}^M.$
This is analogous to the space $\F_{e}^M$ for the Euclidean embedding of text sequences.

For text sequences it has been empirically observed
that it is advantageous to shorten the length of sequences at the expense of a larger vocabulary. A similar trade-off is present
for images, with channel width playing the role of vocabulary size.
The next example illustrates this.

\begin{example}
The CLIP algorithm~\cite{radford2021learning} is a contrastive learning approach~ \cite{baptista2025mathematical} used in the alignment of text and images. It employs a methodology to process images as sequences that trades sequence length with
channel width. To explain CLIP in a concrete context we assume that we are dealing with RGB images, and that $D=\Se^{K',K'}.$ 
Then the image can be viewed as $\io: \Se^{K',K'} \to \R^3.$  
By application of a linear convolution operator this is mapped to $\ic: \Se^{K,K} \to \R^{d_l}$. This
is then flattened to obtain $\is: \Se^{K^2} \to \R^{d_l}.$ If $d_l$ coincides with the text embedding dimension $\de$ then the resulting
sequence is indeed in $\F_e^M.$ In CLIP, the specific values used for the dimensions are $K'=224$, $K=7$ and $d_l=768$; furthermore $d_l$ is indeed chosen to coincide with the text embedding dimension $d_e$. Noting that $768 \gg 3$ 
and that $K^2 \ll (K')^2$ we see that, as with text, we have again shortened
the length of the sequence at the expense of increasing some measure of the
size of the space in which each element of the sequence lies. 
\end{example}

\begin{remark} \label{rem:atebit}
For 8-bit RGB images, the intensity of given colour at each pixel is often given by an integer between $0$ and $255$ rather than an arbitrary real number. Hence, images can also be represented as mappings with discrete outputs, i.e., $\io: \Se^{K',K'} \rightarrow \{0,\cdots,255\}^{3}$, which corresponds to a tensor of integers of size $K' \times K' \times 3$.
However, in practice these images are treated as having continuous outputs, which become the inputs to neural network-based vision encoders; see Section~\ref{ssec:mmode} for more details on the encoding process.
\end{remark}

Other frequently occurring modalities include audio and video. Audio, 
similarly to text, is best thought of as a sequential time-series. 
In many circumstances this sequences takes values in Euclidean space; but, 
as discussed in Remark \ref{rem:atebit} in the setting of RGB images, 
this information is often discretized and represented in a set of integers. 
Video is processed by treating the data as a sequence of images, also known as frames. Hence, for $D=\Se^{K',K'}$, a video can be described by a space-time mapping $f \colon [0,T] \times D \rightarrow \R^{\dv}$ for some final time $T$. In practice, the time-domain is discretized based on the frequency of image capture, known as the frame rate. For example, a video can be discretized into a finite number $D_T$ of frames; this results in a mapping from $D_T \times D$ to $\R^{\dv}$. To reduce the memory required to store a video and reduce redundant information across frames, it is often necessary to further sub-select frames, which are not necessarily equally spaced in time; see~\cite{hu2025m} for a recent approach to frame selection.
Finally we note that video and audio are often coupled, through a common 
time-indexing to create audio-visual data. Multimodal LLMs are being
developed to co-process text, image, audio, video and audio-visual data.

\section{Next-Token Prediction} \label{sec:ntp}

Next-token prediction is central to LLMs. For a sequence of length $m$, 
$f_{1:m} \in \rev{\F^m}$,  
next-token prediction works by assigning probabilities of 
token $f_m$ taking values at each element in $\Voc=\Se^{|\Voc|}$, conditional 
on the entire preceding sequence $f_{1:m-1} \in \rev{\F^{m-1}}$ of tokens.
Throughout this section we use $\prob(\cdot)$ to denote a probability space, either over $\Voc$ and denoted $\prob(\Voc)$, or over
sequences or arbitrary length \rev{$m$}, where each element takes values in $\Voc$, and denoted $\rev{\prob(\F^m)}$. 
The starting point of next-token prediction is a probability model for sequences $f_{1:\ell} \in \F^\ell$, for $\ell=1, \dots, L$,
where $L$ denotes a pre-specified maximum \emph{context length} $L$.\footnote{Sometimes also known as the \emph{context window}.} 

In Subsections~\ref{ssec:autof} we describe \emph{autoregessive} factorizations of probability distributions on sequences; this factorization builds in the
idea of next-token prediction into the factorization. Then in Subsection \ref{ssec:approxf} we describe the form of approximations we
use for this autoregressive factorization. 
Recall from Definition~\ref{def:phis} that application of $\phi$
to $f_{1:m-1} \in \F^{m-1}$ embeds the sequence $h_{1:m-1}$ in 
Euclidean space $\F^{m-1}_e.$ In order to take advantage of advances in neural networks and related architectures for classification defined for inputs in Euclidean space, we make use of approximation architectures that take
as input the sequence $h_{1:m-1}:=\phi \circ f_{1:m-1} \in \F^{m-1}_e.$ 

\subsection{Autoregressive Factorization Of Probability Distributions On Sequences} \label{ssec:autof}

\rev{The starting pointing for this subsection is the observation that
$$\poL(\foL)= \pl(f_L|f_{1:L-1})\pl(f_{1:L-1}).$$
Recursing on this identity leads to the heart of next-token prediction, namely
the \emph{autoregressive factorization} of $\poL \in \prob(\F^L)$ given by
\begin{equation}
    \label{eq:auto}
\poL(\foL)=\Pi_{\ell=1}^L \pl(f_\ell|f_{1:\ell-1}).
\end{equation}
We will also need the analogous autoregressive factorization of $\pol \in \prob(\F^\ell)$ for subsets of sequences of any length $\ell \in \Se^L:$
\begin{equation}
    \label{eq:auto2}
\pol(\fol)=\Pi_{m=1}^\ell \pem(f_m|f_{1:m-1}).
\end{equation}
This is the statement that the probability of a given sequence of tokens
of length $\ell$ may be expressed as the product of the probabilities of
the token in the $m^{th}$ position, given the sequence of tokens preceding it,
for $m=1, \dots, \ell.$
}

Here, for each input sequence $\fomm \in \F^{(m-1)}$,
$\mathsf{p}_m(\cdot|\fomm) \in \mathcal{P}(\Voc).$ Alternatively we may write $\mathsf{p}_m(\cdot|\cdot): \F^{(m-1)}_e \to \mathcal{P}(\Voc).$
We refer to $\mathsf{p}_m(\cdot|\fomm) \in \mathcal{P}(\Voc)$ as the \emph{target next-token probability} at position $m$ in the
text sequence, noting that it is a conditional probability of token at index $m$, given the sequence $\fomm$ of tokens seen up to $m-1.$

\begin{remark} \label{rem:31}
The discrete probability distribution on $\pol \in \prob(\F^{\ell})$, any $\ell=1,\dots, L$,  or on 
$\pl(f_\ell|f_{1:\ell-1}) \in \prob(\Voc)$, any $\ell=1,\dots, L$, should be thought of as being defined
empirically by the set of all text available on the Earth at this moment in time, either on the internet or in hard-copy documents. 
In practice an approximation is learned through accessing a (vast) subset of such data from the internet. More information on the training data is provided in Remarks~\ref{rem:training_data} and~\ref{rem:data_processing}.
\end{remark}

LLMs work by \emph{rolling out} sequences in $\Voc$ step-by-step, starting with a sequence of a given length, then increasing the sequence length by one in steps until it reaches length $L$. The main part of this process is to take 
a sequence of length $m-1$ and predict the probability, for each token in 
$\Voc$, of occupying the $m^{th}$ entry of the sequence. 
The probabilistic structure can be used to roll-out sequences in various different ways, which we describe in 
detail in Section \ref{sec:seqgen}.

\begin{example}[Example \ref{ex:22} continued] \label{ex:31}
All production LLMs have a maximum number of tokens that are processed at once. For instance, the popular LLaMA-2 $7$ billion model has a context window of $L = 4096$ tokens~\cite{touvron2023llama}\footnote{Recent LLaMA models with more parameters and specialized fine-tuning support up to $L = 128K$ tokens~\cite{grattafiori2024llama}}. This length corresponds, in a byte-pair encoded alphabet, to up to several pages of text. The CLIP model for alignment
of (for example) text and images uses $L=77$, corresponding, in a byte-pair encoded alphabet, to up to several paragraphs of text~\cite{radford2021learning}.
Linking to Example \ref{ex:22} demonstrates the trade-off being made here between vocabulary size and sequence length. Both have implications
for computational efficiency: the attention-based architectures, described 
in detail in Section \ref{sec:arch}, have cost which scales quadratically in
sequence length, to evaluate; on the other hand the size of vocabulary has impact on memory requirements and on the training problem
defined in Section~\ref{sec:train}. Understanding how to optimize this trade-off between sequence length and vocabulary size remains an open
question.
\end{example}

\subsection{Approximation Of Probability Distributions On Sequences} \label{ssec:approxf}

\rev{We seek a $\theta-$parameterized family $\pioL(\foL;\theta)$ from which to find an approximation of $\pol(\fol)$ given in \eqref{eq:auto}. 
Here $\theta \in \Theta \subseteq \R^p.$ Motivated by \eqref{eq:auto} 
it is natural to seek the approximation family in autoregressive form:}
\begin{equation}
    \label{eq:auto3}
\pioL(\foL;\theta)=\Pi_{\ell=1}^L \pil(f_\ell|f_{1:\ell-1};\theta_\ell).
\end{equation}
\rev{Then, if $\theta_\ell \in \Theta_\ell \subseteq  \R^{p_\ell}$ and if we set $\theta = (\theta_1,\dots,\theta_L)$, we have $p=\sum_{\ell=1}^L p_\ell$. 
We thus wish to approximate, for each $m$, the next-token conditional probability $\mathsf{p}_m(\cdot|f_{1:m-1})$ by choosing a model distribution $\pim(\cdot|f_{1:m-1};\theta_m) \in \mathcal{P}(\Voc)$.} 

\rev{Because $L$ may be of $\mathcal{O}(10^4)$ or more (see Example \ref{ex:22}), providing enough independent data to learn a parameter $\theta$ of the form
\eqref{eq:auto3} is challenging. Furthermore the methodology would require 
storing a large number of models and parameters. These complexities are typically addressed by \emph{using the same parameter $\theta$ to approximate every conditional probability.} This results in the model}
\begin{equation}
    \label{eq:auto3_shared}
\pioL(\foL;\theta)=\Pi_{\ell=1}^L \pil(f_\ell|f_{1:\ell-1};\theta).
\end{equation}
\rev{This simplification only requires learning a single parameter $\theta \in \Theta \subseteq \R^p$. However, we have created a new challenge: handling conditioning sequences of varying length. The \emph{transformer} architecture, which has proved successful
at addressing this challenge, is presented in detail in Section~\ref{sec:arch}.} 

\rev{We now split parameter $\theta$ into two components with different roles.
We set $\theta=(\varthetaa,\varthetae)$. The component parameter $\varthetae$ is associated with embedding the sequence $f_{1:m-1}$ in a space of sequences taking values in Euclidean space, of size proportional to $m$, rather than the discrete token space $\Voc$. Component parameter $\varthetaa$ is associated with learning how to interpolate in the resulting Euclidean spaces. We now get into the
details of why, and how, this is done.}

The parameterization of the conditional probability
$\pim(\cdot|f_{1:m-1};\theta)$ by sequence $f_{1:m-1} \in \F^{(m-1)}$ makes learning from data difficult, because $\Voc$, and hence $\F^{(m-1)}$, is a discrete space with no notion of continuity. Thus we seek to embed $\Voc$ in Euclidean space $\R^{\de}$ for some embedding dimension $\de.$ To do this, we use the \rev{map $\phi: \Voc \to \R^{\de}$ from Definition~\ref{def:thetam};  we recall that $\phi$ is defined by matrix $\varthetae \in \R^{\de \times |\Voc|} \cong \R^{\pe}.$ We also recall the extension of $\phi$ to act on $\Voc-$valued sequences, defined component-wise
along the sequence, from Definition~\ref{def:phis}. Next, we introduce interpolation parameter $\varthetaa \in \R^{\pa}$
which will enable us to perform interpolation with respect to the conditioning
sequences.} Then $\theta \in \R^p$, where $p=\pa+\pe$ and $\pe=\de|\Voc|$.  
We seek to parameterize $\pim$ in the form
\begin{equation} 
\label{eq:qarc}
\pim(f_m|f_{1:m-1};\theta):=\qem\bigl(f_m|\phi \circ f_{1:m-1};\varthetaa\bigr).
\end{equation}
\rev{The conditional probability $\qem$ now takes Euclidean-valued sequences as input,
because of the mapping $\phi$, which is represented by matrix $\varphi.$ The
conditional probability is
parameterized by $\varthetaa$ to enable learning of the dependence on
Euclidean-valued sequences. For given} $\varthetaa$,  $\qem(\cdot\,|\cdot\,; \varthetaa): \F^{(m-1)}_e \to \mathcal{P}(\Voc).$
For each $m$, learning $\pim$ in this form reduces to learning a map from $\F^{(m-1)}_e$, sequences in $\R^{\de}$ of length $(m-1)$, into the probability simplex $\mathcal{P}(\Voc).$ 

\begin{example} \label{ex:32}
\rev{Learning a map from $\R^d$ to a probability simplex on a finite set may be
seen as underpinning the central machine learning task of \emph{classification}.
This probabilistic approach to classification was first introduced in~\cite{bridle1990probabilistic}. Given an input in $\R^d$, the objective is to learn a function that outputs a probability distribution over a finite set of classes, assigning higher probability to the more likely class. Such an approach is central
to the MNIST digit labeling problem addressed in the seminal paper~\cite{lecun2002gradient}.} 
In that context $d=28^2$ and the cardinality of the set underlying the probability simplex is $10.$
\end{example}

In the context of language models we are learning ``classifiers'' from Euclidean spaces of dimension $(m-1)\de$, for $m=1, \cdots, L$
into a set of cardinality $|\Voc|$. \rev{Despite being a familiar problem, analogous to classification, it is rendered complex here for two reasons: (i) the choice of parameterization couples together the $\qem$ for all $m=1, \cdots, L$ by a model using the \emph{same} parameter, making it is necessary to consider multiple 
sequences of all lengths $m=1, \cdots, L$ simultaneously during parameter learning; in the next section we discuss in detail how this may be addressed; (ii) the number of labels $|\Voc|$ is often large; for example, in the case of byte-pair encoding of internet corpora, we recall from Example \ref{ex:22} that $|\Voc|=\mathcal{O}(10^4).$}

 \begin{remark} \label{rem:concept}
 \rev{As well as the computational advantages of learning a single parameter $\theta$, independent of conditioning sequence length, rather than a collection $\{\theta_m\}_{m=1}^L$ as in \eqref{eq:auto}, there is a conceptual reason for doing so: similar ideas or statements may be expressed in sentences of different length; enabling interpolation across sequences of different length is thus clearly motivated.}
 \end{remark}

\section{Training Of The LLM} \label{sec:train}

\rev{In this section we present the learning problem for choosing the optimal choice of parameter $\theta = (\phi, \vartheta) \in \Theta \subseteq \R^p$, which defines the 
approximate next-token probability model for input sequences of arbitrary length, as described in Subsection~\ref{ssec:approxf}. We denote the optimal choice of $\theta$ by $\theta^\star.$}  \rev{The learning of the model is undertaken on the basis of the following data assumption. The data consists of sequences of length $L$, which we denote as $\mathsf{f}_{1:L}$ to distinguish them from sequences $f_{1:L}$ that are generic inputs to, or samples generated from, the model.}

\begin{dataassumption} \label{das:sequences} We collect samples of i.i.d.\thinspace sequences $\{\fdata_{1:L}^{(n)}\}_{n=1}^N \sim \poL$.
\end{dataassumption}

The objective that is minimized to determine $\theta$ is formulated as 
minimizing a divergence between the data distribution, written in
the form \eqref{eq:auto}, and the proposed $\theta-$dependent 
model distribution, determined by \eqref{eq:auto3_shared}, \eqref{eq:qarc}. We introduce the population-level objective
function and learning problem in Subsection~\ref{ssec:poploss}; and then an empirical objective function, which is implemented in
practice, in Subsection~\ref{ssec:emploss}. 
Using the autoregressive structures \eqref{eq:auto} and \eqref{eq:auto3_shared} we derive a decompositions of the original objective. The original objective is defined over the joint
distribution defined by the entire sequence of length $L$; this is decomposed into a sum of coupled objectives each of which matches the conditional
distribution on the next-token, given the preceding sequence, for all sequences
of lengths up to $L-1.$ This decomposition reveals that the
proposed objective strives to ensure that the model is accurate when conditioning on sequences of arbitrary length.

Before defining population and empirical optimization problems, we recall the Kullback-Leibler (KL) divergence between probability distributions $\mathsf{p}, \mathsf{p}'$ defined over the same space and with common support:
\begin{equation} \label{eq:KLdivergence}
\dkl(\mathsf{p}||\mathsf{p}') = \mathbb{E}_{f \sim \mathsf{p}} \left[\log\frac{\mathsf{p}(f)}{\mathsf{p}'(f)}\right].
\end{equation}
Use of this divergence as the basis for definition of objective functions
leads to the decompositional structure alluded to
in the previous paragraph. In the following sections we demonstrate this fact, 
and highlight a practical optimization problem for $\theta$ that may be implemented as a result. Before this, we make two remarks about the data employed in training. \rev{Thus we define the general optimization problem, in population form in
Subsection~\ref{ssec:poploss} and in empirical form in Subsection~\ref{ssec:emploss}.
We then describe a procedure to fine-tune models to task-specific data in Subsection~\ref{ssec:emplossFT}; and a procedure to increase the model's alignment to human preferences using reinforcement learning in Subsection~\ref{subsec:RL}. We will use $N_{\mathsf{fine}}$ and $N_{\mathsf{RL}}$ to denote the number of data samples available for fine-tuning or reinforcement learning, distinct from $N$, the number
of samples used in the general optimization problem based on Data Assumption~\ref{das:sequences}. We will conclude this section with some details on performing optimization in practice in Subsection~\ref{sec:minimizing_empirical}.}

\begin{remark}[Remark \ref{rem:31} continued] \label{rem:training_data}
We have defined probabilities $\poL$ with respect to all text available on the Earth. But in practice an approximation is learned through accessing a (vast) subset of such data from the internet. So, although $N$ is large, it is not the case that all data defining $\poL$ is used. Many
proprietary models do not report details of training data used, but some do. For example, the LLaMa-3 models were trained on a corpus of about $15$ trillion multilingual tokens~\cite{grattafiori2024llama}, increasing from the $2$ trillion tokens used to train LLaMA-2 models~\cite{touvron2023llama}. The data used arises from several sources. For example, early GPT-3 models used a combination of data sources such as Common Crawl, books and the English Wikipedia pages~\cite{brown2020language}. With the increasing number of tokens used, and with current
understanding about how much data is required, it is estimated that model training requirements will require the totality of all available human-generated text data on the web, which contains a median of $3100$ trillion tokens, between 2026 and 2032~\cite{villalobos2022will}. To address potential data scarcity,
we note that more data is available in the form of images and videos.
Furthermore recent models have also started to incorporate synthetic data for training~\cite{abdin2024phi}; the effect of synthetic data on model 
collapse is studied in~\cite{gerstgrasser2024model}.
\end{remark}

\begin{remark}[Data Processing] \label{rem:data_processing} The training data employed for several commercial models are determined by curating and filtering web data. This process includes removing harmful material, copyrighted content, duplicates and outliers. Furthermore,
training algorithms are often run several times, with different balances of the constituent types of data (text, code, mathematical calculations and so forth) in order to improve the model's performance on specific benchmarks. This can be helpful for tasks  such as mathematical reasoning~\cite{brown2020language}. The processing increases the quality and diversity of samples, thereby reducing correlations between sequences. However, the processing results in shifts to the data distribution which are a function of the performance on the entire dataset. In practice, this violates having i.i.d. samples in Data Assumption~\ref{das:sequences}.
\end{remark}

\subsection{Population-Level Objective} \label{ssec:poploss}
The learning problem for the joint model over sequences of length $L$
is now presented. Recall the true data distribution $\poL$, from 
\eqref{eq:auto}, our approximate distribution $\pi_{1:L}$, from 
\eqref{eq:auto3}, and the KL divergence \eqref{eq:KLdivergence}.
\rev{In particular, we note that in this subsection we work in the general
form of approximate distribution given by \eqref{eq:auto3}, specifying
to the specific case given by \eqref{eq:auto3_shared} in the following subsection.}
Using these constituents we define the following optimization problem:
\begin{definitiono} \label{op:KLjoint}
\begin{align*}
&\J(\theta) \coloneqq \dkl\bigl( \poL||\pi_{1:L}(\cdot\,;\theta)\bigr),\\
&\theta^\star \in \argmin_\theta \; \J(\theta).
\end{align*}
\end{definitiono}
\noindent \rev{We also write $\thetas = (\thetas_1,\dots,\thetas_L).$
The optimization} problem assumes access to the data distribution $\poL$, leading to
the population-level objective. 
The optimization problem is well-defined if $\poL$ has density with respect to
$\pi_{1:L}$, with both viewed as measures supported on $\F^L.$
Solving this optimization  problem is thus targeted at
\rev{choosing the parameter $\theta$ so} that the approximate next-token prediction model \eqref{eq:auto3} gives rise to sequences with similar distribution to those seen in the data distribution \eqref{eq:auto}. The choice of KL divergence has two benefits: (i) as we show in this subsection, it naturally exploits the autoregressive structure that is integral to next-token prediction; 
(ii) as we show in the next subsection, the ordering of the two probability distributions within the KL divergence leads to an objective function which can be evaluated using only samples from $\poL$ as detailed in Data Assumption \ref{das:sequences}.

We thus start by showing how to exploit structure in the objective
function, using autoregressive factorization. \rev{From \eqref{eq:auto}
and \eqref{eq:auto3} we obtain, for $\theta = (\theta_1,\dots,\theta_L)$,}
\begin{subequations}
\allowdisplaybreaks
\label{eq:lossd}
\begin{align}
\J(\theta) &=  \mathbb{E}_{\foL \sim \poL}\Bigg[\log \left(\frac{\poL(\foL)}{\pioL(\foL;\theta)}\right)\Bigg] \label{eq:onea}\\
&=\sum_{\ell=1}^L \mathbb{E}_{\foL \sim \poL}\Bigg[\log \left(\frac{\pl(\fil|\folm)}{\pil(\fil|\folm;\theta_\ell)}\right)\Bigg] \label{eq:obj_sum} \\
&= \sum_{\ell=1}^L \mathbb{E}_{\fol \sim \pol}\Bigg[\log \left(\frac{\pl(\fil|\folm)}{\pil(\fil|\folm;\theta_\ell)}\right)\Bigg] \label{eq:obj_sum_marginal} \\
&= \mathsf{const}-\sum_{\ell=1}^L \mathbb{E}_{\fol \sim \pol} \log\pil(\fil|\folm;\theta_\ell), \label{eq:oned}
\end{align}
\end{subequations}
where $\mathsf{const}$ indicates a scalar that does not depend on $\theta.$ We note that ~\eqref{eq:obj_sum} follows from imposing the autoregressive structure
on \eqref{eq:onea}. Then~\eqref{eq:obj_sum_marginal} follows because the input to the logarithm depends only on sequences of length $\ell.$
Finally we obtain~\eqref{eq:oned} since $\pl(\fil|\folm)$ does not depend on $\theta.$ Given~\eqref{eq:oned}, and the proposed model form in~\eqref{eq:qarc},
it is thus natural to define the objective:
\begin{equation} \label{eq:abi}
 \Loss(\theta) = -\sum_{\ell=1}^L \mathbb{E}_{f_{1:\ell} \sim \mathsf{p}_{1:\ell}} \log\pil(\fil|\folm;\theta_\ell).
\end{equation}
Since $\J(\theta) = \Loss(\theta) + \mathsf{const}$ it follows that the minimizers of $\J(\cdot)$ and $\Loss(\cdot)$ coincide.

\begin{remark} The objective in Optimization Problem~\ref{op:KLjoint}, defined through the KL divergence between sequences of length $L$,
can also be defined in terms of any other choice of divergence $\mathsf{D}$. However, in order for these divergences to be the basis of practical algorithms, it is necessary to be able to minimize this divergence without knowledge of the normalization constant of the target measure we are approximating, and given only samples from $\pol$ for all $\ell \in \Se^L.$

Furthermore we note that use of $\mathsf{D}=\dkl$ yields, using positivity of KL divergence and (\ref{eq:nevesb}), the inequality
$$\mathsf{D}(\poL|\pioL) \geq \mathbb{E}_{f_{1:m-1} \sim \mathsf{p}_{1:m-1}(\cdot)}\mathsf{D}\bigl(\mathsf{p}_m(\cdot|f_{1:m-1})||\pim(\cdot|f_{1:m-1})\bigr).$$
Thus making the left hand side small through minimization also ensures that the right-hand side is small, for any $m \in \Se^L.$ 
This provides a quantitative estimate of the error in next-token prediction for any such $m.$ On the assumption that 
next-token prediction is
used to generate text, use of divergences other than KL
will ideally share this, or similar, desirable property enabling the bounding of next-token probability prediction in terms of the objective
on sequences of length $L.$
\end{remark}

\subsection{Empirical Objective} \label{ssec:emploss}

\rev{In the calculations within the preceding subsection we have worked with a general parameterization of the approximate autoregressive model in the form~\eqref{eq:auto3}. However, as we proceed to discuss empiricalization, and hence what is done in practice,
we will now work with an approximate autoregressive model in the 
form~\eqref{eq:auto3_shared}. Thus, we set $\theta_\ell=\theta,$ independently of $\ell.$ Furthermore we write $\theta=(\varthetaa,\varthetae),$
where $\phi$, mapping from token space $\Voc$ to Euclidean embedding space
$\R^{\de}$, is defined by the matrix $\varthetae.$ We then note that \eqref{eq:abi} yields
\begin{subequations}
\begin{align} 
\Loss(\theta) &= -\sum_{\ell=1}^L \mathbb{E}_{f_{1:\ell} \sim \mathsf{p}_{1:\ell}} \log\pil(\fil|\folm;\theta), \label{eq:lwca}\\
&= -\sum_{\ell=1}^L \mathbb{E}_{f_{1:\ell} \sim \mathsf{p}_{1:\ell}} \log\ql\bigl(\fil|\phi \circ \folm;\varthetaa\bigr). \label{eq:lwcb}
\end{align}
\end{subequations}
Note that form~\eqref{eq:lwcb} makes explicit that, through application
of the mapping $\phi,$ the model takes as input sequences in $\F^\ell_e$, allowing for the generalization properties of machine learning algorithms in Euclidean space to be exploited.} 

\begin{remark}
We note from \eqref{eq:obj_sum_marginal}, and applying the constant $\theta$ structure
given in \eqref{eq:lwca}, that
\begin{subequations}
\label{eq:neves}
\begin{align}    
\J(\theta) &= \sum_{\ell=1}^L \mathbb{E}_{\folm \sim \polm}\mathbb{E}_{\fil \sim \pl(\cdot|\folm)}\Bigg[\log \left(\frac{\pl(\fil|\folm)}{\pil(\fil|\folm;\theta)}\right)\Bigg] \\
&= \sum_{\ell=1}^L \mathbb{E}_{\folm \sim \polm} \dkl\bigl(\pl(\cdot|\folm)||\pil(\cdot|\folm;\theta)\bigr). \label{eq:nevesb}
\end{align}
\end{subequations}
This shows that minimizing the objective function aims at choosing parameter $\theta$ so that next-token prediction accurately mimics next-token probabilities for conditioning sequences of all lengths up to $L-1$, on average; the fact that the same parameter $\theta$
appears for all sequence input lengths couples the different sequence lengths. Furthermore, since all terms in this sum are positive, it is
apparent that had each conditional $\pil$ been parameterized independently by $\theta_\ell$, for different $\ell$, then it would be possible to optimize each conditional alone. \rev{Moreover, if the model class $\pil(\fil|\folm;\cdot)$ is
sufficiently rich so as to include the true conditional probability, then
the minimizer of each term in the objective would be $\pil(\cdot|\folm;\thetas_\ell) = \pl(\cdot|\folm)$, thereby recovering the true autoregressive factorization.}
\end{remark}

We have thus identified $\Loss(\theta)$ in \eqref{eq:lwcb} 
as the function we wish to minimize. We have available to us data as prescribed in Data Assumption~\ref{das:sequences}. Using this we can form the empirical measure that approximates $\pol$, for any $\ell \in \Se^L$:
$$\rev{\pol^N \coloneqq \frac{1}{N} \sum_{n=1}^N \delta_{\fdata_{1:\ell}^{(n)}}.}$$
Using this approximation in the objective function $\Loss(\cdot)$ from~\eqref{eq:lwcb} we obtain the objective function minimized in practice:
\begin{equation}
\label{eq:lossN}    
\rev{\Loss^N(\theta) = -\sum_{n=1}^N \sum_{\ell=1}^L \log\ql\bigl(\fdata_\ell^{(n)}|\phi \circ \fdata_{1:\ell-1}^{(n)};\varthetaa\bigr).}
\end{equation}

\begin{remark} \label{rem:4p7}
An unbiased estimator for the objective function $\Loss$ in~\eqref{eq:lwcb} is given by drawing $N$ i.i.d.\thinspace sequences of length $\ell$, and doing so independently for each $\ell = 1,\dots,L$. However, correlated samples can also be
used: the empirical objective in~\eqref{eq:lossN} is defined by first drawing sequences of maximum length $L$ and then using partial subsets of the same sequence to estimate the expectation for each $\ell$.
\end{remark}

\rev{Minimizing objective~\eqref{eq:lossN} is typically performed by computing derivatives of the objective function with respect to the model parameters. The derivatives with respect to $\theta = (\phi, \vartheta)$ are derived in Appendix~\ref{sec:deriv}. More details on the  optimization algorithms that leverage estimators for these gradients to find optimal models are provided in Section~\ref{sec:minimizing_empirical}.}

Learning the model parameters by minimizing~\eqref{eq:lossN} given text sequences from various domains is also known as~\emph{supervised pretraining}. The resulting next-token prediction
model found this way is often referred to as a \emph{pretrained model} or a \emph{foundation model}. These models can be adapted for specific  domains using fine-tuning or reinforcement learning, which we describe next.

\subsection{Fine-Tuning The Empirical Objective} \label{ssec:emplossFT}

In the previous subsection, we describe the optimization of parameter $\theta=(\varthetaa,\varthetae),$ %
for the LLM through minimization of an empirical objective function~\eqref{eq:lossN} that is based on data reflecting the variety of tasks for which the LLM might be used. In the context of this subsection, we refer to this as \emph{pretraining}. \emph{Fine-tuning} with task-specific data is used to improve the performance of a pretrained model on a specialized class of tasks. The fine-tuned model is often found from an optimization problem identical to that defined by \eqref{eq:lossN}, but using a different
(and typically much smaller) data set. The following assumption provides a setup for the fine-tuning data, \rev{consisting of $N_{\mathsf{fine}}$ sequences}, which may be drawn from a distribution $\poL'$ that is different from the distribution $\poL$ used in pretraining. The distribution $\poL'$ will typically represent data for a specific task, but will ideally be close to $\poL$ after conditioning on a set related to the specific task defining the fine-tuning data.

\begin{dataassumption} \label{das:sequences_finetuning} We collect samples of i.i.d.\thinspace \rev{sequences $\{\fdata_{1:L}^{(n)}\}_{n=1}^{N_{\mathsf{fine}}} \sim \poL'$.}
\end{dataassumption}

\rev{Note that, to avoid proliferation of symbols, we have used $\fdata$, as in
Data Assumption~\ref{das:sequences}, to denote the training data.}
Using this fine-tuning training data, 
the objective function minimized during the fine-tuning  stage is given by:
\begin{equation}
\label{eq:lossN22}    
\rev{\Loss^{N_{\mathsf{fine}}}(\theta) = -\sum_{n=1}^{N_{\mathsf{fine}}} \sum_{\ell=1}^L \log\ql\bigl(\fdata_\ell^{(n)}|\phi \circ \fdata_{1:\ell-1}^{(n)};\varthetaa\bigr).}
\end{equation}
Algorithms for this fine-tuning optimization problem are typically initialized at the model parameters found from pretraining, thus connecting the information
in the smaller, fine-tuning, data set to the larger data set used for pretraining. 
The resulting (approximate) minimizer is denoted by $\theta^{\mathsf{fine}}.$

\begin{remark} 
The updates of the parameters during fine-tuning (e.g., neural network weights) are commonly restricted to be low-rank matrices using the widely-adopted Low-rank Adaptation (LoRA) procedure~\cite{hu2022lora}. This restriction of the parameter changes improves computational efficiency by reducing the number of trainable parameters and results in an implicit form of regularization; see Section~\ref{sec:minimizing_empirical} for discussion of other regularization approaches employed during LLM training.
\end{remark}

\subsection{Reinforcement Learning} \label{subsec:RL}

It is frequently the case that LLMs are modified to increase alignment of model responses with human preferences.
This may be performed instead of, or in addition to, fine-tuning. Alignment to human preferences can be achieved using \emph{reinforcement learning}. In this approach the model parameters are updated to maximize a reward function $\re \colon \F^L \rightarrow \R$ that approximates human preferences. In practice, the reward function is learned~\cite{christiano2017deep, ouyang2022training, ziegler2019fine} or is defined implicitly by preference data~\cite{rafailov2023direct}. A common framework is known as Reinforcement Learning from Human Feedback (RLHF)~\cite{stiennon2020learning, azar2024general}; this approach minimizes the negative reward while staying close to a reference model obtained from a pretraining or fine-tuning procedure as in subsections~\ref{ssec:emploss} and~\ref{ssec:emplossFT}. %
The RLHF objective is given by:
\begin{equation}
\label{eq:lossRL}
\Loss^{\mathsf{RL}}(\theta) = -\mathbb{E}_{\foL \sim \pioL(\cdot;\theta)}[\re(\foL)] + \beta \dkl\bigl(\pioL(\cdot\,;\theta)||\pioL(\cdot\,;\theta^{\mathsf{ref}}));
\end{equation}
here $\beta > 0$ is a tuning parameter and $\theta^{\mathsf{ref}}$ refers to the minimizer of~\eqref{eq:lossN} or~\eqref{eq:lossN22}. This objective function is typically minimized starting from the reference model parameters $\theta^{\mathsf{ref}}$. In practice, the minimization problem is solved numerically by approximating the expectation over $\pioL(\cdot;\theta)$, which is required in the definition of both terms, via sampling. Given a set of $N_{\mathsf{RL}}$ samples $\foL^{(n)} \sim \pioL(\cdot;\theta)$, which depend on $\theta$ and thus are generated at each training iteration, the empirical objective function is given by:
\begin{equation}
\label{eq:lossRLN}
\Loss^{\mathsf{RL},N}(\theta) = \frac{1}{N_{\mathsf{RL}}} \sum_{n=1}^{N_{\mathsf{RL}}} -\re(\foL^{(n)}) + \beta \log\left(\frac{\pioL(\foL^{(n)}\,;\theta)}{\pioL(\foL^{(n)}\,;\theta^{\mathsf{ref}})}\right).
\end{equation}
\begin{remark} Pretraining and fine-tuning, defined in subsections~\ref{ssec:emploss} and~\ref{ssec:emplossFT} respectively, require samples from the data distributions as in Data Assumptions~\ref{das:sequences} and~\ref{das:sequences_finetuning}. Solving these optimization problems does not require generating sequences from the LLM during training. On the other hand, evaluating objective~\eqref{eq:lossRL} in reinforcement learning does require sampling such sequences. This facilitates the calculation of the
expected reward and the regularizing divergence. Given that the densities for both approximate distributions are analytically available, we can employ various divergences here, not just the reverse KL divergence used above.
\end{remark}

\subsection{Minimizing The Objective Functions} \label{sec:minimizing_empirical}

Approximate minimization of all of the objective functions~\eqref{eq:lossN}, \eqref{eq:lossN22} and \eqref{eq:lossRLN}
is typically achieved using a stochastic gradient descent (SGD) procedure~\cite{robbins1951stochastic}. 
Currently deployed versions of the methodology, such as the widely used Adam optimization method~\cite{adam2014method}, regularize to avoid overfitting to the data and to improve generalization: that is, to promote desirable model response in the generation of new sequences not present in the empirical dataset defined in Data Assumption~\ref{das:sequences}; or,
in the fine-tuning set-up, new sequences not present in the empirical dataset defined in Data Assumption~\ref{das:sequences_finetuning}. Common regularization methods applied during SGD include dropout~\cite{zhou2020scheduled, zehui2019dropattention} and early stopping~\cite{kaplan2020scaling}. Sometimes $L^2$ regularization (often termed weight decay in the machine learning community), or other Tikhonov regularization~\cite{engl1996regularization} techniques, is applied directly to the objective function, rather than implicitly regularizing through the iterative optimization method~\cite{d2024we, kobayashi2024weight}. 

\section{Architecture} \label{sec:arch}

In this section we detail the architecture enabling approximation of $\poL$ from \eqref{eq:auto} by $\pioL$ defined by~\eqref{eq:auto3_shared}, \eqref{eq:qarc}. We recall the form of the approximation here:
\begin{subequations}
\label{eq:sumz}
\begin{align}
\pioL(\foL;\theta)&=\Pi_{\ell=1}^L \pil(f_\ell|f_{1:\ell-1};\theta),\\
\pim(f_m|f_{1:m-1};\theta)&=\qem\bigl(f_m|\phi \circ f_{1:m-1};\varthetaa\bigr).
\end{align}
\end{subequations}
The parameter $\theta=(\varthetaa,\varthetae)$ is to be chosen to optimize the approximation. Through Definitions~\ref{def:thetam}, \ref{def:phis} we have detailed the dependence on matrix $\varthetae$ which defines the mapping $\phi$; recall that this mapping enables $\qem$ to take sequences in \emph{Euclidean space} as input, rather than the discrete \emph{token space}, hence facilitating the possibility of \emph{interpolation}. In particular, this interpolation allows the model to make next-token predictions for sequence that are not in the training data. In this section, we detail how the parameter $\varthetaa$ enters (\ref{eq:sumz}b); then the value of the parameter $\theta$ is chosen to make the interpolation effective, via the training process described in Section \ref{sec:train}.

To parameterize the model $\qem$
we will introduce the \emph{set transformer}, which is based on the \emph{attention mechanism}. Subsection \ref{ssec:mopm2} is a self-contained description of the attention mechanism; we introduce \emph{attention maps}, families of parameterized maps acting on the space of Euclidean vector-valued sequences into itself, for sequences of any length. These attention maps are combined with \emph{normalization maps} and \emph{neural networks} to define the \emph{attention block}. In Subsection \ref{ssec:sett2} attention blocks are composed with one another, and with a final map into Euclidean space, to define the set transformer. The resulting map, parameterized by $\varthetaa$, takes as input Euclidean vector-valued sequences of arbitrary length and outputs elements in the probability simplex on $\Voc.$ Subsection \ref{ssec:temp} describes refinements of, or alternatives to, the attention mechanism; the purpose is to incorporate more complex temporal structure, such as causality, and relatedly to reduce the complexity of the evaluation of the attention map.

\subsection{The Attention Mechanism} \label{ssec:mopm2}

We start by defining the \emph{attention mechanism} $\ATT:\F^{(m-1)}_e \to \F^{(m-1)}_e$. The $k^{th}$ entry of the sequence
$\ATT(h_{1:m-1})$ is defined by
\begin{subequations} \label{eq:transport_map2}
\begin{align}
\ATT(h_{1:m-1})_k & := h_k+\frac{1}{Z_k} \sum_{l=1}^{m-1} \exp\bigl(\langle Q h_k, K h_l \rangle_{\R^c}\bigr)\, V h_l, \label{eq:seqseqMap}\\
Z_k & := \sum_{j=1}^{m-1} \exp\bigl(\langle Q h_k, K h_j \rangle_{\R^c}\bigr),
\end{align}
\end{subequations}
where $Q,K \in \R^{c \times \de}$ and $V \in \R^{\de \times \de}$. 
Note that the output sequence at \emph{any} location $k$ depends on the input sequence at \emph{every} location $l=1, \cdots, m-1.$

\begin{remark} \label{rem:atp} \rev{It can be helpful to view the attention map
through the lens of probability. To this end, we first define the conditional probability of index $l$, given index $k$, parameterized by the sequence $h$: 
$$p(l|k;h) = \frac{1}{Z_k} \exp \bigl(\langle Qh_k, K h_l \rangle_{\R^c}\bigr).$$
Then,
$$\ATT(h_{1:m-1})_k = h_k+V\,\mathbb{E}_{l \sim p(\cdot|k;h)}\bigl[h_l\bigr].$$
Notice that the expectation means that $\ATT(h_{1:m-1})_k$ depends on the
entire input sequence $h_{1:m-1}$, making the transformation $\ATT$ highly
nonlocal, as well as nonlinear. The matrices $Q, K$ and $V$, often referred to in the applied literature as \emph{query, key} and \emph{value} matrices~\cite{vaswani2017attention}, can be learned so as to emphasize valuable correlations between different entries in
an input sequence, using this non-locality. The continuum picture of this probabilistic interpretation is presented in Appendix~\ref{sec:arch_probabilistic}.}
\end{remark}

We introduce the map $\NOF: \R^{\de} \to \R^{\de}$,   representing a normalization procedure, and extend it to map
$\NOF:\F^{(m-1)}_e \to \F^{(m-1)}_e$ by allowing it to act pointwise on elements in the input sequence.
Similarly, we introduce map $\NNF: \R^{\de} \to \R^{\de}$, representing a neural network, and extend it to map
$\NNF:\F^{(m-1)}_e \to \F^{(m-1)}_e$, again by allowing it to act pointwise on elements in the input sequence.
Examples \ref{ex:ln} and \ref{ex:nn} describe normalizations and neural networks that are used in practice.

The \emph{attention block} $\Ac:\F^{(m-1)}_e \to \F^{(m-1)}_e$ is then defined by the following concatenation of operations:
\begin{subequations} \label{eq:ow2}
\begin{align}
h_{1:m-1} & \leftarrow \ATT(h_{1:m-1})\\
h_{1:m-1} & \leftarrow \NOF\bigl(h_{1:m-1}\bigr),\\
h_{1:m-1} & \leftarrow h_{1:m-1}+\NNF\bigl(h_{1:m-1}\bigr), \label{eq:owc2} \\
h_{1:m-1} & \leftarrow \NOF\bigl(h_{1:m-1}\bigr).
\end{align}
\end{subequations}
\rev{Here each assignment, denoted by $\leftarrow$, represents an in-place operation that updates the sequence and whose output becomes the input to the operation that follows. Thus, the attention block comprises composition of the attention map,
with a normalization, then with a residual form neural network, and finally 
composition with another normalization.}

\begin{remark}
    \label{rem:qkv}
    \rev{Note that the matrices defining the attention map, and hence the attention block, are independent of $m.$} The maps defining the normalization and neural networks are defined pointwise for each element of the sequences to which they are applied, and hence the parameters defining these maps are also independent of $m.$
    The independence of these parameters with respect to $m$ is central to having a single parameter $\varthetaa$ in the definition of the architecture, and in particular in $\qem$ as defined in \eqref{eq:qarc}. The complete specification of $\varthetaa$ is 
    given in Remark \ref{rem:param}.
\end{remark}

\begin{example}[Layer Normalization] \label{ex:ln} Define a diagonal matrix $D \in \R^{\de \times \de}$ and vector $b \in \R^{\de}$, which are learnable parameters. Layer normalization $\NOF \colon \R^{\de} \rightarrow \R^{\de}$ applies the rescaling:
\begin{equation} \label{eq:mean_variance_rescaling2}
\rev{\NOF(u) = D\frac{(u - m(u))}{\sqrt{\sigma^2(u) + \varepsilon}} + b, \quad
m(u) = \frac{1}{\de} \sum_{i=1}^{\de} u_i, \quad
\sigma^2(u) = \frac{1}{\de} \sum_{i=1}^{\de} (u_i - m(u))^2.}
\end{equation}
Here $\varepsilon > 0$ is a small parameter
included for numerical stability. We define $\theta_{\mathfrak{no}}=(D,b)$, 
the learnable parameters of the normalization map $\NOF.$
Usually $\varepsilon$ is user prescribed, not learned.

\rev{A specific form of normalization makes the choice $m(u) = 0$ and $b = 0$ and applies the rescaling:
\begin{equation} \label{eq:RMSnorm3}
\NOF(u) = D\frac{u}{\sqrt{\sigma^2(u) + \epsilon}}, \qquad \sigma^2(u) = \frac{1}{\de} \sum_{i=1}^{\de} u_i^2.
\end{equation}
We define $\theta_{\mathfrak{no}} = (D)$ as the learnable parameter of the map. This is known as \emph{root mean square normalization} (RMSNorm).
}

\rev{One further simplification} does not involve mean shift,  diagonal rescaling, shift and stabilization; if we make the choices $\varepsilon =0$, $b = 0$, $m(u) = 0$ and $D = I$, 
then we obtain the map given by
\begin{equation} \label{eq:RMSnorm2}
\NOF(u) = \frac{u}{|u|_2}.
\end{equation}
This operation constrains the output to the unit sphere: $\NOF \colon \R^{\de} \rightarrow \mathbb{S}^{\de-1}$. 
In this case, there are no learnable parameters in the map $\NOF$.
\end{example}

\begin{remark} \label{rem:ln}
\rev{In contrast to $\ATT$, map $\NOF$ is local, it acts in the same way
at every point on the sequence; see \cite{ba2016layer}. An important role of normalization is to preserve the scale of hidden representations in a deep network. This stabilizes evaluations of the objective and its gradients, and empirically allows for larger learning rates and  faster convergence in training.}
\end{remark}

\begin{example}[Neural Network] \label{ex:nn} 
Often, $\NNF$ is implemented as a two layer feed-forward neural network of the form
$$\NNF(x)=W_2\, \mathsf{act}\bigl(W_1 x+b_1\bigr)+b_2,$$
where $W_1 \in \R^{c' \times \de}$,
$W_2 \in \R^{\de \times c'}$, $b_1 \in \R^{c'}$ and $b_2 \in \R^{\de}$
and where $\mathsf{act}: \R \to \R$ is an activation function, such as ReLU, 
extended to act pointwise on $\R^{c'}$;  integer $c'$ is a user-prescribed channel width, typically different from channel width $c$ in the definitions of matrices $Q,K.$
Then, $\theta_{\mathfrak{nn}}=(W_1, W_2, b_1, b_2)$ defines the learnable parameters entering $\NNF.$
\end{example}

\begin{remark} \label{rem:nn}
\rev{Like $\NOF$, map $\NNF$ is local and acts in the same way on every point
in the input sequence; and it does so by applying a standard deep neural
network~\cite{goodfellow2016deep}. The combination of $\NNF$, which is local since it processes each point independently, and $\ATT$, which as we discuss
in Remark \ref{rem:atp} acts nonlocally on entire sequences, is
shown to benefit overall expressivity~\cite{vaswani2017attention}.}
\end{remark}

\subsection{Set Transformer} \label{ssec:sett2}

In this subsection we explain how the attention block from Subsection \ref{ssec:mopm2} can be used to
define the next-token probability model $\qem\bigl(\cdot\,|\cdot\,; \varthetaa\bigr): \F^{(m-1)}_e \to \mathcal{P}(\Voc).$
The key component is the \emph{set transformer} $\mathsf{T}: \F^{(m-1)}_e \to \R^{\de}$ defined by
\begin{equation}
\label{eq:settdef}    
\mathsf{T}(h_{1:m-1}) = \mathsf{P} \circ \Ac_K \circ \dots \Ac_1(h_{1:m-1});
\end{equation}
here %
$\Ac_k \colon \F^{(m-1)}_e  \rightarrow \F^{(m-1)}_e$ are attention blocks, as described in Subsection~\ref{ssec:mopm2}, with different parameters for each $k \in \Se^K$,
and the \emph{pooling} operator $\mathsf{P} \colon \F^{(m-1)}_e \rightarrow \R^{\de}$ is now defined.

The pooling operator $\mathsf{P} \colon \F^{(m-1)}_e  \rightarrow \R^{\de}$ is a function acting on a sequence and returning a vector. In the context of the set transformer, $\mathsf{P}$ is typically chosen to be both permutation-invariant and to be defined for input sequences of arbitrary length. The following example contains two widely used choices that satisfy these properties. 

\begin{example} \label{ex:pooling} The mean pooling operation is given by taking the average of the sequence elements: %
$$\mathsf{P}\bigl(h_{1:m-1}\bigr) = \frac{1}{m}\sum_{k=1}^m h_k.$$
The max pooling operation is given by choosing the  maximum along each coordinate of the sequence: %
$$\mathsf{P}\bigl(h_{1:m-1}\bigr)_j = \max((h_1)_j,\dots,(h_m)_j), \qquad j = 1,\dots,\de.$$
\end{example}

We now introduce the softmax function, using it here to convert
elements of $\R^{\Voc}$ into points on the probability simplex $\prob(\Voc).$
Here $\prob(\Se^d)$ denotes the set of probability measures on sets of cardinality $d$, labeled without loss of generality by the integers $\Se^d.$

\begin{definition} \label{def:softmax} For any integer $d$, let $\sigma \colon \R^d \rightarrow \prob(\Se^d)$ be the \emph{softmax} function defined as 
$$\sigma(z)_k = \frac{e^{z_k}}{\sum_{j=1}^d e^{z_j}}.$$
\end{definition}

Given the set transformer $\mathsf{T}: \F^{(m-1)}_e \to \R^{\de}$, together with linear operator $\Bl \colon \R^{\de} \rightarrow \R^{|\Voc|}$, we define $\qem(\cdot|h_{1:m-1};\varthetaa) \in \mathcal{P}(\Voc)$ by
\begin{equation} \label{eq:next_token_model}
\qem(\cdot|h_{1:m-1};\varthetaa) = \sigma\Bigl(\Bl\circ\mathsf{T}\bigl(h_{1:m-1}\bigr)/\tau\Bigr).
\end{equation}
Here, $\tau > 0$ is a \emph{temperature} parameter. Choosing $\tau$ small leads to greater concentration of the output probability distribution on the most probable tokens under the learned model, whilst larger $\tau$ leads to greater variability in the output; this is discussed in Remark~\ref{rem:temp}.

\begin{remark}
Note that it is possible to dispense with $\Bl$ if $\mathsf{T}$ were designed to map to $\R^{|\Voc|}$ rather than to $\R^{\de}$. 
However the object $\mathsf{T}: \F^{(m-1)}_e \to \R^{\de}$ as defined is useful for other tasks, different from next-token prediction,
which is why we have defined $\Bl$ and $\mathsf{T}$ separately. The process of using $\mathsf{T}$ in related tasks is known as \emph{transfer learning}; see~\cite{raffel2020exploring} for an application of pre-trained text transformers to downstream tasks such as semantic analysis and summarization.
\end{remark}

\begin{remark} \label{rem:param}
Taken together, the parameters appearing in all the attention blocks $\{\Ac_k\}_{k=1}^K$, together with the parameters defining linear map $\Bl$, define  the learned parameters $\varthetaa$ appearing in the next-token prediction model. The parameter $\tau > 0$ is typically not learned, but is user-specified; see Remark \ref{rem:temp}.    
\end{remark}

\subsection{Encoding Temporal Structure} \label{ssec:temp}

There are several features of the attention map, all related to temporal structure, which are noteworthy and which suggest
modifications of the basic attention map as defined in the previous subsection. These relate to a permutation property of
the attention map, to the lack of causality in the attention map, and to the undesirable quadratic scaling of the
attention map with respect to sequence length.

\subsubsection{Positional Encoding}

The map $\ATT$ defined in \eqref{eq:transport_map2} is 
permutation equivariant with respect to elements of the 
input sequence: applying a permutation to the sequence before applying
$\ATT$, and then applying the inverse of the permutation afterward, yields the same result as simply applying $\ATT.$  
Consequently, such a map discards any positional information in the sequence. 
The goal of \emph{positional encoding} is to incorporate information on the location of each element in the sequence. 
Arguably the simplest way to do this is to take $h \in \{f: \Se^M \rightarrow \R^{\de}\}$ and map it to $h' \in
\{f: \Se^M \rightarrow \R^{\de+1}\}$ by defining $h'_m=(m,h_m^\top)^\top.$ Generalization of this idea to functions
varying over open subsets of $\R^d$ is widely used in applications of artificial intelligence to problems arising
in science. However, in LLMs different ideas are used and we now outline  how they are defined.

The prevailing methodology in LMMs is to map $h \in \{f: \Se^M \rightarrow \R^{\de}\}$ to $h' \in
\{f: \Se^M \rightarrow \R^{\de}\}$ by setting $h_m' = h_m + r_m$ 
where $r \colon \Se^{M} \rightarrow \R^{\de}$ computes the \emph{absolute positional encoding} as a function of the index in the sequence. The embedding can be learned using independent parameters for each input index, or contain a small number learnable parameters. 

\begin{example}[Trigonometric Functions] \label{ex:tf} Given an even dimension $\de$, the \emph{trigonometric embedding} is defined as follows. Set $\lambda_j: = 2^{2j/\de}$ for $j = 1,\dots,\de/2$ and then
$$r_m = (\sin(m/\lambda_1), \cos(m/\lambda_1), \sin(m/\lambda_2), \cos(m/\lambda_2), \dots, \sin(m/\lambda_{\de/2}), \cos(m/\lambda_{\de/2})).$$
\end{example}

\subsubsection{Modified Inner-Product To Address Positional Encoding}
\label{remark:other_inner_products} 

When evaluating the attention map~\eqref{eq:transport_map2}, on sequence 
$h \in \F^{m-1}$, it is necessary to compute 
\begin{equation} \label{eq:euclidean_inner_product}
\langle Q h_k, K h_l \rangle_{\R^c}.
\end{equation}
The  inner-product \eqref{eq:euclidean_inner_product} can also be adjusted to reflect temporal structure,
and in particular to address the issue of positional encoding. Recall the additive function $\rr \colon \Se^{M} \rightarrow \R^{\de}$ defined in the previous
subsubsection, and here set this additive function to zero: $\rr \equiv 0$. Instead, \emph{relative embedding} defines a weighted inner-product based on the distance between the Euclidean image of tokens in the sequence; in this sense it is a positional encoding. The RoFormer in~\cite{su2024roformer} sets absolute positional encoding sequence $\rr \equiv 0$ and replaces the Euclidean inner-product~\eqref{eq:euclidean_inner_product}
appearing in~\eqref{eq:transport_map2} by 
$$\langle Qh_k, Kh_l\rangle _{\R^{c},\mathsf{roformer}} \coloneqq \text{Re}\left[\sum_{j=1}^c (Qh_k)_j^\top (Kh_l)_j e^{i(k-l)\psi_j}\right];$$
here $\text{Re}(\cdot)$ denotes taking the real part of the argument.
The complex weight accounts for the relative position of the inputs: 
the $\psi_j$ are chosen to decay as $\gamma^{-2j/d}$ for a large constant $\gamma$, resulting in more oscillatory modes for elements that are more separated in the sequence.

\subsubsection{Modified Inner-Product To Reduce $\mathcal{O}(m^2)$ Complexity} \label{ex:sparse_attention2}
Modification of the inner-product, introduced in the preceding subsubsection, can be used for a different purpose: to address the issue of reducing complexity, 
in $m$, of evaluating the attention map~\eqref{eq:transport_map2}. When evaluating 
this map on a sequence $h \in \F^{m-1}$  
it is necessary to compute the inner-product \eqref{eq:euclidean_inner_product}
for all $1 \le k,l \ll m-1$. Thus, $(m-1)^2$ inner-products must be calculated. \emph{Causal attention} defines the attention map in~\eqref{eq:transport_map2} by replacing the inner-product in~\eqref{eq:euclidean_inner_product} with 
$$\mathbbm{1}_{l \leq k}\,\langle Q h_k, K h_l \rangle_{\R^c}.$$
The masking implied by the indicator function ensures that it is only necessary to compute $(m-1)(m-2)/2$, rather than $(m-1)^2$, inner-products.
More generally, given an adjacency matrix $A \in \{0,1\}^{(m-1) \times (m-1)}$, \emph{sparse attention} maps replace \eqref{eq:euclidean_inner_product} with
$$ A_{lk}\, \langle Q h_k, K h_l \rangle_{\R^c}.$$
The adjacency matrix ensures that interactions between sequence elements are only computed when $A_{kl} \neq 0$. For example, given a bandwidth $r > 0$, we can consider a banded adjacency matrix that has nonzero entries only for $A_{kl}$ with 
$|k - l| \leq r.$  For a bandwidth $r \ll m-1$, the complexity of computing the 
attention map will then scale linearly with the sequence length as 
$\mathcal{O}(r(m-1))$; see Section~\ref{ssec:literature} for references to recent computational methodologies with linear dependence.

\subsubsection{Replacing The Attention Map To Reduce $\mathcal{O}(m^2)$ Complexity} \label{sssec:ssm}

The previous subsubsection shows how both causality, and the related wish to obtain linear scaling in $m$, can be achieved (separately) by selecting a subset of the $(m-1)^2$ inner-products required to define the attention map~\eqref{eq:transport_map2}, and
setting all others to zero. An approach to achieving both of these goals simultaneously is through \emph{state-space models}. 
These work by replacing  $\ATT:\F^{(m-1)}_e \to \F^{(m-1)}_e$
in the attention block with a different sequence-to-sequence map $\SSM:\F^{(m-1)}_e \to \F^{(m-1)}_e.$ 
This map is defined by the following iteration, for $k=0, \cdots, m-2:$
\begin{subequations} \label{eq:ssm0}
    \begin{align}
v_{k+1} & =\mf(v_k,h_k;\theta^{\mf}), \quad v_0=\mv,\\
\SSM(h_{1:m-1})_{k+1} & = \mg(v_{k+1},h_{k+1};\theta^{\mg}).
    \end{align}
\end{subequations}
Here each $v_k \in \R^{c''}$ for some preselected channel width $c''$. 
Notice that $v_{0:m-1}$ accumulates information about the sequence
$h_{1:m-1}$, but does so in a causal fashion: $v_{k+1}$ depends only on $h_{1:k}.$ Thus $\SSM(h_{1:m-1})_{k+1}$ depends only on
$h_{1:k+1}.$ The parameters $\mv$, $\theta^{\mf}$ and $\theta^{\mg}$ are learned. The complexity of evaluating $\SSM(h_{1:m-1})$ is
linear in $m.$ In practice, the model is often implemented
using linear $\mf,\mg$, defined through discretizing a linear and continuous time analog of \eqref{eq:ssm0}; 
details may be found in Appendix~\ref{sec:ssmcts}.

\rev{The potential for state-space models was first realized in Mamba~\cite{gu2023mamba}. The approach not only reduces the quadratic complexity of attention in sequence length, but also leads to less demanding memory requirements, 
in comparison with attention-based methods. 
State space models have produced state-of-the-art performance on audio and genomics datasets for sequences of length $\mathcal{O}(10^6)$. \cite{gu2023mamba, lahoti2026mamba}.
Other recurrent structures have been proposed, such as the RWKV approach in \cite{peng2023rwkv}; this too reduces sequence length complexity to linear, and reduces memory demand. However such efficiency gains can come at the cost of expressivity of the resulting architecture~\cite{waleffe2024empirical}. There is current research activity aimed at balancing efficiency and expressivity by combining  transformer and state-space approaches~\cite{blakeman2025nemotron,lahoti2026mamba}.}

\section{Sequence Generation} \label{sec:seqgen}

In Section~\ref{sec:ntp} we described the framework for next-token prediction. Here we detail how it may be used to generate text sequences. Recall that LLMs roll out sequences in $\Voc$ step-by-step, starting with a sequence of a given length, then generating a sequence of length one greater, and increasing the sequence length by one in steps until it reaches length $L$. 
The basic building block for this roll-out is the learned model \eqref{eq:auto33} which we recall here.
Combining \eqref{eq:auto3_shared} and \eqref{eq:qarc}, and deploying the optimal choice  $\theta^\star=\bigl(\varthetaa^{\star},\varthetae^{\star})$ found through (typically only approximately) minimizing $\Loss^N(\theta)$ as defined in \eqref{eq:lossN}, we obtain
\begin{subequations}
 \label{eq:auto33}
 \begin{align}
\pioL\bigl(\foL;\theta^\star\bigr)& =\Pi_{\ell=1}^L \pim\bigl(f_\ell|f_{1:\ell-1};\theta^{\star}\bigr),\\
\pim\bigl(f_\ell|f_{1:\ell-1};\theta^{\star}\bigr) & =\qem\bigl(f_\ell|\phi^\star \circ f_{1:\ell-1};\varthetaa^{\star}\bigr);
\end{align}
\end{subequations}
here $\phi^\star$ is given by function $\phi$ from Definition \ref{def:thetam}, with $\phi=\phi^\star$ defined by choice $\varthetae=\varthetae^{\star}.$ The process of deploying \eqref{eq:auto33} to generate a sequence of increasing length starting from an input sequence is known as \emph{decoding}. Subsection~\ref{ssec:SS} describes how this decoding is started and how it is stopped.
In Subsection~\ref{ssec:detg} we consider deterministic techniques for sequence generation between the starting and stopping points, based on maximizing probabilities in \eqref{eq:auto33}.
Subsection \ref{ssec:stog} is devoted to randomized techniques for doing so, based on sampling.
And in Subsection \ref{ssec:mdstog} we illustrate how these two approaches may be combined.

\subsection{Starting and Stopping The Decoding}
\label{ssec:SS}

In general, the model in \eqref{eq:auto33} will be deployed by rolling out next-token prediction, starting 
from a given text \emph{prompt}. Examples of applications of LLMs, and the choice of prompts 
for these applications, are presented in Section~\ref{sec:applications}. Here we  simply assume that the prompt 
is given and comprises a sequence $\fp_{1:\ellp}$ for some $\ellp<L.$ We define $\FLp \subset \F^L$ by
\begin{equation} \label{eq:FLp}
    \FLp=\{f \in \F^L: f_m=\fp_m,\, m=1, \cdots, \ellp.\}
\end{equation}
The prompting sequence $\fp_{1:\ellp}$ is, in some settings, referred to as the \emph{context}. To separate the prompt from the generated text, a special token in the vocabulary known as {\tt{beginning of sequence}} ({\tt{BOS}}) is often used in element $\ellp$ of the sequence.

We may then use next-token prediction based on \eqref{eq:auto33} 
to find sequences conditioned on being in $\FLp.$ One of the tokens in the vocabulary is the {\tt{end of sequence}} ({\tt{EOS}}) token which typically arises at some step $\Ls$ satisfying $\ellp < \Ls \le L$. When this token arises, the
 LLM's response to the prompt is essentially complete. However a special {\tt{padding}} ({\tt{PAD}}) token is used
in each of the remaining entries of the sequence, in order to build a sequence of the standardized length $L.$\footnote{Often this standardized length is in fact defined as $L+\ellp$, rather than $L$, as the prompt is not viewed as part of the response.}
This padding corresponds to defining a transition probability in (\ref{eq:auto33}b) which is a Dirac measure on the {\tt{padding}} token whenever the conditioning sequence contains an {\tt{EOS}} token. 

From a probabilistic point of view, the preceding procedure corresponds to
\emph{conditioning} the next-token prediction model \eqref{eq:auto33} on the 
sequence $\fp_{1:\ellp}$. This sequence on which we condition generates the $\sigma-$algebra $\FLp.$

\subsection{Deterministic Generation} \label{ssec:detg}

The model \eqref{eq:auto33} for next-token prediction is probabilistic.
However, it may be desirable to derive from it a deterministic method
for sequence generation, given a prompt. 
One natural deterministic approach to using it is to seek the most likely 
sequence under the learned model, given the text prompt:
\begin{equation} \label{eq:globalMAP}
(f_1,\dots,f_L) \in \argmax_{f \in \FLp} \pi_{1:L}\bigl(f;\theta^\star\bigr).
\end{equation}
For a discrete state-space where each token can take one of $|\Voc|$ elements in the vocabulary, identifying the optimal sequence solving~\eqref{eq:globalMAP} can, in principle, be found by enumeration: computing the probability for all $|\Voc|^{L-\ellp}$ possible sequences and picking the one with maximum probability under the joint model. However, the computational cost of this procedure grows exponentially with $L$ and as a result it is intractable to solve~\eqref{eq:globalMAP} in practice. 

An alternative, tractable, approach is to use the autoregressive factorization of the joint probability distribution in \eqref{eq:auto33}. We solve the sequence of optimization problems
\begin{equation} \label{eq:greedy}
f_m \in \argmax_{v \in \Voc} \pi_m(v|f_{1:m-1};\theta), \qquad m = \ellp+1,\dots,L;
\end{equation}
this rolls-out the desired sequence.
Performing each optimization step requires computing the $|\Voc|$ probabilities for the elements in the vocabulary, leading to cost that
scales as $|\Voc| \times (L-\ellp)$. This procedure is known as \emph{greedy decoding}.
In general, the greedy procedure is sub-optimal with respect to the globally optimal solution in~\eqref{eq:globalMAP}. In particular, there may exist high probability paths that are hidden by greedily picking the most likely next-token and not considering a lower probability choice that results in a better overall solution when considering 
an entire sequence. We emphasize that the greedy solution does not change past elements of the sequence after they have been generated.  

\begin{example}
To ameliorate this effect, \emph{beam search} optimizes on a windowed sequence of fixed, short length $\ell$ greater than one~\cite{freitag2017beam}. To describe the method precisely we define $\pi_{m:m+\ell-1}(\cdot\,|f_{1:m-1})$ to be the probability distribution on sequence $f_{m:m+\ell-1}:=(f_m, \cdots, f_{m+\ell-1})$ of length $\ell$ implied by the autoregressive model \eqref{eq:auto}. The method
approximates the probability distribution on this sequence, and maximizes over it.
For example the basic version of beam search with $\ell=2$ chooses elements at step $m$ by solving the two-dimensional optimization problem:
$$f_{m:m+1} \in \argmax_{(v,v') \in \Voc \times \Voc} \pi_{m:m+1}(v,v'|f_{1:m-1}),  \qquad m = \ellp+1,\dots,L-1.$$
This has cost that scales as $|\Voc|^2.$ In practice the search is restricted to avoid the quadratic complexity in the
number of tokens. This is typically done by restricting the search related to $f_m$ to a subset of elements from $\Voc$ with total size $k \ll |\Voc|$, which is determined based on greedy search. This restriction reduces the complexity of the two-dimensional
optimization problem to $k |\Voc|$ leading to an overall cost that
scales as $k|\Voc| \times (L-\ellp)$ for generation of the entire
sequence. Although more expensive than the basic greedy approach, the improved quality of the decoded sequence, for example in code generation \cite{freitag2017beam}, may justify the approach in some settings. 
\end{example}

\subsection{Stochastic Generation} \label{ssec:stog}

Rather than adopting a deterministic approach to sequence generation by identifying some point of maximum probability, an alternative is to sample from $\pioL\bigl(\foL;\theta^\star\bigr)$, conditioned on the sequence being in $\FLp.$ It is possible to draw exactly from this conditional probability distribution by using next-token prediction
as follows:
$$f_m \sim \pi_m(\cdot|f_{1:m-1};\theta), \qquad m = \ellp+1,\dots,L.$$
Throughout this procedure $f_{1:\ellp}$ is fixed at the prompt $\fp_{1:\ellp},$
ensuring that the generated sequence is in $\FLp.$

\begin{remark} \label{rem:temp} Recall that a user-specified temperature parameter $\tau$ is chosen to define the learned next-token probability
model: see~\eqref{eq:next_token_model}. It is not uncommon for this parameter to be chosen differently when decoding than when 
training the model. Choosing temperature $\tau \ll 1$ when decoding increases the weight associated to the most likely token; hence the distribution over tokens becomes more concentrated. In the limit of $\tau \rightarrow 0$, the sampling procedure reduces to the greedy deterministic procedure described in the previous subsection. Choosing larger $\tau$ increases the probability of sampling rare tokens and hence the diversity of outputs from the LLM.
\end{remark}

\subsection{Mixing Deterministic and Stochastic Generation} \label{ssec:mdstog}
There are numerous ways of mixing the deterministic and stochastic approaches to decoding. A simple example is to consider sampling 
the first $L'$ tokens, in positions $\{\ellp+1,\dots,\ellp+L'\}$, and thereafter use greedy decoding:
\begin{align*}
& f_{m}  \sim \pi_{m}(\cdot|f_{1:m-1};\theta), \qquad &m = \ellp +1,\dots,\ellp+L'\\
& f_m  \in \argmax_{v \in \Voc} \pi_m(v|f_{1:m-1};\theta), \qquad &m = \ellp+L'+1,\dots,L.
\end{align*}

\begin{remark} In Subsection~\ref{ssec:qac} we introduce
chain-of-thought reasoning as a methodology to improve the quality of
responses provided by an LLM. In~\cite{wei2022chain} mixed
deterministic and stochastic generation, with $L'=1$, is used in this
context. Introducing stochasticity in the first token is found to increase 
diversity, and resulting accuracy, of the responses. 
\end{remark}

\section{Applications Of Autoregressive LLMs} \label{sec:applications}

In this section, we demonstrate how LLMs may be employed for various tasks. In order to present a consistent picture, our illustrations use the Qwen-2.5 model with 3B parameters~\cite{hui2024Qwen2}, accessed in November 2025.  Unless we state otherwise, this algorithm is deployed using the greedy decoding procedure, for sequence generation, that we described in Section~\ref{ssec:detg}. We illustrate question answering and conversation (Subsections \ref{ssec:qac} and \ref{ssec:conv}),
mathematical problem solving (Subsection \ref{ssec:icl}),  the generation of
computer code (Subsection \ref{ssec:code}), translation between different languages (Subsection \ref{ssec:trans}), text summarization (Subsection \ref{ssec:summ}) 
and multimodal data processing %
(Subsection \ref{ssec:mmode}).

In showcasing these applications of LLMs we will not only highlight the variety of
tasks for which LLMs may be used, but we will also focus on shortcomings of the
methodology; furthermore, we will describe approaches to address these shortcomings. 
Underlying these shortcomings is the inability to impose logical constraints within
an LLM; we will describe approaches such as chain-of-thought reasoning, in-context learning and fine-tuning which are used to address this lack of logical
constraint. All these approaches may be seen as attempts to increase the probability
of achieving desirable, or logically consistent, responses from the LLM, with
respect to randomness inherent in the sequence generation, or with respect to a probability measure on the prompts used to generate a response. 
In studying multi-modal processing we will also introduce encoder-decoder pairs, enabling LLMs to exploit other data modalities such
as images, audio and video\footnote{Code to reproduce all of the examples can be found at: \url{https://github.com/baptistar/LLMExamples}.}.

\subsection{Question Answering} \label{ssec:qac}

In this section we present examples of prompting the LLM with questions that, for a human, require logical reasoning to provide an answer. It is therefore of interest to study the performance of the LLM since the LLM is based entirely on a probabilistic model of next-token prediction, and has no logical constraints on its output. We start with the following example:

\begin{lstlisting}[language=json, mathescape=true, escapeinside={(*}{*)}]
"prompt":  A man named Frank travels from NYC to Philadelphia by road, on a Monday, leaving at 9:00am and arriving at 12noon. He then travels from Philadelphia to NYC the next day, Tuesday, also leaving at 9:00am and arriving at 12 noon. He takes the same route, but on Tuesday reverses the route taken on Monday. Is it necessarily the case that there is at least one point that he will be located at the same time, albeit on the two different days? Just give a Yes or No answer. 
"output": Yes<|endoftext|>
\end{lstlisting}

This answer is correct and so it might appear to a casual user that the LLM is able
to apply logical reasoning; but as we know this is not, in fact, the case. 
We probe this issue by introducing the concept of
Chain-of-Thought (CoT) reasoning. This refers to a group of techniques that are used to improve the ability of LLMs to solve complex problems by including reasoning within their response. The simplest strategy to achieve this is by appending to the prompt an instruction such as ``Explain your answer''.  We now study the output of the model, in response to the same question about Frank's journeys, but now where the model is additionally prompted to include a step-by-step explanations. Perhaps the most succinct explanation for the correct answer of ``Yes'' is to explain that the question is logically equivalent to one in which two distinct people make the journey at the same time on the same day; once this logical equivalence is noted it is clear that they  must therefore cross paths at least once. However, as the response in Appendix~\ref{ssec:aqac} shows, the LLM provides a rather verbose answer, centered around the intermediate-value theorem, when asked to provide its reasoning. \rev{Furthermore, the response contains  unjustified statements such as ``Since he [Frank] takes the same route on both days, his position as a function of time will be continuous.'' While the later is naturally true for humans, it does not follow from taking the same route on both days, but rather from the fact that humans cannot teleport. In addition, the response contains mathematical errors, such as modeling the position on both days as being the same at the final time, despite the fact that Frank ends in two distinct cities.}

\begin{remark} Despite the mixed response shown in Appendix~\ref{ssec:aqac}, there are settings in which CoT reasoning has measurable positive effect on the correctness of the final answer. The potential benefits from using the approach are highlighted in~\cite{wei2022chain}. 
\end{remark}

CoT reasoning attempts to impose logical structure, approximately, by guiding the LLM to provide \emph{output} which explains the answer to the question. In contrast, In-context Learning (ICL), described in Subsection~\ref{ssec:icl}, supplements the \emph{input} to the LLM in a manner which aims
to to impose logical structure, approximately. The word \emph{approximately}, in both cases, can be
interpreted as increasing the probability that the LLM produces the correct answer to the
question at hand; or to an increase in the probability that the LLM produces a logically consistent reasoning leading to the correct answer. This probabilistic notion of
accuracy is to be interpreted with respect to the inherent probabilistic
nature of next-token prediction 
and with respect to a probability distribution on prompts; see Section~\ref{sec:evaluation}.

\subsection{Conversations} \label{ssec:conv}

The previous section demonstrates question and answering where there is a definitive,
logically correct, answer. In this subsection, we describe question and answering in
a different context: prompts which require extraction of knowledge contained in the
data used to train the LLM. We also use this setting to illustrate conversing with an LLM, when a user requests refinement of the original question. We arrange the discussion
around Bayes' theorem:
\begin{lstlisting}[language=json, mathescape=true, escapeinside={(*}{*)}]
"prompt":  Explain Bayes' theorem simply without formulas.
"output": Bayes' theorem is a way to update our beliefs about something based on new information. It helps us figure out how likely an event is, given some evidence.<|endoftext|>
\end{lstlisting}
The answer is reasonable and can be followed up to ask for more technical detail; the output for a follow-up question that requests a formal presentation of Bayes' theorem is provided in Appendix~\ref{ssec:aconv}. The response given there is factually correct, succinct and describes some context as well as the underlying mathematics. It is notable that complex knowledge of this type is encoded in the capabilities of an LLM and can be extracted by next-token prediction.

Taken together, the example in this subsection and the linked appendix demonstrate the impressive abilities of LLMs to extract knowledge, using next-token prediction. On the other hand, the examples in Subsection~\ref{ssec:qac} and follow-up in the appendix demonstrate the potential pitfalls related to the lack of logical constraints 
on the next-token prediction methodology.

\begin{remark}
The text explaining Bayes' theorem is generated from an LLM that was trained on a static data-set from a particular point of time in the past. Since Bayes' theorem itself
does not change with time, this does not present a problem. However,
for some queries, such as planning an itinerary for a vacation, or preparation of a legal document, it can
be important to access up-to-date information from the internet. In this context,
a particularly effective use of fine-tuning (a concept introduced in Subsection \ref{ssec:emplossFT}) is a methodology known
as retrieval-augmented generation \cite{lewis2020retrieval}. Fine-tuning can, for example, be used to incorporate knowledge from Wikipedia. The version of Qwen that we consistently deploy in this paper does not have this capability; but other more recent versions do. 
\end{remark}

\subsection{Mathematics} \label{ssec:icl}

The reader who has followed our development of LLMs based on next-token prediction will be aware that there is no constraint of logical consistency imposed by the methodology. For example there is nothing to prevent an LLM implying, of even explicitly stating, as part of a response, that {\tt{2*2=5}}. The probability of it doing so might well be low; but there is no hard constraint to prevent it. In this subsection we probe the ability of LLMs to tackle tasks in mathematics. The lack of logical constraints presents as a significant barrier in this arena.

An important general concept when using LLMs is the question of how to incorporate additional prior knowledge, relevant to a specific limited set of tasks, to improve the performance of the LLM. One approach is through fine-tuning, the methodology outlined in Subsection \ref{ssec:emplossFT}, and as studied through an example in Subsection \ref{ssec:code}. A second approach is through ICL, discussed above, and which we now illustrate in detail in the context of mathematical problem solving.

ICL works by adding, to the input, text known as \emph{context} which aims to guide the LLM so that next-token prediction gives desirable output. To illustrate this idea, we consider specific multiplication problems such as $473 \times 45.$ Without context the LLM provides the incorrect answer $21385.$ With appropriately chosen context, it provides the correct answer $21285.$ In Appendix~\ref{ssec:acontext} we show how the basic LLM arrives at its incorrect answer, we show the context from~\cite{chen2023skills} that we used, and we show how the LLM, equipped with this context, arrives at the correct answer. In the same appendix, we also show the model output with stochasticity and how the inclusion of stochasticity leads to an improved probability of producing the correct answer, with and without ICL. 

\begin{remark} CoT thought reasoning can be combined with ICL, providing not only examples of response to similar prompts, but also ensuring that the examples themselves have step-by-step responses. The use of specialized sampling strategies that elicit a reasoning path, rather than generation of deterministic sequences, typically leads to answers that contain fewer factual errors or logical inconsistencies. As with our discussion at the end of Subsection \ref{ssec:qac}, such statements about improved approximation of the desired response, should be interpreted probabilistically.
\end{remark}

\subsection{Coding} \label{ssec:code}

We now illustrate the use of LLMs to generate computer code. We highlight the role of fine-tuning, introduced in Subsection \ref{ssec:emplossFT}, a methodology for introducing task-specific information to improve an LLM. In contrast to ICL, which employs additional information at the time that the LLM is deployed, fine-tuning uses additional
information in a second training phase; see~\cite{hui2024Qwen2} for the details of the coding-specific dataset consisting of over 5.5 trillion tokens that was used to fine tune the Qwen-2.5-Coder model series. We used the 3B model to generate the following examples. 

We give two examples of the LLM response in Appendix~\ref{ssec:acode}, first without fine-tuning, and then with fine-tuning. Both examples produce an ordinary differential equation (ODE) that is integrated using {\tt RK45.} This choice coincides with the method that {\sc{Matlab}} recommends as the first choice to solve non-stiff ODEs~\cite{MathWorks_ChooseODE}. However, neither code is general purpose, enabling solution of an arbitrary ODE in $\R^n$; instead, a specific example in two dimensions is given. In Appendix~\ref{ssec:acode} we also include the
graphical output produced by both codes when running the {\sc{Matlab}} code to solve the specific second order ODE chosen by next token prediction. We note that while both models produce a valid {\sc{Matlab}} script without syntax errors, the first model, which is not fine-tuned for coding, produces an example whose solution is unstable and blows up in the selected time range.

\subsection{Translation} \label{ssec:trans}

Translation has been a guiding application in the development of LLMs since the field emerged over the last decade. It can be performed by simply using a prompt that requests the translation of a provided piece of text into a specified, different, language. 
One approach to enhance a general purpose LLM for translation is to use fine-tuning, introduced in Subsection~\ref{ssec:emplossFT}. Fine-tuning can be performed by using pairs of translated sequences $\{(f_{1:L}^{(n)}, f'_{1:L'})^{(n)}\}_{n=1}^{N_{\mathsf{fine}}}$, where $f_{1:L} \sim p(f_{1:L})$ is a sequence in a given language, and $f'_{1:L'} = \mathcal{T}(f_{1:L})$ where $\mathcal{T}$ is a mapping of the sequence to a different language. We note that a given piece of text will typically be represented by different numbers of tokens in different languages; hence $L' \ne L$ in general.
However, current LLMs are so powerful that they can successfully translate, without being fine-tuned for translation; this approach is known as \emph{zero-shot} translation. The example below shows zero-shot translation of a short phrase from English into French, and back again. This example was performed using Qwen2.5-3B-Instruct model, which was fine tuned for improved instruction following, but not for language specifically. It is notable, but should not be surprising, that the map from
English to French is not inverted exactly, but only approximately.

\begin{lstlisting}[language=json, mathescape=true, escapeinside={(*}{*)}]
"prompt": Translate ``Thank you for your help with answering my math question.'' into French.
"output": ``Merci pour votre aide à répondre à ma question de mathématiques.''<|im_end|>
"prompt": Translate ``Merci pour votre aide à répondre à ma question de mathématiques'' into English
"output": ``Thank you for your help in answering my math question.''<|im_end|>
\end{lstlisting}

\subsection{Summarization} \label{ssec:summ}

LLMs are also used to summarize large corpora of documents that would be time-consuming for humans to process manually. LLMs achieve this task by identifying latent semantic representations of the corpora's content, in some cases using a hierarchical procedure to summarize intermediate summaries to mitigate the context window constraint for large corpora~\cite{guu2020retrieval, zhang2020pegasus}. The following example shows LLMs applied to summarize the introduction (Section~\ref{sec:introduction}) of this paper. We omit the introduction in the prompt below for conciseness. We note that it was provided in a LaTeX format, emphasizing the model's ability to process text in various languages and formats.

\begin{lstlisting}[language=json, mathescape=true, escapeinside={(*}{*)}]
"prompt": Summarize the following introduction section of a paper into a single sentence. 
"output": This paper aims to present a precise mathematical formulation of large language models, elucidating their core methodology and underlying principles for a mathematically inclined audience, thereby addressing their current shortcomings and offering a foundation for further computational and theoretical studies by mathematicians.<|im_end|>
\end{lstlisting}

\subsection{Multimodal Processing And Generation} \label{ssec:mmode}

Multimodal Large Language Models (MLLMs) extend the capabilities of LLMs, which
were originally developed for processing text sequences, to other data modalities;
typical examples of other data modalities include images, audio and video. MLLMs commonly leverage a pre-trained LLM, often referred to as the \emph{backbone} model, that autoregressively process and generate elements in $\F_{e}^M$, sequences of length $M$ taking values in $\R^{e}$, as presented in Sections~\ref{sec:seq} and~\ref{sec:ntp}. In this subsection, we will illustrate MLLMs by describing how to leverage a backbone model for \emph{image understanding}: the task of processing an input image and generating an output text sequence. This procedure will immediately generalize to processing other multimodal data, as well as combinations of data from different modalities. We will then describe one popular multimodal processing task known as \emph{Visual Question Answering} (VQA), where the input contains an image, together with a text prompt containing a question based on the image; and the output is the answer to that question.

\subsubsection{Image Understanding} \label{ssec:IU}

For image understanding, we assume that the image is represented as a sequence $f^{\mathsf{im}}_{1:M'}$ in $\F_{\ell'}^{M'}$: the space of sequences of length $M'$ taking values in $\R^{d_{\ell'}}$. For an example of how different data modalities are represented as a sequence, see Section~\ref{sec:encoding_other_modalities}; there we have a brief discussion of how an image is represented as a flattened sequence. To process the image, an MLLM first encodes the input sequence in $\F_{\ell'}^{M'}$ into the sequence space for the backbone LLM. This is achieved by defining the map $\mathsf{M}^{\im} \colon \F_{\ell'}^{M'} \rightarrow \F_{e}^M$. The map is commonly composed of two pieces: $\mathsf{M}^{\im} = \mathsf{P}^{\im} \circ \mathsf{E}^{\im}.$ Here $\mathsf{E}^{\im}$ is a pretrained \emph{encoder} $\mathsf{E}^{\im}: \F_{\ell'}^{M'} \to \F_{\ell}^M$ for images; typically this pretraining will have been undertaken independently of the training of the backbone LLM and, as a consequence, typically $d_\ell \ne \de.$ The encoder uses specialized architectures for images or other data types; for example a vision transformer is typically used for processing images into a sequence. The map $\mathsf{P}^{\im}: \F_{\ell}^M \to \F_{e}^M$ is based on an elementwise \emph{projector} $\mathsf{P}^{\im}: \R^{d_\ell} \to \R^{d_e}$, extended to operate on sequences of length $M$ according to the definition $(\mathsf{P}^{\im}(f^{\im}))_k := \mathsf{P}^{\im}(f_k^{\im})$ for $k=1,\dots,M$. %
The projector $\mathsf{P}^{\im}(\cdot;\varthetaa^{\im})$ is commonly taken to be a feedforward neural network with parameters $\varthetaa^\im$ that are chosen to align the range of the encoder with the input space to the LLM. These parameters are identified using a fine-tuning procedure built from the following data. 

\begin{dataassumption} \label{das:image-text}
    We collect pairs of images and text sequences drawn
    i.i.d.\thinspace from a joint distribution $\{f_{1:M'}^{\im,(n)}, f_{1:L}^{(n)}\}_{n=1}^{N_{\im}} \sim \mathsf{p}(f^{\im}_{1:M'},f_{1:L})$.
\end{dataassumption}

Using this data, we learn the parameters $\varthetaa^{\im}$ of the projector $\mathsf{P}^{\im}(\cdot;\varthetaa^{\im});$ the parameters of the backbone LLM and the encoder are fixed during this fine-tuning process. The empirical objective function minimized, using the same procedures described in Subsection~\ref{sec:minimizing_empirical}, is given by:
\begin{equation}
\label{eq:lossN_ImageText}    
\Loss^{N_{\im}}(\theta) = -\sum_{n=1}^{N_{\im}} \log\pioL \bigl(\foL^{(n)}\vert\mathsf{P}^{\im}(\mathsf{E}^{\im}(f_{1:M'}^{\im,(n)});\varthetaa^{\im});\thetas\bigr),
\end{equation}
where $\theta := \vartheta^{\im}$ and $\thetas$ denotes the LLM parameters found through pretraining.

After training, the backbone LLM may now be used to generate text in response to an image input, using the generative approaches outlined in Section~\ref{sec:seqgen}, with
the image as a prompt. Given an image, not in the training data, we embed it into Euclidean space using the encoder $\mathsf{E}^{\im}$ and learned projector $\mathsf{P}^{\im}$; the resulting Euclidean sequence may be used as the prompt for the LLM, as in Subsection~\ref{ssec:SS}. The generated output tokens can then be decoded into text using the vocabulary. 

\begin{remark}
    We note that modalities other than images can also be formatted as a Euclidean sequence, enabling generalization of the preceding methodology on image understanding
    to other tasks such as audio or video understanding. Audio is a sequence natively indexed by time; video may be time-indexed with image (and possibly sound) at each time, although other organizations of the space-time data may also be considered.
\end{remark}

This basic idea, to generate text from an image or other data modality, may
be further refined. For example, it may be used to provide reasoning or follow instructions based on semantic visual and language information. We now illustrate this idea with VQA.

\subsubsection{Visual Question and Answering}

We now build on the preceding subsubsection to describe VQA: question answering when the question refers to an input image and together the question and image form the prompt. When the input consists of more than one modality, such as a text sequence $\fp_{1:\ellp}$ and an associated image $f_{1:\ell_{\im}}^{\im}$, we process each input into the same Euclidean space. After embedding, the text and image are sequences in $\F^{\ellp}_e$ and $\F^{\ell_{\im}}_e$, respectively for $\ellp + \ell_{\im} < L$. By concatenation, we can define the prompting sequence to be:
\begin{equation*}
    \FLp=\{f \in \F^L: f_m=\fp_m,\, m=1, \cdots, \ellp, f_m=f^{\im}_m,\, m=\ellp+1, \cdots,\ellp+\ell_{\im}\}.
\end{equation*}
We can then generate an output text sequence of maximum context length $L$ given the prompt. We now provide an example using the Qwen-3 model for VQA. The following shows an example where the LLM is asked to describe an image from a simulation to the Navier-Stokes equation; see~\cite{wan2023debias} for details on the image. The output text correctly identifies the vortical structures present in the flow and suggests that it might be (decaying) turbulence created by flow past an obstacle or in a channel. These are reasonable suggestions, consistent with the image. 

\begin{lstlisting}[language=json, mathescape=true, escapeinside={(*}{*)}]
"prompt": Describe this image representing a solution to Navier Stokes equation.
\end{lstlisting}
\begin{figure}[!ht]
    \centering
    \includegraphics[width=0.25\linewidth]{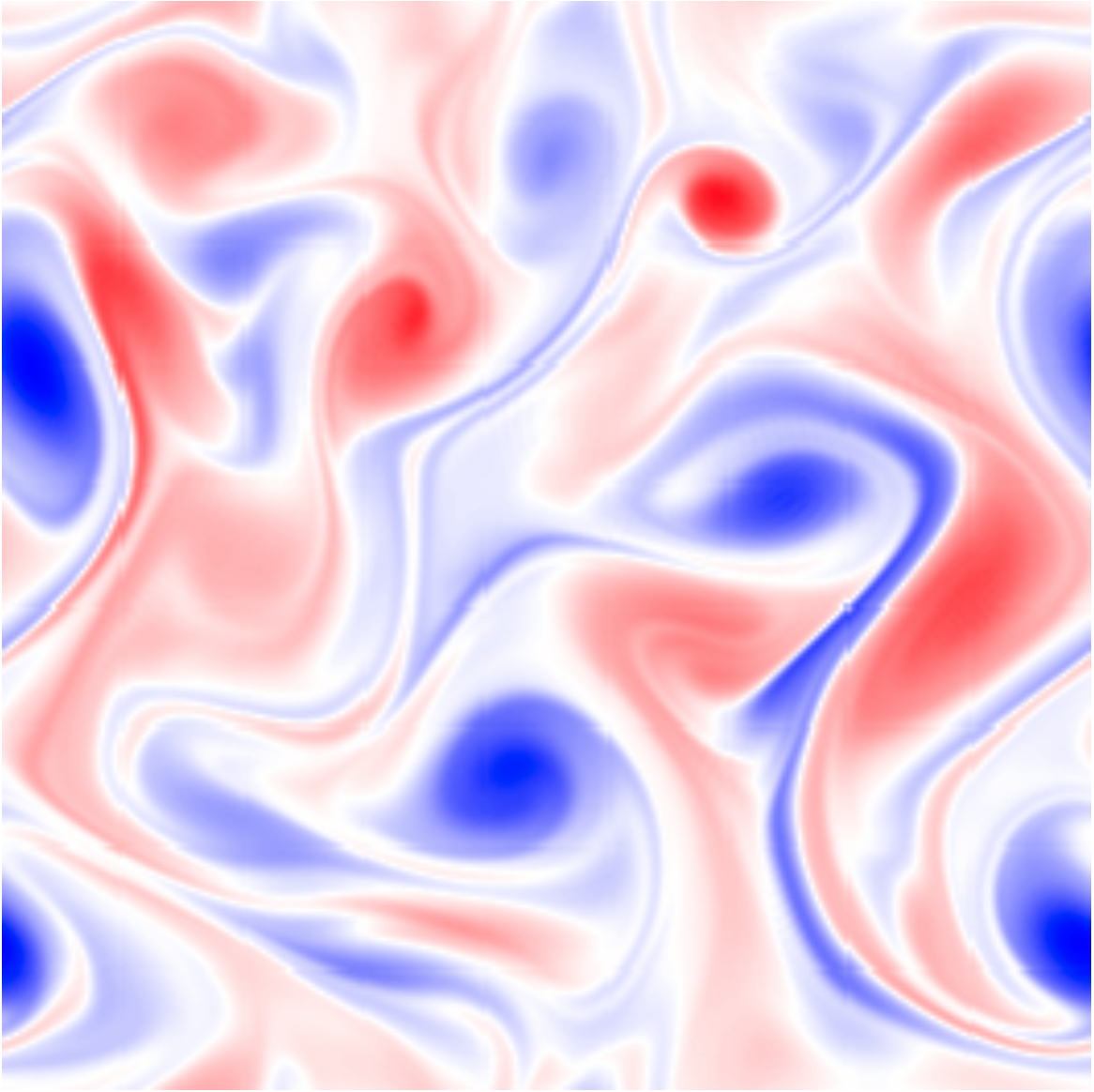}
    \label{fig:placeholder}
\end{figure}
\begin{lstlisting}[language=json, mathescape=true, escapeinside={(*}{*)}]
"output": This image is a visualization of a solution to the Navier-Stokes equations, likely representing a two-dimensional turbulent or transitional flow field. While the exact physical context (e.g., flow past an obstacle, channel flow, or decaying turbulence) isn't specified, the visual features strongly suggest it depicts vortical structures and coherent eddies in a fluid.
\end{lstlisting}

\begin{remark} A MLLM can also be used to generate data from modalities other than text. We illustrate this by explaining how the generated data may itself be an image.
(Generalizations to producing data from other modalities are similar). An example of a problem where this is needed is the use of an LLM to produce an edited image, given a starting image and a description of the features to be changed. Generating image data rather than text data requires a mapping from Euclidean sequences representing text to Euclidean sequences representing images. Recall that in Subsubsection~\ref{ssec:IU} we have already introduced a map in the converse direction: $\mathsf{M}^{\im} \colon \F_{\ell'}^{M'} \rightarrow \F_{e}^M$. Here we introduce 
$\D^{\im} \colon \F_e^M \rightarrow \F_\ell^{M'}$, an approximate inverse. %
We can use $\D^{\im}$ to map a (Euclidean representation of a) text sequence to a flattened representation of an image as a Euclidean sequence. 

Although we have presented the inverse map as working on Euclidean sequences, for
simplicity of exposition, the map $\D^{\im}$ is typically built from a decoder which works over a discrete vocabulary. This involves creating an (approximate) 
representation of any image data sample as a sequence over a finite vocabulary, often referred to as a \emph{codebook}. Note that a finite vocabulary is not used to process image data into text, as described in Subsubsection~\ref{ssec:IU}; it is just needed for generation. Approaches to learn the vocabulary and the decoders include DC-GAN and VQ-VAE for images~\cite{radford2015unsupervised, razavi2019generating}. The decoder architectures for different data modalities follow a similar structure to the encoder architectures described in Subsubsection~\ref{ssec:IU} that compose a pre-trained map with fine-tuned projector. Fine-tuning the encoder and decoder parameters requires datasets of paired samples where each sample contains related representations of a common underlying object in every data modality, as in Data Assumption~\ref{das:image-text} for image and text data. This approach has been used to build models that can simultaneously understand and/or generate up to 21 different data modalities~\cite{bachmann20244m}.
\end{remark}

\section{Evaluation of LLMs} \label{sec:evaluation}

The preceding sections have shown how LLMs work using probabilistic next token prediction, highlighted their remarkable capabilities across numerous applications of the methodology, and demonstrated shortcomings, in particular in relation to their inability to impose logical structure as hard constraints. Given both the potential
wide-ranging uses of LLMs, together with their demonstrable shortcomings, evaluation
of them is clearly of importance. This section is devoted to how the question of
evaluation may be confronted. In Subsection \ref{ssec:evalM} we give an abstract
mathematical formulation of evaluation, followed in Subsection \ref{ssec:evalB} by
citation to papers conducting benchmarking in practice. 

\subsection{Mathematical Formulation} \label{ssec:evalM}

Recall from \eqref{eq:auto} the true probability distribution defined, in principle, by all available text:
\begin{equation}
    \label{eq:autoZ}
\poL(\foL)=\Pi_{\ell=1}^L \pl(f_\ell|f_{1:\ell-1}).
\end{equation}
From \eqref{eq:auto3_shared} and \eqref{eq:qarc} we have our approximate probability distribution on text, taking the form:
\begin{equation}
    \label{eq:auto3Z}
\pioL(\foL;\theta)=\Pi_{\ell=1}^L \ql\bigl(f_l|\phi \circ f_{1:l-1};\varthetaa\bigr).
\end{equation}
After application of a specific choice of one of the training procedures detailed in Section \ref{sec:train}, a choice $\thetas$ is made for the parameter vector defining the LLM
\eqref{eq:auto3Z}. This choice is aimed at aligning the LLM with the true probability distribution 
\eqref{eq:autoZ}. We can now ask for  precise measures of the performance of the LLM on
specific tasks.

Recall $\FLp$ from equation \eqref{eq:FLp},
representing the $\sigma-$algebra on sequences in $\F^L$ created by a fixed text prompt $\fp_{1:\ellp}$. This text prompt may contain in-context information; see Section~\ref{ssec:icl}. Given a specific text prompt defining $\FLp$ we might be interested in the difference between the conditional probabilities $\poL(\foL|\FLp)$ and $\pioL(\foL;\theta|\FLp).$ Using a divergence $\dvg: \prob(\F^L) \times \prob(\F^L) \to \R^+$, we define a measure of the effectiveness of the LLM, for this specific text prompt, to be
\begin{equation} \label{eq:LLMeffectiveness0}
\E(\fp_{1:\ellp}) := \dvg\Bigl(\poL(\foL|\FLp)|| \pioL(\foL;\thetas|\FLp)\Bigr).
\end{equation}
This divergence quantifies the difference between the ``true'' distribution of responses (constrained to have total length less than $L$) given a specific prompt,
and the distribution of the LLM, given the same prompt. Averaging over text
prompts from a given distribution $\pip$, noting that such a distribution will typically
be over prompts of different lengths, gives a measure of the average effectiveness of a trained LLM over a class of prompts:
\begin{equation} \label{eq:LLMeffectiveness}
\mathsf{e}=\mathbb{E}_{\fp_{1:\ellp} \sim \pip} \E(\fp_{1:\ellp}).
\end{equation}
The choice of distribution on text prompts defines
a class of problems of interest.
The choice of $\dvg$ is crucial in making evaluation of the
effectiveness feasible in practice.

\begin{example} A natural choice for $\dvg$ in \eqref{eq:LLMeffectiveness0} 
is the KL divergence~\eqref{eq:KLdivergence} between the approximate model distribution and the data distribution; this leads to the same objective used to train the LLM in Section~\ref{sec:train}, constrained by the specific text prompt. Making this choice and evaluating the average effectiveness $\mathsf{e}$ given in~\eqref{eq:LLMeffectiveness} leads to 
$$\mathsf{e} = \mathsf{const} - \mathbb{E}_{\fp_{1:\ellp} \sim \pip} \mathbb{E}_{\foL \sim \poL(\cdot|\FLp)} \log \pioL(\foL;\thetas|\FLp);$$
here $\mathsf{const}$ is a scalar that is independent of the model parameters $\thetas$. Thus, the following second term can be used to compare the effectiveness of two different parameter values:
$$-\mathbb{E}_{\fp_{1:\ellp} \sim \pip} \mathbb{E}_{\foL \sim \poL(\cdot|\FLp)} \log \pioL(\foL;\thetas|\FLp).$$
In particular, a lower value for the preceding expressions indicates that the approximate model is closer in KL divergence to the data distribution. An empirical estimate of this expression, given i.i.d.\thinspace samples of prompts and responses drawn from the data distribution, is referred to as the negative log-likelihood or cross-entropy of the model. The exponentiated negative log-likelihood of a single sequence is often referred to as~\emph{perplexity}.
\end{example}

\begin{example} 
Another potentially useful choice of $\dvg$ is the energy distance or maximum mean discrepancy (MMD). This is because both can be evaluated using only samples; see \citep[Chapter 11]{bach2024inverse}. Evaluating the effectiveness $\E(\fp_{1:\ellp})$
using the energy distance or MMD, for a given prompt, requires generating samples from the approximate model and collecting samples from the data distribution which are consistent with the prompt. In practice, the samples may not have the same length so it may be advantageous to project both samples into a common space. For example, we may ask the LLM for a ``Yes/No'' response and process each sample from the data distribution in a binary space prior to computing the divergence.
\end{example}

For logical problems with a single, correct response, we can compare the modeled distribution with a specific answer. This can be achieved using scoring-rules $\mathsf{S}: \prob(\F^L) \times \F^L \rightarrow \R$ which compare a distribution to a point. Examples of scoring rules include continuous ranked probability score (CRPS) or the energy score\footnote{In one dimension, for example, the CRPS computes the square of the $L^2$ norm of the difference between the cumulative distribution function (CDF) of the distribution and the CDF of a Dirac mass at the point.}. The expectation of the scoring rule over a distribution on the point can be used to define a divergence. See~\citep[Chapter 11]{bach2024inverse} for the definitions and details of these rules. When applying LLMs to problems with a single correct response, this will be the point that enters the scoring rule; the distribution will be that of the LLM.

For evaluation of logical problems with a single correct response, another
approach is to summarize the modeled probabilities with a single point. This, for
example, might be the answer implied by the output of the greedy decoding algorithm 
defined by~\eqref{eq:greedy}; and then this point is compared directly with the correct response in some metric or, if the answer is numerical, a norm. When the model is used to retrieve a set of information, its output can also be evaluated by counting the number of true positives, false positives, and false negatives in the output relative to a ground-truth set. These counts are often used to report evaluation metrics such as \emph{precision} and \emph{recall}, the proportion of true positive instances identified by the LLM relative to the total positive instances (true and false positives) and actually correct instances (true positives and false negatives), respectively. Intuitively, precision measures the model's correctness when it makes a claim, while recall measures how much of the desired truth is actually output by the model.

\subsection{Benchmarking and Evaluation Metrics} \label{ssec:evalB}

Recently, several datasets have been proposed to test the multitask accuracy of
a model, across a range of academic and professional domains including mathematics, computer science, history and law~\cite{hendrycks2020measuring, srivastava2023beyond}. These benchmarks test the ability of the model to retrieve knowledge and answer reasoning tasks. Benchmarks for the chain-of-thought methodology, discussed in Subsection \ref{ssec:qac}, which evaluate reasoning, are introduced in \cite{suzgun2023challenging, jiang2025mme} and reviewed in  surveys~\cite{DBLP:journals/corr/abs-2502-12289, mondorf2024beyond}.

Translation tasks, introduced in Subsection~\ref{ssec:trans}, are difficult
to benchmark, particularly for longer texts; availability of sufficient data for translation across multiple languages is a bottleneck in this area. We
refer the reader to the discussion in~\cite{chauhan2023comprehensive}.
The paper~\cite{nllb2024scaling} introduces the idea of using
transfer learning to enable translation between several hundred different languages,
and thereby enrich the available data sets.
Summarization, the subject of Subsection~\ref{ssec:summ}, is studied, for news articles, in the recent paper~\cite{zhang2024benchmarking}; as for translation tasks, it is in general difficult to obtain large data sets on which to evaluate the performance of a summarization algorithm.

Recently, there has been increasing use of LLMs to evaluate other LLMs, as well as to generate text sequences for training models. This self-referential approach has clear potential drawbacks; the paper~\cite{alemohammad2024self} illustrates these drawbacks directly in the context of using AI-generated image-data to train image-generation models.

\section{Diffusion-Based Models} \label{sec:mask}

LLMs based on the autoregressive factorization~\eqref{eq:auto3_shared} currently predominate the landscape of models. However, it is entirely possible that other approaches may eventually supersede this approach. Language itself can have long-range correlations that are not necessarily well-represented by the autoregressive factorization which
imposes a fixed ordering for text generation. In the setting of image generation, \emph{diffusion-based generative modeling} has proven remarkably successful. A recently adopted line of inquiry asks how this methodology might be repurposed to generate text.

Diffusion-based generative modeling uses a forward stochastic process, defined on Euclidean space, to add noise to a data set until it is dominated
by noise; reverse processes~\cite{anderson1982reverse} may be identified which map this noise back to the data set. By learning a regularized approximation of this reverse process, a methodology may be constructed to (approximately) generate new images from the distribution underlying the data set, by initializing the reverse process from noise \cite{song2020score}. Choosing the forward process to be a Gaussian diffusion process (stochastic differential equation) enables efficient learning of the reverse process.

In the context of language, an alternative to sequential generation of each element of a text sequence is to predict the entire sequence (or part of a sequence) using an
analog of diffusion-based generative modeling. This approach may be formulated by introducing a forward process that corrupts the entire sequence by progressively (and randomly) masking parts of it; as for diffusions, reverse processes exist which take a masked sequence back to the data distribution. Learning a regularized version of this reverse process leads to a methodology that produces an unmasked sequence drawn (approximately) from the data distribution of text sequences.  This procedure is known as a \emph{discrete diffusion model}, recognizing the
fact that text sequences contain elements in a finite vocabulary $\Voc$. The primary advantage of diffusion models over autoregressive models is that they generate an entire sequence at once, allowing for more efficient and sophisticated representation of causality relationships in the text sequence. The trade-off is that the reverse process requires multiple iterative steps for inference, which can have a potentially greater cost.

Recall that $\F^L$ contains sequences of length $L$ taking values in $\Voc.$
For such a sequence, the total number of possible states is $|\Voc|^L$. We now define a forward process
that yields a time-indexed  
probability distribution $\mathsf{p}(\cdot,t) \in \mathcal{P}(\F^L)$ for $t \geq 0$. This process is initialized at the data distribution in Assumption~\ref{das:sequences} at $t =0$, i.e., $\mathsf{p}(\cdot,0) = \mathsf{p}_{1:L}$. The sequence of distributions then evolves according to the following  continuous-time Markov chain (CTMC) \cite{norris1998markov}:
\begin{equation} \label{eq:CTMC}
\frac{\partial \mathsf{p}}{\partial t} = Q\mathsf{p}, \quad \mathsf{p}(\cdot,0) =  \mathsf{p}_{1:L}.
\end{equation}
The matrix-valued function of time $Q(t) \in \R^{|\Voc|^L \times |\Voc|^L}$ is known as the transition rate matrix. The matrix entries satisfy $\sum_{j} Q(t)_{ij} = 0$ for $i = 1,\dots,|\Voc|^L$~\cite{norris1998markov}; this ensures that solutions $\mathsf{p}(\cdot,t)$ defined by \eqref{eq:CTMC} indeed remain in $\mathcal{P}(\F^L)$. The solution to system~\eqref{eq:CTMC} defines the marginal probability of a time-indexed sequence $f_{1:L}(\cdot):[0,T] \to \F^L$ being in each state in $\F^L$ at time $t \ge 0$, starting from an initial condition $f_{1:L}(0) \sim \mathsf{p}_{1:L}$. 

\begin{remark} Recall from Section~\ref{ssec:bpe} that we chose to simply identify
each element in $\Voc$ with an integer in $\{1, \cdots, |\Voc|\}.$ Thus we may identify the space of sequences of length $L$ to be $\{1,\dots,|\Voc|^L\}$.
Recalling that a typical byte-pair encoded vocabulary has size $\mathcal{O}(10^4)$
(Example \ref{ex:22})
and that $L=\mathcal{O}(10^4)$ for state of the art LLMs (Example \ref{ex:31}) we see that the state-space of this Markov chain can be enormous. Hence, it is necessary to impose a sparse structure on $Q(\cdot)$ to define tractable numerical algorithms. 
\end{remark}

In practice, it is common to choose $Q(t) = \sigma^2(t) Q^0$ for some diffusion coefficient $\sigma(t)$ and some fixed $Q^0$. That is, 
\begin{equation} \label{eq:CTMC2}
\frac{\partial \mathsf{p}}{\partial t} = \sigma^2(t) Q^0 \mathsf{p}.
\end{equation}
If we define function $g(\cdot)$ by $g(t)=\int_0^t \sigma^2(s)ds$ and new time $\tau=g(t)$ then in terms of new time variable $\tau$
\eqref{eq:CTMC2} becomes, for $\mathsf{p}'(\tau)$ defined by $\mathsf{p}'\bigl(g(t)\bigr)=\mathsf{p}(t)$,
\begin{equation} \label{eq:CTMC3}
\frac{\partial \mathsf{p}'}{\partial \tau} =  Q^0 \mathsf{p}'.
\end{equation}
If we assume that this process is ergodic so that there is a unique reference distribution $\mathsf{p}_{\textrm{ref}} \in \mathcal{P}(\F^L)$ with property 
$\mathsf{p}'(\tau) \to \mathsf{p}_{\textrm{ref}}$ as $\tau \to \infty$, and that $\sigma(\cdot)$ is such that $g \to \infty$ as $t \to \infty$,
then $\mathsf{p}(t) \to \mathsf{p}_{\textrm{ref}}$ as $t \to \infty.$
The following two examples present common choices for $Q^0$ in this setting. The two choices result, respectively, in a uniform reference distribution over $\F^L$, or a reference distribution concentrated on a specific sequence in $\F^L$ in which each element takes value given by the same specific token (often referred to as a \emph{mask}. In both cases we assume that $Q(t) = \sigma^2(t) Q^0$ for some diffusion coefficient $\sigma(t)$ and some fixed $Q^0$ chosen to obtain the desired reference distribution $\mathsf{p}_{\textrm{ref}}.$

\begin{example}[Uniform distribution] Consider the rate-transition matrix $Q^0$ defined by
$$Q^0_{ij} = \begin{cases} 1 - \frac{|\Voc| - 1}{|\Voc|}\beta & \text{if } i=j \\ \frac{\beta}{|\Voc|} & \text{if } i \neq j\end{cases}.$$
The matrix $Q^0$ has any constant vector in its null-space, as required.  
Thus, the process in~\eqref{eq:CTMC} is ergodic \cite{norris1998markov} and converges to the uniform distribution $(\mathsf{p}_{\textrm{ref}})_{k} = \frac{1}{|\Voc|}$ for $k = 1,\dots,|\Voc|$
as $t \to \infty.$ By choosing large $T$ we obtain $\mathsf{p}(\cdot;T) \approx $$\mathsf{p}_{\textrm{ref}}.$
\end{example}

\begin{example}[Absorbing state] Let $m$ represent the index of an absorbing state and consider the rate-transition matrix
$$Q^0_{ij} = \begin{cases} 1 & \text{if } j=i=m \\ \beta & \text{if } j = m, i \neq m \\ 1 - \beta & \text{if } j \neq m, i \neq m \\ 0 & \text{otherwise} \end{cases} $$
As in the previous example, the matrix $Q^0$ has any constant vector in its null-space. And again
the process in~\eqref{eq:CTMC} is ergodic \cite{norris1998markov} and converges to the Dirac distribution $\mathsf{p}_{\textrm{ref}} = \delta_{m}$ as $t \to \infty.$
\end{example}

The state-space picture of the CTMC in~\eqref{eq:CTMC} is a Markov process for the sequence $(f_{1:L}(t))_{t \geq 0}$ in $\F^L$ that evolves continuously in time with random holding times. As a result, the process is a piecewise-constant trajectory that transitions between two states at jump times. The time between jumps is a random variable, known as the holding time, which has an exponential distribution with rate parameter $Q_{ii}(t)$~\cite{norris1998markov}. This provides the basis of numerical methods for simulating CTMCs, such as the step-by-step Gillespie algorithm \cite{gillespie1977exact} (that simulates each transition) and the $\tau$-leaping algorithm \cite{cao2006efficient} (that applies all transitions in a time interval at once); see~\cite{campbell2022continuous} for application of these ideas in the context of discrete state-space diffusion models.

Given the forward process, we consider a reverse process of the form
\begin{equation} \label{eq:reverseCTMC}
\frac{\partial \mathsf{\bar{p}}}{\partial t} = \overline{Q}\mathsf{\bar{p}}, \quad \mathsf{\bar{p}}(\cdot,T) = \mathsf{p}(\cdot,T),
\end{equation}
evolving backwards in time from $t = T$ to $t = 0$. Here 
$\overline{Q}_{ij}(t) = Q_{ji}(t)\frac{\mathsf{p}(j,t)}{\mathsf{p}(i,t)}$
is the reverse rate matrix. Then, the
choice $\mathsf{\bar{p}}(\cdot,t) = \mathsf{p}(\cdot,t)$ for all $t \in [0,T]$ solves~\eqref{eq:reverseCTMC}; see the proof of this result in Appendix~\ref{app:discrete_reversibility}. In particular $\mathsf{\bar{p}}(\cdot,0) = \mathsf{p}(\cdot,0)=\mathsf{p}_{1:L}.$ As for the forward process, we can simulate a sequence $(\bar{f}_{1:L}(t))_{t \in [0,T]}$ in $\F^L$ that satisfies the reverse process.
The goal of discrete diffusions is to learn the reverse rate-transition matrix $\overline{Q}$ and simulate the reverse process to sample from the data distribution by
sampling $\mathsf{\bar{p}}(\cdot,T).$ In practice $\overline{Q}$ is known only approximately and the reverse process is initiated at $\mathsf{\bar{p}}(\cdot,T)=\mathsf{p}_{\textrm{ref}}$, which approximates $\mathsf{p}(\cdot,T).$ Hence the samples are only approximate samples from the desired distribution $\mathsf{p}_{1:L}.$ 

To define the discrete reverse process, we may learn a model for the rate transition matrix $\overline{Q}$ using a discrete-analog of \emph{score matching}, the methodology
underlying diffusion-based image generation \cite{song2020score}. We also refer the reader to~\cite{campbell2022continuous}, which introduces a learning objective in discrete-time based on lower bounds for the log-likelihood of the data, and~\cite{campbell2024generative} for formulations based on flow matching. In~\cite{lou2024discrete}, the authors propose to learn a network $s(\cdot,\cdot;\theta)\colon \F^L \times [0,T] \rightarrow |\Voc|^L$ that approximates the probability ratios of the marginal probabilities appearing in the reverse rate matrix. That is, $\mathsf{s}(i,t;\theta)_{j} \approx \frac{\mathsf{p}(j,t)}{\mathsf{p}(i,t)}$ for all pairs of inputs $i,j \in \F^L$. To identify $\mathsf{s}$, we may minimize the loss function
$$\mathsf{L}(\theta) := \int_0^T \mathbb{E}_{i \sim \mathsf{p}(\cdot,t)}\left[\sum_{j \neq i} \mathsf{s}(i,t;\theta)_{j} - \frac{\mathsf{p}(j,t)}{\mathsf{p}(i,t)} \log \mathsf{s}(i,t;\theta)_{j}\right] \textrm{d}t.$$
Given that $s \mapsto s - \alpha \log s$ is minimized at $s = \alpha$ we have that $\mathsf{L}$ is minimized at the desired probability ratio. Moreover, by the tower rule, the loss can be equivalently written in terms of the conditional distributions, which are specified by the marginal forward process in~\eqref{eq:CTMC} and hence known. That is,
$$\mathsf{L}(\theta) = \int_0^T \mathbb{E}_{i \sim \mathsf{p}(\cdot,t|k), k \sim \mathsf{p}_{1:L}}\left[\sum_{j \neq i} \mathsf{s}(i,t;\theta)_{j} - \frac{\mathsf{p}(j,t|k)}{\mathsf{p}(i,t|k)} \log \mathsf{s}(i,t;\theta)_{j}\right] \textrm{d}t.$$
This loss is implementable given samples from the data distribution and yields an estimator for the reverse rate matrix, which can be used to generate new samples by evolving the reverse process in~\eqref{eq:reverseCTMC} with this approximate rate. That is, by evolving $\overline{f}(t)$ backwards-in-time starting from $\overline{f}(T) \sim \mathsf{p}_{\textrm{ref}}$ to $t = 0$ according to the CTMC with the rate matrix whose entries are given by 
\begin{equation} \label{eq:reverseCTMC2}
\overline{Q}_{ij}(t) = Q_{ji}(t)\mathsf{s}(i,t;\theta)_{j},
\end{equation}
for all $i,j$. The resulting sequence $\overline{f}(0)$ is an approximate sample from $\mathsf{p}_{1:L}$.

\section{Summary and Challenges}
\label{sec:SAC}

\rev{
\paragraph{Beyond Autoregressive and Transformer-Based Methodologies} As explained
at the end of Subsection \ref{ssec:GOP}, our focus in this paper is on
autoregressive and transformer-based LLMs because they have dominated the landscape
over the past half decade, and continue to do so. However, as also highlighted at the end of Subsection~\ref{ssec:GOP}, and in the literature review
Subsection~\ref{ssec:literature}, different approaches are showing promise, 
and the field is rapidly evolving. Further
mathematization of diffusion-based models, in place of autoregressive
models, and of state-space models or hybrid models, in place of transformer-based
models, is very much of interest and presents opportunities for the applied
mathematics community.
}

\paragraph{Text Encoding in Euclidean Space:} In Remark \ref{rem:byte}
and Example \ref{ex:31}
we highlight the trade-off between vocabulary size and sequence length.
The question of how to optimally balance this trade-off will
depend on the details of the LLM architecture. Current algorithms
simply demonstrate empirical success with the use of byte-pair encoding.
Analysis of this trade-off has the potential to guide new architectural 
developments, linked to concepts from information theory.
Relatedly, the question of how to embed sequences in $\Voc$ as sequences
in $\Voc_e$ leaves open many natural questions, alluded to in
Remark \ref{rem:close}: what is the optimal
choice of dimension $\de$? Do LLMs automatically learn an embedding for text sequences that has semantic meaning and how does regularity in the embedding space affect the resulting model learning problems?

\paragraph{Long Context Length:} As mentioned in Subsection~\ref{ssec:icl}, prompts which contain substantial information (context) are used to avoid fine-tuning of LLMs and adapt the model to specific tasks. However, the computational complexity of next-token prediction models for long sequence generation creates a bottleneck in this approach. This has motivated new methodologies that can compress long sequences. One such methodology replaces commonplace fragments of sequences using GIST tokens~\cite{mu2023learning, petrov2025long}. Another constraint in most practical LLMs is the maximum context length. For example, the diffusion processes in Section~\ref{sec:mask} set this length in advance and sample sequences of fixed length. It is of interest to develop  approaches that can operate across different dimensional spaces, and which can adjust the maximum context length based on each task.

\paragraph{Robustness:} LLMs have been shown to result in wide variety of responses under adversarial perturbations to their inputs~\cite{xu2023llm, zhu2023promptrobust}; this has potential to affect their safe development and public trust in the technology. In particular,
quite reasonable text prompts can lead to \emph{hallucinations}: generated content which is factually incorrect or nonsensical. An important 
direction for study is to investigate the sensitivity of the next-token probability model, and the resulting semantic meaning of the generated output, to various perturbations in the model. These might include the input context, model parameters and fine-tuning. This will be useful for designing architectures whose output is more robust to certain perturbations. \rev{The natural mathematical setting for this
question is Section \ref{sec:seqgen}. We can formulate questions of robustness
in terms of the continuity, or Lipschitz continuity, of maps taking inputs to
the LLM (such as parameter $\theta$, prompt, context) into outputs; both
deterministic (Subsection~\ref{ssec:detg}) and probabilistic (Subsection~\ref{ssec:stog}) outputs can be considered, by using appropriate metrics on the vocabulary space or distribution space, respectively.}

\paragraph{Optimal Use Of Data:} Remark \ref{rem:4p7} discusses the fact that typically the desired  population level objective function is approximated
in a biased way, using correlated data samples.
A statistical evaluation of this finite-sample estimator, in comparison with unbiased ample-based approximation, 
is an interesting research direction. More generally, there are several interesting questions about the trade-off between accuracy and
computational complexity in the use of data. Another example is in the decomposition of data to enable it to fit into the
context window; what is the optimal way to break up and use the data?

\paragraph{Scaling Laws:} Recent empirical studies \rev{suggest} that the accuracy of LLMs (as measured using various metrics, see Section~\ref{sec:evaluation}) \rev{tends to} follow seemingly deterministic power-law relationships with increasing model parameters, the volume of training data, and the amount of computational resources (FLOPs)~\cite{kaplan2020scaling, hoffmann2022training}. These trends are often used to inform economic decisions by predicting the performance of models under different training configurations; for example to assess whether training a larger model on more data is cost-effective. However, scaling laws fit simple functional forms to historical training data and assume that power-laws extrapolate to new model settings. Developing mechanistic models that explain these trends by accounting for different architectures, variations in dataset quality, and changing parameters in optimization algorithms with quantified uncertainty will underpin these empirical studies, and yield scaling laws that extrapolate more reliably. 

\paragraph{Synthetic Data and Watermarking:} With the increasing amount of human-generated data being used for training (see Remark~\ref{rem:training_data}), there is growing possibility that models will be trained on synthetic data that has been generated by another model. This behavior can lead to a degradation in the quality of the generated samples and a bias relative to the true data distribution, a phenomenon referred to as \emph{mode collapse}~\cite{shumailov2024ai, alemohammad2024self}. To avoid this, \emph{watermarking} is  used, an approach to embed a signal in the generated text that can be used to detect LLM-generated data~\cite{kirchenbauer2023watermark}. Such data can be excluded from fine-tuning, or can be weighted differently from naturally occurring data; research is needed to understand how best to do this.

\paragraph{Test-Time Reasoning:} In Subsection \ref{ssec:qac} we discussed the Chain-of-Thought approach to making reasoning explicit via prompting. A related emerging %
approach is \emph{test-time reasoning} (also known as test-time scaling or thinking). This refers to devoting appreciable computing resources at inference-time to improve the correctness of the model's predictions on hard tasks. 
In several modern deployments, this is achieved through interaction of
the model with the user by generating  intermediate tokens (sometimes hidden, sometimes shown) before committing to an answer. The value of the trade-off between latency/cost of this procedure
and its reliability will depend on the task. Algorithmic approaches for test-time reasoning include eliciting longer chains-of-thought \cite{wei2022chain}, sampling multiple reasoning paths and selecting a subset of them for consistency \cite{wang2022self}, and more explicit search over thought states, for example via Tree/Forest-of-Thought methods~\cite{yao2023tree}. Mathematically, a natural perspective is to treat thoughts as latent variables. Test-time reasoning may then be viewed as a form of budget-constrained inference with the goal of maximizing an evaluation metric; see Section~\ref{sec:evaluation}. 
For example, we can optimize over trajectories to shift the model's output distribution toward more globally consistent responses. Deeper understanding of this process has the potential to create
more effective algorithms.

\section*{Acknowledgments}
The work of AMS is partially supported by a Department of Defense (DoD) Vannevar Bush Faculty Fellowship (award N00014-22-1-2790). 

\bibliographystyle{imsart-number}
\bibliography{references}  

\begin{thebibliography}{128}

\bibitem{abdin2024phi}
\begin{barticle}[author]
\bauthor{\bsnm{Abdin},~\bfnm{Marah}\binits{M.}},
  \bauthor{\bsnm{Aneja},~\bfnm{Jyoti}\binits{J.}},
  \bauthor{\bsnm{Behl},~\bfnm{Harkirat}\binits{H.}},
  \bauthor{\bsnm{Bubeck},~\bfnm{S{\'e}bastien}\binits{S.}},
  \bauthor{\bsnm{Eldan},~\bfnm{Ronen}\binits{R.}},
  \bauthor{\bsnm{Gunasekar},~\bfnm{Suriya}\binits{S.}},
  \bauthor{\bsnm{Harrison},~\bfnm{Michael}\binits{M.}},
  \bauthor{\bsnm{Hewett},~\bfnm{Russell~J}\binits{R.~J.}},
  \bauthor{\bsnm{Javaheripi},~\bfnm{Mojan}\binits{M.}},
  \bauthor{\bsnm{Kauffmann},~\bfnm{Piero}\binits{P.}} \betal{et~al.}
(\byear{2024}).
\btitle{Phi-4 technical report}.
\bjournal{arXiv preprint arXiv:2412.08905}.
\end{barticle}
\endbibitem

\bibitem{achiam2023gpt}
\begin{barticle}[author]
\bauthor{\bsnm{Achiam},~\bfnm{Josh}\binits{J.}},
  \bauthor{\bsnm{Adler},~\bfnm{Steven}\binits{S.}},
  \bauthor{\bsnm{Agarwal},~\bfnm{Sandhini}\binits{S.}},
  \bauthor{\bsnm{Ahmad},~\bfnm{Lama}\binits{L.}},
  \bauthor{\bsnm{Akkaya},~\bfnm{Ilge}\binits{I.}},
  \bauthor{\bsnm{Aleman},~\bfnm{Florencia~Leoni}\binits{F.~L.}},
  \bauthor{\bsnm{Almeida},~\bfnm{Diogo}\binits{D.}},
  \bauthor{\bsnm{Altenschmidt},~\bfnm{Janko}\binits{J.}},
  \bauthor{\bsnm{Altman},~\bfnm{Sam}\binits{S.}},
  \bauthor{\bsnm{Anadkat},~\bfnm{Shyamal}\binits{S.}} \betal{et~al.}
(\byear{2023}).
\btitle{{GPT}-4 technical report}.
\bjournal{arXiv preprint arXiv:2303.08774}.
\end{barticle}
\endbibitem

\bibitem{adam2014method}
\begin{barticle}[author]
\bauthor{\bsnm{Adam},~\bfnm{Kingma DP Ba~J}\binits{K.~D. B.~J.}} \betal{et~al.}
(\byear{2014}).
\btitle{A method for stochastic optimization}.
\bjournal{arXiv preprint arXiv:1412.6980}
\bvolume{1412}.
\end{barticle}
\endbibitem

\bibitem{alemohammad2024self}
\begin{binproceedings}[author]
\bauthor{\bsnm{Alemohammad},~\bfnm{Sina}\binits{S.}},
  \bauthor{\bsnm{Casco-Rodriguez},~\bfnm{Josue}\binits{J.}},
  \bauthor{\bsnm{Luzi},~\bfnm{Lorenzo}\binits{L.}},
  \bauthor{\bsnm{Humayun},~\bfnm{Ahmed~Imtiaz}\binits{A.~I.}},
  \bauthor{\bsnm{Babaei},~\bfnm{Hossein}\binits{H.}},
  \bauthor{\bsnm{LeJeune},~\bfnm{Daniel}\binits{D.}},
  \bauthor{\bsnm{Siahkoohi},~\bfnm{Ali}\binits{A.}} \AND
  \bauthor{\bsnm{Baraniuk},~\bfnm{Richard}\binits{R.}}
(\byear{2024}).
\btitle{Self-Consuming Generative Models Go {MAD}}.
In \bbooktitle{The Twelfth International Conference on Learning
  Representations}.
\end{binproceedings}
\endbibitem

\bibitem{anderson1982reverse}
\begin{barticle}[author]
\bauthor{\bsnm{Anderson},~\bfnm{Brian~DO}\binits{B.~D.}}
(\byear{1982}).
\btitle{Reverse-time diffusion equation models}.
\bjournal{Stochastic Processes and their Applications}
\bvolume{12}
\bpages{313--326}.
\end{barticle}
\endbibitem

\bibitem{ansari2024chronos}
\begin{barticle}[author]
\bauthor{\bsnm{Ansari},~\bfnm{Abdul~Fatir}\binits{A.~F.}},
  \bauthor{\bsnm{Stella},~\bfnm{Lorenzo}\binits{L.}},
  \bauthor{\bsnm{Turkmen},~\bfnm{Caner}\binits{C.}},
  \bauthor{\bsnm{Zhang},~\bfnm{Xiyuan}\binits{X.}},
  \bauthor{\bsnm{Mercado},~\bfnm{Pedro}\binits{P.}},
  \bauthor{\bsnm{Shen},~\bfnm{Huibin}\binits{H.}},
  \bauthor{\bsnm{Shchur},~\bfnm{Oleksandr}\binits{O.}},
  \bauthor{\bsnm{Rangapuram},~\bfnm{Syama~Sundar}\binits{S.~S.}},
  \bauthor{\bsnm{Arango},~\bfnm{Sebastian~Pineda}\binits{S.~P.}},
  \bauthor{\bsnm{Kapoor},~\bfnm{Shubham}\binits{S.}} \betal{et~al.}
(\byear{2024}).
\btitle{Chronos: Learning the language of time series}.
\bjournal{arXiv preprint arXiv:2403.07815}.
\end{barticle}
\endbibitem

\bibitem{claude4_2025}
\begin{bmisc}[author]
\bauthor{\bsnm{{Anthropic}}}
(\byear{2025}).
\btitle{{Claude 4 Family: Opus, Sonnet, and Haiku Models}}.
\bhowpublished{\url{https://www.anthropic.com/index/claude-4}}.
\bnote{Accessed: 2025-07-29}.
\end{bmisc}
\endbibitem

\bibitem{austin2021structured}
\begin{barticle}[author]
\bauthor{\bsnm{Austin},~\bfnm{Jacob}\binits{J.}},
  \bauthor{\bsnm{Johnson},~\bfnm{Daniel~D}\binits{D.~D.}},
  \bauthor{\bsnm{Ho},~\bfnm{Jonathan}\binits{J.}},
  \bauthor{\bsnm{Tarlow},~\bfnm{Daniel}\binits{D.}} \AND \bauthor{\bsnm{Van
  Den~Berg},~\bfnm{Rianne}\binits{R.}}
(\byear{2021}).
\btitle{Structured denoising diffusion models in discrete state-spaces}.
\bjournal{Advances in neural information processing systems}
\bvolume{34}
\bpages{17981--17993}.
\end{barticle}
\endbibitem

\bibitem{azar2024general}
\begin{binproceedings}[author]
\bauthor{\bsnm{Azar},~\bfnm{Mohammad~Gheshlaghi}\binits{M.~G.}},
  \bauthor{\bsnm{Guo},~\bfnm{Zhaohan~Daniel}\binits{Z.~D.}},
  \bauthor{\bsnm{Piot},~\bfnm{Bilal}\binits{B.}},
  \bauthor{\bsnm{Munos},~\bfnm{Remi}\binits{R.}},
  \bauthor{\bsnm{Rowland},~\bfnm{Mark}\binits{M.}},
  \bauthor{\bsnm{Valko},~\bfnm{Michal}\binits{M.}} \AND
  \bauthor{\bsnm{Calandriello},~\bfnm{Daniele}\binits{D.}}
(\byear{2024}).
\btitle{A general theoretical paradigm to understand learning from human
  preferences}.
In \bbooktitle{International Conference on Artificial Intelligence and
  Statistics}
\bpages{4447--4455}.
\bpublisher{PMLR}.
\end{binproceedings}
\endbibitem

\bibitem{ba2016layer}
\begin{barticle}[author]
\bauthor{\bsnm{Ba},~\bfnm{Jimmy~Lei}\binits{J.~L.}},
  \bauthor{\bsnm{Kiros},~\bfnm{Jamie~Ryan}\binits{J.~R.}} \AND
  \bauthor{\bsnm{Hinton},~\bfnm{Geoffrey~E}\binits{G.~E.}}
(\byear{2016}).
\btitle{Layer normalization}.
\bjournal{arXiv preprint arXiv:1607.06450}.
\end{barticle}
\endbibitem

\bibitem{bach2025learning}
\begin{barticle}[author]
\bauthor{\bsnm{Bach},~\bfnm{Eviatar}\binits{E.}},
  \bauthor{\bsnm{Baptista},~\bfnm{Ricardo}\binits{R.}},
  \bauthor{\bsnm{Calvello},~\bfnm{Edoardo}\binits{E.}},
  \bauthor{\bsnm{Chen},~\bfnm{Bohan}\binits{B.}} \AND
  \bauthor{\bsnm{Stuart},~\bfnm{Andrew}\binits{A.}}
(\byear{2025}).
\btitle{Learning Enhanced Ensemble Filters}.
\bjournal{arXiv preprint arXiv:2504.17836}.
\end{barticle}
\endbibitem

\bibitem{bach2024inverse}
\begin{barticle}[author]
\bauthor{\bsnm{Bach},~\bfnm{Eviatar}\binits{E.}},
  \bauthor{\bsnm{Baptista},~\bfnm{Ricardo}\binits{R.}},
  \bauthor{\bsnm{Sanz-Alonso},~\bfnm{Daniel}\binits{D.}} \AND
  \bauthor{\bsnm{Stuart},~\bfnm{Andrew}\binits{A.}}
(\byear{2024}).
\btitle{Machine Learning for Inverse Problems and Data Assimilation}.
\bjournal{arXiv preprint arXiv:2410.10523}.
\end{barticle}
\endbibitem

\bibitem{bachmann20244m}
\begin{barticle}[author]
\bauthor{\bsnm{Bachmann},~\bfnm{Roman}\binits{R.}},
  \bauthor{\bsnm{Kar},~\bfnm{O{\u{g}}uzhan~F}\binits{O.~F.}},
  \bauthor{\bsnm{Mizrahi},~\bfnm{David}\binits{D.}},
  \bauthor{\bsnm{Garjani},~\bfnm{Ali}\binits{A.}},
  \bauthor{\bsnm{Gao},~\bfnm{Mingfei}\binits{M.}},
  \bauthor{\bsnm{Griffiths},~\bfnm{David}\binits{D.}},
  \bauthor{\bsnm{Hu},~\bfnm{Jiaming}\binits{J.}},
  \bauthor{\bsnm{Dehghan},~\bfnm{Afshin}\binits{A.}} \AND
  \bauthor{\bsnm{Zamir},~\bfnm{Amir}\binits{A.}}
(\byear{2024}).
\btitle{4{M}-21: An any-to-any vision model for tens of tasks and modalities}.
\bjournal{Advances in Neural Information Processing Systems}
\bvolume{37}
\bpages{61872--61911}.
\end{barticle}
\endbibitem

\bibitem{baptista2025mathematical}
\begin{barticle}[author]
\bauthor{\bsnm{Baptista},~\bfnm{Ricardo}\binits{R.}},
  \bauthor{\bsnm{Stuart},~\bfnm{Andrew~M}\binits{A.~M.}} \AND
  \bauthor{\bsnm{Tran},~\bfnm{Son}\binits{S.}}
(\byear{2025}).
\btitle{A Mathematical Perspective On Contrastive Learning}.
\bjournal{arXiv preprint arXiv:2505.24134}.
\end{barticle}
\endbibitem

\bibitem{bengio2003neural}
\begin{barticle}[author]
\bauthor{\bsnm{Bengio},~\bfnm{Yoshua}\binits{Y.}},
  \bauthor{\bsnm{Ducharme},~\bfnm{R{\'e}jean}\binits{R.}},
  \bauthor{\bsnm{Vincent},~\bfnm{Pascal}\binits{P.}} \AND
  \bauthor{\bsnm{Jauvin},~\bfnm{Christian}\binits{C.}}
(\byear{2003}).
\btitle{A neural probabilistic language model}.
\bjournal{Journal of machine learning research}
\bvolume{3}
\bpages{1137--1155}.
\end{barticle}
\endbibitem

\bibitem{blakeman2025nemotron}
\begin{barticle}[author]
\bauthor{\bsnm{Blakeman},~\bfnm{Aaron}\binits{A.}},
  \bauthor{\bsnm{Basant},~\bfnm{Aarti}\binits{A.}},
  \bauthor{\bsnm{Khattar},~\bfnm{Abhinav}\binits{A.}},
  \bauthor{\bsnm{Renduchintala},~\bfnm{Adithya}\binits{A.}},
  \bauthor{\bsnm{Bercovich},~\bfnm{Akhiad}\binits{A.}},
  \bauthor{\bsnm{Ficek},~\bfnm{Aleksander}\binits{A.}},
  \bauthor{\bsnm{Bjorlin},~\bfnm{Alexis}\binits{A.}},
  \bauthor{\bsnm{Taghibakhshi},~\bfnm{Ali}\binits{A.}},
  \bauthor{\bsnm{Deshmukh},~\bfnm{Amala~Sanjay}\binits{A.~S.}},
  \bauthor{\bsnm{Mahabaleshwarkar},~\bfnm{Ameya~Sunil}\binits{A.~S.}}
  \betal{et~al.}
(\byear{2025}).
\btitle{Nemotron-h: A family of accurate and efficient hybrid mamba-transformer
  models}.
\bjournal{arXiv preprint arXiv:2504.03624}.
\end{barticle}
\endbibitem

\bibitem{bridle1990probabilistic}
\begin{bincollection}[author]
\bauthor{\bsnm{Bridle},~\bfnm{John~S}\binits{J.~S.}}
(\byear{1990}).
\btitle{Probabilistic interpretation of feedforward classification network
  outputs, with relationships to statistical pattern recognition}.
In \bbooktitle{Neurocomputing: Algorithms, architectures and applications}
\bpages{227--236}.
\bpublisher{Springer}.
\end{bincollection}
\endbibitem

\bibitem{brown2020language}
\begin{barticle}[author]
\bauthor{\bsnm{Brown},~\bfnm{Tom}\binits{T.}},
  \bauthor{\bsnm{Mann},~\bfnm{Benjamin}\binits{B.}},
  \bauthor{\bsnm{Ryder},~\bfnm{Nick}\binits{N.}},
  \bauthor{\bsnm{Subbiah},~\bfnm{Melanie}\binits{M.}},
  \bauthor{\bsnm{Kaplan},~\bfnm{Jared~D}\binits{J.~D.}},
  \bauthor{\bsnm{Dhariwal},~\bfnm{Prafulla}\binits{P.}},
  \bauthor{\bsnm{Neelakantan},~\bfnm{Arvind}\binits{A.}},
  \bauthor{\bsnm{Shyam},~\bfnm{Pranav}\binits{P.}},
  \bauthor{\bsnm{Sastry},~\bfnm{Girish}\binits{G.}},
  \bauthor{\bsnm{Askell},~\bfnm{Amanda}\binits{A.}} \betal{et~al.}
(\byear{2020}).
\btitle{Language models are few-shot learners}.
\bjournal{Advances in neural information processing systems}
\bvolume{33}
\bpages{1877--1901}.
\end{barticle}
\endbibitem

\bibitem{calvello2024continuum}
\begin{barticle}[author]
\bauthor{\bsnm{Calvello},~\bfnm{Edoardo}\binits{E.}},
  \bauthor{\bsnm{Kovachki},~\bfnm{Nikola~B}\binits{N.~B.}},
  \bauthor{\bsnm{Levine},~\bfnm{Matthew~E}\binits{M.~E.}} \AND
  \bauthor{\bsnm{Stuart},~\bfnm{Andrew~M}\binits{A.~M.}}
(\byear{2024}).
\btitle{Continuum attention for neural operators}.
\bjournal{arXiv preprint arXiv:2406.06486}.
\end{barticle}
\endbibitem

\bibitem{campbell2022continuous}
\begin{barticle}[author]
\bauthor{\bsnm{Campbell},~\bfnm{Andrew}\binits{A.}},
  \bauthor{\bsnm{Benton},~\bfnm{Joe}\binits{J.}},
  \bauthor{\bsnm{De~Bortoli},~\bfnm{Valentin}\binits{V.}},
  \bauthor{\bsnm{Rainforth},~\bfnm{Thomas}\binits{T.}},
  \bauthor{\bsnm{Deligiannidis},~\bfnm{George}\binits{G.}} \AND
  \bauthor{\bsnm{Doucet},~\bfnm{Arnaud}\binits{A.}}
(\byear{2022}).
\btitle{A continuous time framework for discrete denoising models}.
\bjournal{Advances in Neural Information Processing Systems}
\bvolume{35}
\bpages{28266--28279}.
\end{barticle}
\endbibitem

\bibitem{campbell2024generative}
\begin{binproceedings}[author]
\bauthor{\bsnm{Campbell},~\bfnm{Andrew}\binits{A.}},
  \bauthor{\bsnm{Yim},~\bfnm{Jason}\binits{J.}},
  \bauthor{\bsnm{Barzilay},~\bfnm{Regina}\binits{R.}},
  \bauthor{\bsnm{Rainforth},~\bfnm{Tom}\binits{T.}} \AND
  \bauthor{\bsnm{Jaakkola},~\bfnm{Tommi}\binits{T.}}
(\byear{2024}).
\btitle{Generative Flows on Discrete State-Spaces: Enabling Multimodal Flows
  with Applications to Protein Co-Design}.
In \bbooktitle{International Conference on Machine Learning}
\bpages{5453--5512}.
\bpublisher{PMLR}.
\end{binproceedings}
\endbibitem

\bibitem{cao2006efficient}
\begin{barticle}[author]
\bauthor{\bsnm{Cao},~\bfnm{Yang}\binits{Y.}},
  \bauthor{\bsnm{Gillespie},~\bfnm{Daniel~T}\binits{D.~T.}} \AND
  \bauthor{\bsnm{Petzold},~\bfnm{Linda~R}\binits{L.~R.}}
(\byear{2006}).
\btitle{Efficient step size selection for the tau-leaping simulation method}.
\bjournal{The Journal of chemical physics}
\bvolume{124}.
\end{barticle}
\endbibitem

\bibitem{castin2025unified}
\begin{barticle}[author]
\bauthor{\bsnm{Castin},~\bfnm{Val{\'e}rie}\binits{V.}},
  \bauthor{\bsnm{Ablin},~\bfnm{Pierre}\binits{P.}},
  \bauthor{\bsnm{Carrillo},~\bfnm{Jos{\'e}~Antonio}\binits{J.~A.}} \AND
  \bauthor{\bsnm{Peyr{\'e}},~\bfnm{Gabriel}\binits{G.}}
(\byear{2025}).
\btitle{A unified perspective on the dynamics of deep transformers}.
\bjournal{arXiv preprint arXiv:2501.18322}.
\end{barticle}
\endbibitem

\bibitem{chauhan2023comprehensive}
\begin{barticle}[author]
\bauthor{\bsnm{Chauhan},~\bfnm{Shweta}\binits{S.}} \AND
  \bauthor{\bsnm{Daniel},~\bfnm{Philemon}\binits{P.}}
(\byear{2023}).
\btitle{A comprehensive survey on various fully automatic machine translation
  evaluation metrics}.
\bjournal{Neural Processing Letters}
\bvolume{55}
\bpages{12663--12717}.
\end{barticle}
\endbibitem

\bibitem{chen2023skills}
\begin{barticle}[author]
\bauthor{\bsnm{Chen},~\bfnm{Jiaao}\binits{J.}},
  \bauthor{\bsnm{Pan},~\bfnm{Xiaoman}\binits{X.}},
  \bauthor{\bsnm{Yu},~\bfnm{Dian}\binits{D.}},
  \bauthor{\bsnm{Song},~\bfnm{Kaiqiang}\binits{K.}},
  \bauthor{\bsnm{Wang},~\bfnm{Xiaoyang}\binits{X.}},
  \bauthor{\bsnm{Yu},~\bfnm{Dong}\binits{D.}} \AND
  \bauthor{\bsnm{Chen},~\bfnm{Jianshu}\binits{J.}}
(\byear{2023}).
\btitle{Skills-in-context prompting: Unlocking compositionality in large
  language models}.
\bjournal{arXiv preprint arXiv:2308.00304}.
\end{barticle}
\endbibitem

\bibitem{chen2021decision}
\begin{barticle}[author]
\bauthor{\bsnm{Chen},~\bfnm{Lili}\binits{L.}},
  \bauthor{\bsnm{Lu},~\bfnm{Kevin}\binits{K.}},
  \bauthor{\bsnm{Rajeswaran},~\bfnm{Aravind}\binits{A.}},
  \bauthor{\bsnm{Lee},~\bfnm{Kimin}\binits{K.}},
  \bauthor{\bsnm{Grover},~\bfnm{Aditya}\binits{A.}},
  \bauthor{\bsnm{Laskin},~\bfnm{Misha}\binits{M.}},
  \bauthor{\bsnm{Abbeel},~\bfnm{Pieter}\binits{P.}},
  \bauthor{\bsnm{Srinivas},~\bfnm{Aravind}\binits{A.}} \AND
  \bauthor{\bsnm{Mordatch},~\bfnm{Igor}\binits{I.}}
(\byear{2021}).
\btitle{Decision transformer: Reinforcement learning via sequence modeling}.
\bjournal{Advances in neural information processing systems}
\bvolume{34}
\bpages{15084--15097}.
\end{barticle}
\endbibitem

\bibitem{chen2025quantitative}
\begin{barticle}[author]
\bauthor{\bsnm{Chen},~\bfnm{Shi}\binits{S.}},
  \bauthor{\bsnm{Lin},~\bfnm{Zhengjiang}\binits{Z.}},
  \bauthor{\bsnm{Polyanskiy},~\bfnm{Yury}\binits{Y.}} \AND
  \bauthor{\bsnm{Rigollet},~\bfnm{Philippe}\binits{P.}}
(\byear{2025}).
\btitle{Quantitative Clustering in Mean-Field Transformer Models}.
\bjournal{arXiv preprint arXiv:2504.14697}.
\end{barticle}
\endbibitem

\bibitem{child2019generating}
\begin{barticle}[author]
\bauthor{\bsnm{Child},~\bfnm{Rewon}\binits{R.}},
  \bauthor{\bsnm{Gray},~\bfnm{Scott}\binits{S.}},
  \bauthor{\bsnm{Radford},~\bfnm{Alec}\binits{A.}} \AND
  \bauthor{\bsnm{Sutskever},~\bfnm{Ilya}\binits{I.}}
(\byear{2019}).
\btitle{Generating long sequences with sparse transformers}.
\bjournal{arXiv preprint arXiv:1904.10509}.
\end{barticle}
\endbibitem

\bibitem{choromanski2020rethinking}
\begin{barticle}[author]
\bauthor{\bsnm{Choromanski},~\bfnm{Krzysztof}\binits{K.}},
  \bauthor{\bsnm{Likhosherstov},~\bfnm{Valerii}\binits{V.}},
  \bauthor{\bsnm{Dohan},~\bfnm{David}\binits{D.}},
  \bauthor{\bsnm{Song},~\bfnm{Xingyou}\binits{X.}},
  \bauthor{\bsnm{Gane},~\bfnm{Andreea}\binits{A.}},
  \bauthor{\bsnm{Sarlos},~\bfnm{Tamas}\binits{T.}},
  \bauthor{\bsnm{Hawkins},~\bfnm{Peter}\binits{P.}},
  \bauthor{\bsnm{Davis},~\bfnm{Jared}\binits{J.}},
  \bauthor{\bsnm{Mohiuddin},~\bfnm{Afroz}\binits{A.}},
  \bauthor{\bsnm{Kaiser},~\bfnm{Lukasz}\binits{L.}} \betal{et~al.}
(\byear{2020}).
\btitle{Rethinking attention with performers}.
\bjournal{arXiv preprint arXiv:2009.14794}.
\end{barticle}
\endbibitem

\bibitem{christiano2017deep}
\begin{barticle}[author]
\bauthor{\bsnm{Christiano},~\bfnm{Paul~F}\binits{P.~F.}},
  \bauthor{\bsnm{Leike},~\bfnm{Jan}\binits{J.}},
  \bauthor{\bsnm{Brown},~\bfnm{Tom}\binits{T.}},
  \bauthor{\bsnm{Martic},~\bfnm{Miljan}\binits{M.}},
  \bauthor{\bsnm{Legg},~\bfnm{Shane}\binits{S.}} \AND
  \bauthor{\bsnm{Amodei},~\bfnm{Dario}\binits{D.}}
(\byear{2017}).
\btitle{Deep reinforcement learning from human preferences}.
\bjournal{Advances in neural information processing systems}
\bvolume{30}.
\end{barticle}
\endbibitem

\bibitem{d2024we}
\begin{barticle}[author]
\bauthor{\bsnm{D'Angelo},~\bfnm{Francesco}\binits{F.}},
  \bauthor{\bsnm{Andriushchenko},~\bfnm{Maksym}\binits{M.}},
  \bauthor{\bsnm{Varre},~\bfnm{Aditya~Vardhan}\binits{A.~V.}} \AND
  \bauthor{\bsnm{Flammarion},~\bfnm{Nicolas}\binits{N.}}
(\byear{2024}).
\btitle{Why do we need weight decay in modern deep learning?}
\bjournal{Advances in Neural Information Processing Systems}
\bvolume{37}
\bpages{23191--23223}.
\end{barticle}
\endbibitem

\bibitem{dao2022flashattention}
\begin{barticle}[author]
\bauthor{\bsnm{Dao},~\bfnm{Tri}\binits{T.}},
  \bauthor{\bsnm{Fu},~\bfnm{Dan}\binits{D.}},
  \bauthor{\bsnm{Ermon},~\bfnm{Stefano}\binits{S.}},
  \bauthor{\bsnm{Rudra},~\bfnm{Atri}\binits{A.}} \AND
  \bauthor{\bsnm{R{\'e}},~\bfnm{Christopher}\binits{C.}}
(\byear{2022}).
\btitle{FlashAttention: Fast and memory-efficient exact attention with
  {IO}-awareness}.
\bjournal{Advances in neural information processing systems}
\bvolume{35}
\bpages{16344--16359}.
\end{barticle}
\endbibitem

\bibitem{devlin2019bert}
\begin{binproceedings}[author]
\bauthor{\bsnm{Devlin},~\bfnm{Jacob}\binits{J.}},
  \bauthor{\bsnm{Chang},~\bfnm{Ming-Wei}\binits{M.-W.}},
  \bauthor{\bsnm{Lee},~\bfnm{Kenton}\binits{K.}} \AND
  \bauthor{\bsnm{Toutanova},~\bfnm{Kristina}\binits{K.}}
(\byear{2019}).
\btitle{Bert: Pre-training of deep bidirectional transformers for language
  understanding}.
In \bbooktitle{Proceedings of the 2019 conference of the North American chapter
  of the association for computational linguistics: human language
  technologies, volume 1 (long and short papers)}
\bpages{4171--4186}.
\end{binproceedings}
\endbibitem

\bibitem{ding2024fewer}
\begin{barticle}[author]
\bauthor{\bsnm{Ding},~\bfnm{Hantian}\binits{H.}},
  \bauthor{\bsnm{Wang},~\bfnm{Zijian}\binits{Z.}},
  \bauthor{\bsnm{Paolini},~\bfnm{Giovanni}\binits{G.}},
  \bauthor{\bsnm{Kumar},~\bfnm{Varun}\binits{V.}},
  \bauthor{\bsnm{Deoras},~\bfnm{Anoop}\binits{A.}},
  \bauthor{\bsnm{Roth},~\bfnm{Dan}\binits{D.}} \AND
  \bauthor{\bsnm{Soatto},~\bfnm{Stefano}\binits{S.}}
(\byear{2024}).
\btitle{Fewer truncations improve language modeling}.
\bjournal{arXiv preprint arXiv:2404.10830}.
\end{barticle}
\endbibitem

\bibitem{donoho2024data}
\begin{barticle}[author]
\bauthor{\bsnm{Donoho},~\bfnm{David}\binits{D.}}
(\byear{2024}).
\btitle{Data science at the singularity}.
\bjournal{Harvard Data Science Review}
\bvolume{6}.
\end{barticle}
\endbibitem

\bibitem{dosovitskiy2020image}
\begin{binproceedings}[author]
\bauthor{\bsnm{Dosovitskiy},~\bfnm{Alexey}\binits{A.}},
  \bauthor{\bsnm{Beyer},~\bfnm{Lucas}\binits{L.}},
  \bauthor{\bsnm{Kolesnikov},~\bfnm{Alexander}\binits{A.}},
  \bauthor{\bsnm{Weissenborn},~\bfnm{Dirk}\binits{D.}},
  \bauthor{\bsnm{Zhai},~\bfnm{Xiaohua}\binits{X.}},
  \bauthor{\bsnm{Unterthiner},~\bfnm{Thomas}\binits{T.}},
  \bauthor{\bsnm{Dehghani},~\bfnm{Mostafa}\binits{M.}},
  \bauthor{\bsnm{Minderer},~\bfnm{Matthias}\binits{M.}},
  \bauthor{\bsnm{Heigold},~\bfnm{G}\binits{G.}},
  \bauthor{\bsnm{Gelly},~\bfnm{S}\binits{S.}} \betal{et~al.}
(\byear{2020}).
\btitle{An Image is Worth 16x16 Words: Transformers for Image Recognition at
  Scale}.
In \bbooktitle{International Conference on Learning Representations}.
\end{binproceedings}
\endbibitem

\bibitem{engl1996regularization}
\begin{barticle}[author]
\bauthor{\bsnm{Engl},~\bfnm{HW}\binits{H.}},
  \bauthor{\bsnm{Hanke},~\bfnm{M}\binits{M.}} \AND
  \bauthor{\bsnm{Neubauer},~\bfnm{A}\binits{A.}}
(\byear{1996}).
\btitle{Regularization of Inverse Problems}.
\bjournal{Kluwer Academic Publishers; Dodrecht, Boston, London}.
\end{barticle}
\endbibitem

\bibitem{freitag2017beam}
\begin{binproceedings}[author]
\bauthor{\bsnm{Freitag},~\bfnm{Markus}\binits{M.}} \AND
  \bauthor{\bsnm{Al-Onaizan},~\bfnm{Yaser}\binits{Y.}}
(\byear{2017}).
\btitle{Beam Search Strategies for Neural Machine Translation}.
In \bbooktitle{Proceedings of the First Workshop on Neural Machine Translation}
\bpages{56--60}.
\end{binproceedings}
\endbibitem

\bibitem{frieder2023large}
\begin{barticle}[author]
\bauthor{\bsnm{Frieder},~\bfnm{Simon}\binits{S.}},
  \bauthor{\bsnm{Berner},~\bfnm{Julius}\binits{J.}},
  \bauthor{\bsnm{Petersen},~\bfnm{Philipp}\binits{P.}} \AND
  \bauthor{\bsnm{Lukasiewicz},~\bfnm{Thomas}\binits{T.}}
(\byear{2023}).
\btitle{Large language models for mathematicians}.
\bjournal{arXiv preprint arXiv:2312.04556}.
\end{barticle}
\endbibitem

\bibitem{gage1994new}
\begin{barticle}[author]
\bauthor{\bsnm{Gage},~\bfnm{Philip}\binits{P.}}
(\byear{1994}).
\btitle{A new algorithm for data compression}.
\bjournal{The C Users Journal}
\bvolume{12}
\bpages{23--38}.
\end{barticle}
\endbibitem

\bibitem{gerstgrasser2024model}
\begin{barticle}[author]
\bauthor{\bsnm{Gerstgrasser},~\bfnm{Matthias}\binits{M.}},
  \bauthor{\bsnm{Schaeffer},~\bfnm{Rylan}\binits{R.}},
  \bauthor{\bsnm{Dey},~\bfnm{Apratim}\binits{A.}},
  \bauthor{\bsnm{Rafailov},~\bfnm{Rafael}\binits{R.}},
  \bauthor{\bsnm{Sleight},~\bfnm{Henry}\binits{H.}},
  \bauthor{\bsnm{Hughes},~\bfnm{John}\binits{J.}},
  \bauthor{\bsnm{Korbak},~\bfnm{Tomasz}\binits{T.}},
  \bauthor{\bsnm{Agrawal},~\bfnm{Rajashree}\binits{R.}},
  \bauthor{\bsnm{Pai},~\bfnm{Dhruv}\binits{D.}},
  \bauthor{\bsnm{Gromov},~\bfnm{Andrey}\binits{A.}} \betal{et~al.}
(\byear{2024}).
\btitle{Is model collapse inevitable? breaking the curse of recursion by
  accumulating real and synthetic data}.
\bjournal{arXiv preprint arXiv:2404.01413}.
\end{barticle}
\endbibitem

\bibitem{geshkovski2024dynamic}
\begin{barticle}[author]
\bauthor{\bsnm{Geshkovski},~\bfnm{Borjan}\binits{B.}},
  \bauthor{\bsnm{Koubbi},~\bfnm{Hugo}\binits{H.}},
  \bauthor{\bsnm{Polyanskiy},~\bfnm{Yury}\binits{Y.}} \AND
  \bauthor{\bsnm{Rigollet},~\bfnm{Philippe}\binits{P.}}
(\byear{2024}).
\btitle{Dynamic metastability in the self-attention model}.
\bjournal{arXiv preprint arXiv:2410.06833}.
\end{barticle}
\endbibitem

\bibitem{geshkovski2023emergence}
\begin{barticle}[author]
\bauthor{\bsnm{Geshkovski},~\bfnm{Borjan}\binits{B.}},
  \bauthor{\bsnm{Letrouit},~\bfnm{Cyril}\binits{C.}},
  \bauthor{\bsnm{Polyanskiy},~\bfnm{Yury}\binits{Y.}} \AND
  \bauthor{\bsnm{Rigollet},~\bfnm{Philippe}\binits{P.}}
(\byear{2023}).
\btitle{The emergence of clusters in self-attention dynamics}.
\bjournal{Advances in Neural Information Processing Systems}
\bvolume{36}
\bpages{57026--57037}.
\end{barticle}
\endbibitem

\bibitem{geshkovski2025mathematical}
\begin{barticle}[author]
\bauthor{\bsnm{Geshkovski},~\bfnm{Borjan}\binits{B.}},
  \bauthor{\bsnm{Letrouit},~\bfnm{Cyril}\binits{C.}},
  \bauthor{\bsnm{Polyanskiy},~\bfnm{Yury}\binits{Y.}} \AND
  \bauthor{\bsnm{Rigollet},~\bfnm{Philippe}\binits{P.}}
(\byear{2025}).
\btitle{A mathematical perspective on transformers}.
\bjournal{Bulletin of the American Mathematical Society}
\bvolume{62}
\bpages{427--479}.
\end{barticle}
\endbibitem

\bibitem{geshkovski2024measure}
\begin{barticle}[author]
\bauthor{\bsnm{Geshkovski},~\bfnm{Borjan}\binits{B.}},
  \bauthor{\bsnm{Rigollet},~\bfnm{Philippe}\binits{P.}} \AND
  \bauthor{\bsnm{Ruiz-Balet},~\bfnm{Dom{\`e}nec}\binits{D.}}
(\byear{2024}).
\btitle{Measure-to-measure interpolation using Transformers}.
\bjournal{arXiv preprint arXiv:2411.04551}.
\end{barticle}
\endbibitem

\bibitem{gillespie1977exact}
\begin{barticle}[author]
\bauthor{\bsnm{Gillespie},~\bfnm{Daniel~T}\binits{D.~T.}}
(\byear{1977}).
\btitle{Exact stochastic simulation of coupled chemical reactions}.
\bjournal{The journal of physical chemistry}
\bvolume{81}
\bpages{2340--2361}.
\end{barticle}
\endbibitem

\bibitem{goodfellow2016deep}
\begin{bbook}[author]
\bauthor{\bsnm{Goodfellow},~\bfnm{Ian}\binits{I.}},
  \bauthor{\bsnm{Bengio},~\bfnm{Yoshua}\binits{Y.}},
  \bauthor{\bsnm{Courville},~\bfnm{Aaron}\binits{A.}} \AND
  \bauthor{\bsnm{Bengio},~\bfnm{Yoshua}\binits{Y.}}
(\byear{2016}).
\btitle{Deep Learning}
\bvolume{1}.
\bpublisher{MIT press Cambridge}.
\end{bbook}
\endbibitem

\bibitem{grattafiori2024llama}
\begin{barticle}[author]
\bauthor{\bsnm{Grattafiori},~\bfnm{Aaron}\binits{A.}},
  \bauthor{\bsnm{Dubey},~\bfnm{Abhimanyu}\binits{A.}},
  \bauthor{\bsnm{Jauhri},~\bfnm{Abhinav}\binits{A.}},
  \bauthor{\bsnm{Pandey},~\bfnm{Abhinav}\binits{A.}},
  \bauthor{\bsnm{Kadian},~\bfnm{Abhishek}\binits{A.}},
  \bauthor{\bsnm{Al-Dahle},~\bfnm{Ahmad}\binits{A.}},
  \bauthor{\bsnm{Letman},~\bfnm{Aiesha}\binits{A.}},
  \bauthor{\bsnm{Mathur},~\bfnm{Akhil}\binits{A.}},
  \bauthor{\bsnm{Schelten},~\bfnm{Alan}\binits{A.}},
  \bauthor{\bsnm{Vaughan},~\bfnm{Alex}\binits{A.}} \betal{et~al.}
(\byear{2024}).
\btitle{The {Llama} 3 herd of models}.
\bjournal{arXiv preprint arXiv:2407.21783}.
\end{barticle}
\endbibitem

\bibitem{gu2023mamba}
\begin{barticle}[author]
\bauthor{\bsnm{Gu},~\bfnm{Albert}\binits{A.}} \AND
  \bauthor{\bsnm{Dao},~\bfnm{Tri}\binits{T.}}
(\byear{2023}).
\btitle{Mamba: Linear-time sequence modeling with selective state spaces}.
\bjournal{arXiv preprint arXiv:2312.00752}.
\end{barticle}
\endbibitem

\bibitem{guu2020retrieval}
\begin{binproceedings}[author]
\bauthor{\bsnm{Guu},~\bfnm{Kelvin}\binits{K.}},
  \bauthor{\bsnm{Lee},~\bfnm{Kenton}\binits{K.}},
  \bauthor{\bsnm{Tung},~\bfnm{Zora}\binits{Z.}},
  \bauthor{\bsnm{Pasupat},~\bfnm{Panupong}\binits{P.}} \AND
  \bauthor{\bsnm{Chang},~\bfnm{Mingwei}\binits{M.}}
(\byear{2020}).
\btitle{Retrieval augmented language model pre-training}.
In \bbooktitle{International conference on machine learning}
\bpages{3929--3938}.
\bpublisher{PMLR}.
\end{binproceedings}
\endbibitem

\bibitem{hendrycks2020measuring}
\begin{barticle}[author]
\bauthor{\bsnm{Hendrycks},~\bfnm{Dan}\binits{D.}},
  \bauthor{\bsnm{Burns},~\bfnm{Collin}\binits{C.}},
  \bauthor{\bsnm{Basart},~\bfnm{Steven}\binits{S.}},
  \bauthor{\bsnm{Zou},~\bfnm{Andy}\binits{A.}},
  \bauthor{\bsnm{Mazeika},~\bfnm{Mantas}\binits{M.}},
  \bauthor{\bsnm{Song},~\bfnm{Dawn}\binits{D.}} \AND
  \bauthor{\bsnm{Steinhardt},~\bfnm{Jacob}\binits{J.}}
(\byear{2020}).
\btitle{Measuring massive multitask language understanding}.
\bjournal{arXiv preprint arXiv:2009.03300}.
\end{barticle}
\endbibitem

\bibitem{hoffmann2022training}
\begin{binproceedings}[author]
\bauthor{\bsnm{Hoffmann},~\bfnm{Jordan}\binits{J.}},
  \bauthor{\bsnm{Borgeaud},~\bfnm{Sebastian}\binits{S.}},
  \bauthor{\bsnm{Mensch},~\bfnm{Arthur}\binits{A.}},
  \bauthor{\bsnm{Buchatskaya},~\bfnm{Elena}\binits{E.}},
  \bauthor{\bsnm{Cai},~\bfnm{Trevor}\binits{T.}},
  \bauthor{\bsnm{Rutherford},~\bfnm{Eliza}\binits{E.}}, \bauthor{\bparticle{de}
  \bsnm{Las~Casas},~\bfnm{Diego}\binits{D.}},
  \bauthor{\bsnm{Hendricks},~\bfnm{Lisa~Anne}\binits{L.~A.}},
  \bauthor{\bsnm{Welbl},~\bfnm{Johannes}\binits{J.}},
  \bauthor{\bsnm{Clark},~\bfnm{Aidan}\binits{A.}} \betal{et~al.}
(\byear{2022}).
\btitle{Training compute-optimal large language models}.
In \bbooktitle{Proceedings of the 36th International Conference on Neural
  Information Processing Systems}
\bpages{30016--30030}.
\end{binproceedings}
\endbibitem

\bibitem{holtzmancurious}
\begin{binproceedings}[author]
\bauthor{\bsnm{Holtzman},~\bfnm{Ari}\binits{A.}},
  \bauthor{\bsnm{Buys},~\bfnm{Jan}\binits{J.}},
  \bauthor{\bsnm{Du},~\bfnm{Li}\binits{L.}},
  \bauthor{\bsnm{Forbes},~\bfnm{Maxwell}\binits{M.}} \AND
  \bauthor{\bsnm{Choi},~\bfnm{Yejin}\binits{Y.}}
(\byear{2020}).
\btitle{The Curious Case of Neural Text Degeneration}.
In \bbooktitle{International Conference on Learning Representations}.
\end{binproceedings}
\endbibitem

\bibitem{hoogeboom2021argmax}
\begin{barticle}[author]
\bauthor{\bsnm{Hoogeboom},~\bfnm{Emiel}\binits{E.}},
  \bauthor{\bsnm{Nielsen},~\bfnm{Didrik}\binits{D.}},
  \bauthor{\bsnm{Jaini},~\bfnm{Priyank}\binits{P.}},
  \bauthor{\bsnm{Forr{\'e}},~\bfnm{Patrick}\binits{P.}} \AND
  \bauthor{\bsnm{Welling},~\bfnm{Max}\binits{M.}}
(\byear{2021}).
\btitle{Argmax flows and multinomial diffusion: Learning categorical
  distributions}.
\bjournal{Advances in neural information processing systems}
\bvolume{34}
\bpages{12454--12465}.
\end{barticle}
\endbibitem

\bibitem{hu2022lora}
\begin{binproceedings}[author]
\bauthor{\bsnm{Hu},~\bfnm{Edward~J}\binits{E.~J.}},
  \bauthor{\bsnm{Shen},~\bfnm{Yelong}\binits{Y.}},
  \bauthor{\bsnm{Wallis},~\bfnm{Phillip}\binits{P.}},
  \bauthor{\bsnm{Allen-Zhu},~\bfnm{Zeyuan}\binits{Z.}},
  \bauthor{\bsnm{Li},~\bfnm{Yuanzhi}\binits{Y.}},
  \bauthor{\bsnm{Wang},~\bfnm{Shean}\binits{S.}},
  \bauthor{\bsnm{Wang},~\bfnm{Lu}\binits{L.}},
  \bauthor{\bsnm{Chen},~\bfnm{Weizhu}\binits{W.}} \betal{et~al.}
(\byear{2022}).
\btitle{{LoRA}: Low-rank adaptation of large language models.}
In \bbooktitle{International Conference on Learning Representations}
\bvolume{1}
\bpages{3}.
\end{binproceedings}
\endbibitem

\bibitem{hu2025m}
\begin{binproceedings}[author]
\bauthor{\bsnm{Hu},~\bfnm{Kai}\binits{K.}},
  \bauthor{\bsnm{Gao},~\bfnm{Feng}\binits{F.}},
  \bauthor{\bsnm{Nie},~\bfnm{Xiaohan}\binits{X.}},
  \bauthor{\bsnm{Zhou},~\bfnm{Peng}\binits{P.}},
  \bauthor{\bsnm{Tran},~\bfnm{Son}\binits{S.}},
  \bauthor{\bsnm{Neiman},~\bfnm{Tal}\binits{T.}},
  \bauthor{\bsnm{Wang},~\bfnm{Lingyun}\binits{L.}},
  \bauthor{\bsnm{Shah},~\bfnm{Mubarak}\binits{M.}},
  \bauthor{\bsnm{Hamid},~\bfnm{Raffay}\binits{R.}},
  \bauthor{\bsnm{Yin},~\bfnm{Bing}\binits{B.}} \betal{et~al.}
(\byear{2025}).
\btitle{{M-LLM} based video frame selection for efficient video understanding}.
In \bbooktitle{Proceedings of the Computer Vision and Pattern Recognition
  Conference}
\bpages{13702--13712}.
\end{binproceedings}
\endbibitem

\bibitem{hui2024Qwen2}
\begin{barticle}[author]
\bauthor{\bsnm{Hui},~\bfnm{Binyuan}\binits{B.}},
  \bauthor{\bsnm{Yang},~\bfnm{Jian}\binits{J.}},
  \bauthor{\bsnm{Cui},~\bfnm{Zeyu}\binits{Z.}},
  \bauthor{\bsnm{Yang},~\bfnm{Jiaxi}\binits{J.}},
  \bauthor{\bsnm{Liu},~\bfnm{Dayiheng}\binits{D.}},
  \bauthor{\bsnm{Zhang},~\bfnm{Lei}\binits{L.}},
  \bauthor{\bsnm{Liu},~\bfnm{Tianyu}\binits{T.}},
  \bauthor{\bsnm{Zhang},~\bfnm{Jiajun}\binits{J.}},
  \bauthor{\bsnm{Yu},~\bfnm{Bowen}\binits{B.}},
  \bauthor{\bsnm{Lu},~\bfnm{Keming}\binits{K.}} \betal{et~al.}
(\byear{2024}).
\btitle{Qwen2. 5-coder technical report}.
\bjournal{arXiv preprint arXiv:2409.12186}.
\end{barticle}
\endbibitem

\bibitem{cohere_tokens_2025}
\begin{bmisc}[author]
\bauthor{\bsnm{{Cohere Inc. }}}
(\byear{2025}).
\btitle{A Guide to Tokens and Tokenizers}.
\bhowpublished{\url{https://docs.cohere.com/v2/docs/tokens-and-tokenizers}}.
\bnote{Accessed: 2025-09-23}.
\end{bmisc}
\endbibitem

\bibitem{jiang2025mme}
\begin{barticle}[author]
\bauthor{\bsnm{Jiang},~\bfnm{Dongzhi}\binits{D.}},
  \bauthor{\bsnm{Zhang},~\bfnm{Renrui}\binits{R.}},
  \bauthor{\bsnm{Guo},~\bfnm{Ziyu}\binits{Z.}},
  \bauthor{\bsnm{Li},~\bfnm{Yanwei}\binits{Y.}},
  \bauthor{\bsnm{Qi},~\bfnm{Yu}\binits{Y.}},
  \bauthor{\bsnm{Chen},~\bfnm{Xinyan}\binits{X.}},
  \bauthor{\bsnm{Wang},~\bfnm{Liuhui}\binits{L.}},
  \bauthor{\bsnm{Jin},~\bfnm{Jianhan}\binits{J.}},
  \bauthor{\bsnm{Guo},~\bfnm{Claire}\binits{C.}},
  \bauthor{\bsnm{Yan},~\bfnm{Shen}\binits{S.}} \betal{et~al.}
(\byear{2025}).
\btitle{{MME-CoT}: Benchmarking chain-of-thought in large multimodal models for
  reasoning quality, robustness, and efficiency}.
\bjournal{arXiv preprint arXiv:2502.09621}.
\end{barticle}
\endbibitem

\bibitem{jumper2021highly}
\begin{barticle}[author]
\bauthor{\bsnm{Jumper},~\bfnm{John}\binits{J.}},
  \bauthor{\bsnm{Evans},~\bfnm{Richard}\binits{R.}},
  \bauthor{\bsnm{Pritzel},~\bfnm{Alexander}\binits{A.}},
  \bauthor{\bsnm{Green},~\bfnm{Tim}\binits{T.}},
  \bauthor{\bsnm{Figurnov},~\bfnm{Michael}\binits{M.}},
  \bauthor{\bsnm{Ronneberger},~\bfnm{Olaf}\binits{O.}},
  \bauthor{\bsnm{Tunyasuvunakool},~\bfnm{Kathryn}\binits{K.}},
  \bauthor{\bsnm{Bates},~\bfnm{Russ}\binits{R.}},
  \bauthor{\bsnm{{\v{Z}}{\'\i}dek},~\bfnm{Augustin}\binits{A.}},
  \bauthor{\bsnm{Potapenko},~\bfnm{Anna}\binits{A.}} \betal{et~al.}
(\byear{2021}).
\btitle{Highly accurate protein structure prediction with AlphaFold}.
\bjournal{nature}
\bvolume{596}
\bpages{583--589}.
\end{barticle}
\endbibitem

\bibitem{kaplan2020scaling}
\begin{barticle}[author]
\bauthor{\bsnm{Kaplan},~\bfnm{Jared}\binits{J.}},
  \bauthor{\bsnm{McCandlish},~\bfnm{Sam}\binits{S.}},
  \bauthor{\bsnm{Henighan},~\bfnm{Tom}\binits{T.}},
  \bauthor{\bsnm{Brown},~\bfnm{Tom~B}\binits{T.~B.}},
  \bauthor{\bsnm{Chess},~\bfnm{Benjamin}\binits{B.}},
  \bauthor{\bsnm{Child},~\bfnm{Rewon}\binits{R.}},
  \bauthor{\bsnm{Gray},~\bfnm{Scott}\binits{S.}},
  \bauthor{\bsnm{Radford},~\bfnm{Alec}\binits{A.}},
  \bauthor{\bsnm{Wu},~\bfnm{Jeffrey}\binits{J.}} \AND
  \bauthor{\bsnm{Amodei},~\bfnm{Dario}\binits{D.}}
(\byear{2020}).
\btitle{Scaling laws for neural language models}.
\bjournal{arXiv preprint arXiv:2001.08361}.
\end{barticle}
\endbibitem

\bibitem{karagodin2024clustering}
\begin{barticle}[author]
\bauthor{\bsnm{Karagodin},~\bfnm{Nikita}\binits{N.}},
  \bauthor{\bsnm{Polyanskiy},~\bfnm{Yury}\binits{Y.}} \AND
  \bauthor{\bsnm{Rigollet},~\bfnm{Philippe}\binits{P.}}
(\byear{2024}).
\btitle{Clustering in causal attention masking}.
\bjournal{Advances in Neural Information Processing Systems}
\bvolume{37}
\bpages{115652--115681}.
\end{barticle}
\endbibitem

\bibitem{kirchenbauer2023watermark}
\begin{binproceedings}[author]
\bauthor{\bsnm{Kirchenbauer},~\bfnm{John}\binits{J.}},
  \bauthor{\bsnm{Geiping},~\bfnm{Jonas}\binits{J.}},
  \bauthor{\bsnm{Wen},~\bfnm{Yuxin}\binits{Y.}},
  \bauthor{\bsnm{Katz},~\bfnm{Jonathan}\binits{J.}},
  \bauthor{\bsnm{Miers},~\bfnm{Ian}\binits{I.}} \AND
  \bauthor{\bsnm{Goldstein},~\bfnm{Tom}\binits{T.}}
(\byear{2023}).
\btitle{A watermark for large language models}.
In \bbooktitle{International Conference on Machine Learning}
\bpages{17061--17084}.
\bpublisher{PMLR}.
\end{binproceedings}
\endbibitem

\bibitem{kobayashi2024weight}
\begin{barticle}[author]
\bauthor{\bsnm{Kobayashi},~\bfnm{Seijin}\binits{S.}},
  \bauthor{\bsnm{Akram},~\bfnm{Yassir}\binits{Y.}} \AND
  \bauthor{\bsnm{Von~Oswald},~\bfnm{Johannes}\binits{J.}}
(\byear{2024}).
\btitle{Weight decay induces low-rank attention layers}.
\bjournal{Advances in Neural Information Processing Systems}
\bvolume{37}
\bpages{4481--4510}.
\end{barticle}
\endbibitem

\bibitem{lahoti2026mamba}
\begin{binproceedings}[author]
\bauthor{\bsnm{Lahoti},~\bfnm{Aakash~Sunil}\binits{A.~S.}},
  \bauthor{\bsnm{Li},~\bfnm{Kevin}\binits{K.}},
  \bauthor{\bsnm{Chen},~\bfnm{Berlin}\binits{B.}},
  \bauthor{\bsnm{Wang},~\bfnm{Caitlin}\binits{C.}},
  \bauthor{\bsnm{Bick},~\bfnm{Aviv}\binits{A.}},
  \bauthor{\bsnm{Kolter},~\bfnm{Zico}\binits{Z.}},
  \bauthor{\bsnm{Dao},~\bfnm{Tri}\binits{T.}} \AND
  \bauthor{\bsnm{Gu},~\bfnm{Albert}\binits{A.}}
(\byear{2026}).
\btitle{Mamba-3: Improved sequence modeling using state space principles}.
In \bbooktitle{International Conference on Learning Representations}
\bvolume{2026}
\bpages{86173--86202}.
\end{binproceedings}
\endbibitem

\bibitem{lecun2002gradient}
\begin{barticle}[author]
\bauthor{\bsnm{LeCun},~\bfnm{Yann}\binits{Y.}},
  \bauthor{\bsnm{Bottou},~\bfnm{L{\'e}on}\binits{L.}},
  \bauthor{\bsnm{Bengio},~\bfnm{Yoshua}\binits{Y.}} \AND
  \bauthor{\bsnm{Haffner},~\bfnm{Patrick}\binits{P.}}
(\byear{2002}).
\btitle{Gradient-based learning applied to document recognition}.
\bjournal{Proceedings of the IEEE}
\bvolume{86}
\bpages{2278--2324}.
\end{barticle}
\endbibitem

\bibitem{DBLP:journals/corr/abs-2502-12289}
\begin{barticle}[author]
\bauthor{\bsnm{Lee},~\bfnm{Jinu}\binits{J.}} \AND
  \bauthor{\bsnm{Hockenmaier},~\bfnm{Julia}\binits{J.}}
(\byear{2025}).
\btitle{Evaluating Step-by-step Reasoning Traces: A Survey}.
\bjournal{CoRR}
\bvolume{abs/2502.12289}.
\end{barticle}
\endbibitem

\bibitem{lewis2020retrieval}
\begin{barticle}[author]
\bauthor{\bsnm{Lewis},~\bfnm{Patrick}\binits{P.}},
  \bauthor{\bsnm{Perez},~\bfnm{Ethan}\binits{E.}},
  \bauthor{\bsnm{Piktus},~\bfnm{Aleksandra}\binits{A.}},
  \bauthor{\bsnm{Petroni},~\bfnm{Fabio}\binits{F.}},
  \bauthor{\bsnm{Karpukhin},~\bfnm{Vladimir}\binits{V.}},
  \bauthor{\bsnm{Goyal},~\bfnm{Naman}\binits{N.}},
  \bauthor{\bsnm{K{\"u}ttler},~\bfnm{Heinrich}\binits{H.}},
  \bauthor{\bsnm{Lewis},~\bfnm{Mike}\binits{M.}},
  \bauthor{\bsnm{Yih},~\bfnm{Wen-tau}\binits{W.-t.}},
  \bauthor{\bsnm{Rockt{\"a}schel},~\bfnm{Tim}\binits{T.}} \betal{et~al.}
(\byear{2020}).
\btitle{Retrieval-augmented generation for knowledge-intensive {NLP} tasks}.
\bjournal{Advances in neural information processing systems}
\bvolume{33}
\bpages{9459--9474}.
\end{barticle}
\endbibitem

\bibitem{lou2024discrete}
\begin{binproceedings}[author]
\bauthor{\bsnm{Lou},~\bfnm{Aaron}\binits{A.}},
  \bauthor{\bsnm{Meng},~\bfnm{Chenlin}\binits{C.}} \AND
  \bauthor{\bsnm{Ermon},~\bfnm{Stefano}\binits{S.}}
(\byear{2024}).
\btitle{Discrete diffusion modeling by estimating the ratios of the data
  distribution}.
In \bbooktitle{Proceedings of the 41st International Conference on Machine
  Learning}
\bpages{32819--32848}.
\end{binproceedings}
\endbibitem

\bibitem{mizrahi20234m}
\begin{barticle}[author]
\bauthor{\bsnm{Mizrahi},~\bfnm{David}\binits{D.}},
  \bauthor{\bsnm{Bachmann},~\bfnm{Roman}\binits{R.}},
  \bauthor{\bsnm{Kar},~\bfnm{Oguzhan}\binits{O.}},
  \bauthor{\bsnm{Yeo},~\bfnm{Teresa}\binits{T.}},
  \bauthor{\bsnm{Gao},~\bfnm{Mingfei}\binits{M.}},
  \bauthor{\bsnm{Dehghan},~\bfnm{Afshin}\binits{A.}} \AND
  \bauthor{\bsnm{Zamir},~\bfnm{Amir}\binits{A.}}
(\byear{2023}).
\btitle{{4M}: Massively multimodal masked modeling}.
\bjournal{Advances in Neural Information Processing Systems}
\bvolume{36}
\bpages{58363--58408}.
\end{barticle}
\endbibitem

\bibitem{mondorf2024beyond}
\begin{binproceedings}[author]
\bauthor{\bsnm{Mondorf},~\bfnm{Philipp}\binits{P.}} \AND
  \bauthor{\bsnm{Plank},~\bfnm{Barbara}\binits{B.}}
(\byear{2024}).
\btitle{Beyond Accuracy: Evaluating the Reasoning Behavior of Large Language
  Models - A Survey}.
In \bbooktitle{First Conference on Language Modeling}.
\end{binproceedings}
\endbibitem

\bibitem{mu2023learning}
\begin{barticle}[author]
\bauthor{\bsnm{Mu},~\bfnm{Jesse}\binits{J.}},
  \bauthor{\bsnm{Li},~\bfnm{Xiang}\binits{X.}} \AND
  \bauthor{\bsnm{Goodman},~\bfnm{Noah}\binits{N.}}
(\byear{2023}).
\btitle{Learning to compress prompts with gist tokens}.
\bjournal{Advances in Neural Information Processing Systems}
\bvolume{36}
\bpages{19327--19352}.
\end{barticle}
\endbibitem

\bibitem{norris1998markov}
\begin{bbook}[author]
\bauthor{\bsnm{Norris},~\bfnm{James~R}\binits{J.~R.}}
(\byear{1998}).
\btitle{Markov chains}
\bvolume{2}.
\bpublisher{Cambridge university press}.
\end{bbook}
\endbibitem

\bibitem{ouyang2022training}
\begin{barticle}[author]
\bauthor{\bsnm{Ouyang},~\bfnm{Long}\binits{L.}},
  \bauthor{\bsnm{Wu},~\bfnm{Jeffrey}\binits{J.}},
  \bauthor{\bsnm{Jiang},~\bfnm{Xu}\binits{X.}},
  \bauthor{\bsnm{Almeida},~\bfnm{Diogo}\binits{D.}},
  \bauthor{\bsnm{Wainwright},~\bfnm{Carroll}\binits{C.}},
  \bauthor{\bsnm{Mishkin},~\bfnm{Pamela}\binits{P.}},
  \bauthor{\bsnm{Zhang},~\bfnm{Chong}\binits{C.}},
  \bauthor{\bsnm{Agarwal},~\bfnm{Sandhini}\binits{S.}},
  \bauthor{\bsnm{Slama},~\bfnm{Katarina}\binits{K.}},
  \bauthor{\bsnm{Ray},~\bfnm{Alex}\binits{A.}} \betal{et~al.}
(\byear{2022}).
\btitle{Training language models to follow instructions with human feedback}.
\bjournal{Advances in neural information processing systems}
\bvolume{35}
\bpages{27730--27744}.
\end{barticle}
\endbibitem

\bibitem{peng2023rwkv}
\begin{binproceedings}[author]
\bauthor{\bsnm{Peng},~\bfnm{Bo}\binits{B.}},
  \bauthor{\bsnm{Alcaide},~\bfnm{Eric}\binits{E.}},
  \bauthor{\bsnm{Anthony},~\bfnm{Quentin}\binits{Q.}},
  \bauthor{\bsnm{Albalak},~\bfnm{Alon}\binits{A.}},
  \bauthor{\bsnm{Arcadinho},~\bfnm{Samuel}\binits{S.}},
  \bauthor{\bsnm{Biderman},~\bfnm{Stella}\binits{S.}},
  \bauthor{\bsnm{Cao},~\bfnm{Huanqi}\binits{H.}},
  \bauthor{\bsnm{Cheng},~\bfnm{Xin}\binits{X.}},
  \bauthor{\bsnm{Chung},~\bfnm{Michael}\binits{M.}},
  \bauthor{\bsnm{Derczynski},~\bfnm{Leon}\binits{L.}} \betal{et~al.}
(\byear{2023}).
\btitle{Rwkv: Reinventing rnns for the transformer era}.
In \bbooktitle{Findings of the association for computational linguistics: EMNLP
  2023}
\bpages{14048--14077}.
\end{binproceedings}
\endbibitem

\bibitem{petrov2025long}
\begin{barticle}[author]
\bauthor{\bsnm{Petrov},~\bfnm{Aleksandar}\binits{A.}},
  \bauthor{\bsnm{Sandler},~\bfnm{Mark}\binits{M.}},
  \bauthor{\bsnm{Zhmoginov},~\bfnm{Andrey}\binits{A.}},
  \bauthor{\bsnm{Miller},~\bfnm{Nolan}\binits{N.}} \AND
  \bauthor{\bsnm{Vladymyrov},~\bfnm{Max}\binits{M.}}
(\byear{2025}).
\btitle{Long Context In-Context Compression by Getting to the Gist of Gisting}.
\bjournal{arXiv preprint arXiv:2504.08934}.
\end{barticle}
\endbibitem

\bibitem{phuong2022formal}
\begin{barticle}[author]
\bauthor{\bsnm{Phuong},~\bfnm{Mary}\binits{M.}} \AND
  \bauthor{\bsnm{Hutter},~\bfnm{Marcus}\binits{M.}}
(\byear{2022}).
\btitle{Formal algorithms for transformers}.
\bjournal{arXiv preprint arXiv:2207.09238}.
\end{barticle}
\endbibitem

\bibitem{radford2021learning}
\begin{binproceedings}[author]
\bauthor{\bsnm{Radford},~\bfnm{Alec}\binits{A.}},
  \bauthor{\bsnm{Kim},~\bfnm{Jong~Wook}\binits{J.~W.}},
  \bauthor{\bsnm{Hallacy},~\bfnm{Chris}\binits{C.}},
  \bauthor{\bsnm{Ramesh},~\bfnm{Aditya}\binits{A.}},
  \bauthor{\bsnm{Goh},~\bfnm{Gabriel}\binits{G.}},
  \bauthor{\bsnm{Agarwal},~\bfnm{Sandhini}\binits{S.}},
  \bauthor{\bsnm{Sastry},~\bfnm{Girish}\binits{G.}},
  \bauthor{\bsnm{Askell},~\bfnm{Amanda}\binits{A.}},
  \bauthor{\bsnm{Mishkin},~\bfnm{Pamela}\binits{P.}},
  \bauthor{\bsnm{Clark},~\bfnm{Jack}\binits{J.}} \betal{et~al.}
(\byear{2021}).
\btitle{Learning transferable visual models from natural language supervision}.
In \bbooktitle{International conference on machine learning}
\bpages{8748--8763}.
\bpublisher{PMLR}.
\end{binproceedings}
\endbibitem

\bibitem{radford2015unsupervised}
\begin{barticle}[author]
\bauthor{\bsnm{Radford},~\bfnm{Alec}\binits{A.}},
  \bauthor{\bsnm{Metz},~\bfnm{Luke}\binits{L.}} \AND
  \bauthor{\bsnm{Chintala},~\bfnm{Soumith}\binits{S.}}
(\byear{2015}).
\btitle{Unsupervised representation learning with deep convolutional generative
  adversarial networks}.
\bjournal{arXiv preprint arXiv:1511.06434}.
\end{barticle}
\endbibitem

\bibitem{radford2018improving}
\begin{barticle}[author]
\bauthor{\bsnm{Radford},~\bfnm{Alec}\binits{A.}},
  \bauthor{\bsnm{Narasimhan},~\bfnm{Karthik}\binits{K.}},
  \bauthor{\bsnm{Salimans},~\bfnm{Tim}\binits{T.}},
  \bauthor{\bsnm{Sutskever},~\bfnm{Ilya}\binits{I.}} \betal{et~al.}
(\byear{2018}).
\btitle{Improving language understanding by generative pre-training}.
\end{barticle}
\endbibitem

\bibitem{radford2019language}
\begin{barticle}[author]
\bauthor{\bsnm{Radford},~\bfnm{Alec}\binits{A.}},
  \bauthor{\bsnm{Wu},~\bfnm{Jeffrey}\binits{J.}},
  \bauthor{\bsnm{Child},~\bfnm{Rewon}\binits{R.}},
  \bauthor{\bsnm{Luan},~\bfnm{David}\binits{D.}},
  \bauthor{\bsnm{Amodei},~\bfnm{Dario}\binits{D.}},
  \bauthor{\bsnm{Sutskever},~\bfnm{Ilya}\binits{I.}} \betal{et~al.}
(\byear{2019}).
\btitle{Language models are unsupervised multitask learners}.
\bjournal{OpenAI blog}
\bvolume{1}
\bpages{9}.
\end{barticle}
\endbibitem

\bibitem{rafailov2023direct}
\begin{barticle}[author]
\bauthor{\bsnm{Rafailov},~\bfnm{Rafael}\binits{R.}},
  \bauthor{\bsnm{Sharma},~\bfnm{Archit}\binits{A.}},
  \bauthor{\bsnm{Mitchell},~\bfnm{Eric}\binits{E.}},
  \bauthor{\bsnm{Manning},~\bfnm{Christopher~D}\binits{C.~D.}},
  \bauthor{\bsnm{Ermon},~\bfnm{Stefano}\binits{S.}} \AND
  \bauthor{\bsnm{Finn},~\bfnm{Chelsea}\binits{C.}}
(\byear{2023}).
\btitle{Direct preference optimization: Your language model is secretly a
  reward model}.
\bjournal{Advances in neural information processing systems}
\bvolume{36}
\bpages{53728--53741}.
\end{barticle}
\endbibitem

\bibitem{raffel2020exploring}
\begin{barticle}[author]
\bauthor{\bsnm{Raffel},~\bfnm{Colin}\binits{C.}},
  \bauthor{\bsnm{Shazeer},~\bfnm{Noam}\binits{N.}},
  \bauthor{\bsnm{Roberts},~\bfnm{Adam}\binits{A.}},
  \bauthor{\bsnm{Lee},~\bfnm{Katherine}\binits{K.}},
  \bauthor{\bsnm{Narang},~\bfnm{Sharan}\binits{S.}},
  \bauthor{\bsnm{Matena},~\bfnm{Michael}\binits{M.}},
  \bauthor{\bsnm{Zhou},~\bfnm{Yanqi}\binits{Y.}},
  \bauthor{\bsnm{Li},~\bfnm{Wei}\binits{W.}} \AND
  \bauthor{\bsnm{Liu},~\bfnm{Peter~J}\binits{P.~J.}}
(\byear{2020}).
\btitle{Exploring the limits of transfer learning with a unified text-to-text
  transformer}.
\bjournal{Journal of Machine Learning Research}
\bvolume{21}
\bpages{1--67}.
\end{barticle}
\endbibitem

\bibitem{rajaraman2024toward}
\begin{barticle}[author]
\bauthor{\bsnm{Rajaraman},~\bfnm{Nived}\binits{N.}},
  \bauthor{\bsnm{Jiao},~\bfnm{Jiantao}\binits{J.}} \AND
  \bauthor{\bsnm{Ramchandran},~\bfnm{Kannan}\binits{K.}}
(\byear{2024}).
\btitle{Toward a theory of tokenization in {LLMs}}.
\bjournal{arXiv preprint arXiv:2404.08335}.
\end{barticle}
\endbibitem

\bibitem{razavi2019generating}
\begin{barticle}[author]
\bauthor{\bsnm{Razavi},~\bfnm{Ali}\binits{A.}}, \bauthor{\bparticle{Van~den}
  \bsnm{Oord},~\bfnm{Aaron}\binits{A.}} \AND
  \bauthor{\bsnm{Vinyals},~\bfnm{Oriol}\binits{O.}}
(\byear{2019}).
\btitle{Generating diverse high-fidelity images with {VQ}-{VAE}-2}.
\bjournal{Advances in Neural Information Processing Systems}
\bvolume{32}.
\end{barticle}
\endbibitem

\bibitem{robbins1951stochastic}
\begin{barticle}[author]
\bauthor{\bsnm{Robbins},~\bfnm{Herbert}\binits{H.}} \AND
  \bauthor{\bsnm{Monro},~\bfnm{Sutton}\binits{S.}}
(\byear{1951}).
\btitle{A stochastic approximation method}.
\bjournal{The annals of mathematical statistics}
\bpages{400--407}.
\end{barticle}
\endbibitem

\bibitem{sahoo2024simple}
\begin{barticle}[author]
\bauthor{\bsnm{Sahoo},~\bfnm{Subham}\binits{S.}},
  \bauthor{\bsnm{Arriola},~\bfnm{Marianne}\binits{M.}},
  \bauthor{\bsnm{Schiff},~\bfnm{Yair}\binits{Y.}},
  \bauthor{\bsnm{Gokaslan},~\bfnm{Aaron}\binits{A.}},
  \bauthor{\bsnm{Marroquin},~\bfnm{Edgar}\binits{E.}},
  \bauthor{\bsnm{Chiu},~\bfnm{Justin}\binits{J.}},
  \bauthor{\bsnm{Rush},~\bfnm{Alexander}\binits{A.}} \AND
  \bauthor{\bsnm{Kuleshov},~\bfnm{Volodymyr}\binits{V.}}
(\byear{2024}).
\btitle{Simple and effective masked diffusion language models}.
\bjournal{Advances in Neural Information Processing Systems}
\bvolume{37}
\bpages{130136--130184}.
\end{barticle}
\endbibitem

\bibitem{sennrich2015neural}
\begin{binproceedings}[author]
\bauthor{\bsnm{Sennrich},~\bfnm{Rico}\binits{R.}},
  \bauthor{\bsnm{Haddow},~\bfnm{Barry}\binits{B.}} \AND
  \bauthor{\bsnm{Birch},~\bfnm{Alexandra}\binits{A.}}
(\byear{2016}).
\btitle{Neural Machine Translation of Rare Words with Subword Units}.
In \bbooktitle{Proceedings of the 54th Annual Meeting of the Association for
  Computational Linguistics (Volume 1: Long Papers)}
(\beditor{\bfnm{Katrin}\binits{K.}~\bsnm{Erk}} \AND
  \beditor{\bfnm{Noah~A.}\binits{N.~A.}~\bsnm{Smith}}, eds.)
\bpages{1715--1725}.
\bpublisher{Association for Computational Linguistics}, \baddress{Berlin,
  Germany}.
\bdoi{10.18653/v1/P16-1162}
\end{binproceedings}
\endbibitem

\bibitem{shannon1948mathematical}
\begin{barticle}[author]
\bauthor{\bsnm{Shannon},~\bfnm{Claude~E}\binits{C.~E.}}
(\byear{1948}).
\btitle{A mathematical theory of communication}.
\bjournal{The Bell system technical journal}
\bvolume{27}
\bpages{379--423}.
\end{barticle}
\endbibitem

\bibitem{shao2024deepseekmath}
\begin{barticle}[author]
\bauthor{\bsnm{Shao},~\bfnm{Zhihong}\binits{Z.}},
  \bauthor{\bsnm{Wang},~\bfnm{Peiyi}\binits{P.}},
  \bauthor{\bsnm{Zhu},~\bfnm{Qihao}\binits{Q.}},
  \bauthor{\bsnm{Xu},~\bfnm{Runxin}\binits{R.}},
  \bauthor{\bsnm{Song},~\bfnm{Junxiao}\binits{J.}},
  \bauthor{\bsnm{Bi},~\bfnm{Xiao}\binits{X.}},
  \bauthor{\bsnm{Zhang},~\bfnm{Haowei}\binits{H.}},
  \bauthor{\bsnm{Zhang},~\bfnm{Mingchuan}\binits{M.}},
  \bauthor{\bsnm{Li},~\bfnm{YK}\binits{Y.}},
  \bauthor{\bsnm{Wu},~\bfnm{Yang}\binits{Y.}} \betal{et~al.}
(\byear{2024}).
\btitle{Deepseekmath: Pushing the limits of mathematical reasoning in open
  language models}.
\bjournal{arXiv preprint arXiv:2402.03300}.
\end{barticle}
\endbibitem

\bibitem{shi2024simplified}
\begin{barticle}[author]
\bauthor{\bsnm{Shi},~\bfnm{Jiaxin}\binits{J.}},
  \bauthor{\bsnm{Han},~\bfnm{Kehang}\binits{K.}},
  \bauthor{\bsnm{Wang},~\bfnm{Zhe}\binits{Z.}},
  \bauthor{\bsnm{Doucet},~\bfnm{Arnaud}\binits{A.}} \AND
  \bauthor{\bsnm{Titsias},~\bfnm{Michalis}\binits{M.}}
(\byear{2024}).
\btitle{Simplified and generalized masked diffusion for discrete data}.
\bjournal{Advances in neural information processing systems}
\bvolume{37}
\bpages{103131--103167}.
\end{barticle}
\endbibitem

\bibitem{shumailov2024ai}
\begin{barticle}[author]
\bauthor{\bsnm{Shumailov},~\bfnm{Ilia}\binits{I.}},
  \bauthor{\bsnm{Shumaylov},~\bfnm{Zakhar}\binits{Z.}},
  \bauthor{\bsnm{Zhao},~\bfnm{Yiren}\binits{Y.}},
  \bauthor{\bsnm{Papernot},~\bfnm{Nicolas}\binits{N.}},
  \bauthor{\bsnm{Anderson},~\bfnm{Ross}\binits{R.}} \AND
  \bauthor{\bsnm{Gal},~\bfnm{Yarin}\binits{Y.}}
(\byear{2024}).
\btitle{AI models collapse when trained on recursively generated data}.
\bjournal{Nature}
\bvolume{631}
\bpages{755--759}.
\end{barticle}
\endbibitem

\bibitem{song2020score}
\begin{binproceedings}[author]
\bauthor{\bsnm{Song},~\bfnm{Yang}\binits{Y.}},
  \bauthor{\bsnm{Sohl-Dickstein},~\bfnm{Jascha}\binits{J.}},
  \bauthor{\bsnm{Kingma},~\bfnm{Diederik~P}\binits{D.~P.}},
  \bauthor{\bsnm{Kumar},~\bfnm{Abhishek}\binits{A.}},
  \bauthor{\bsnm{Ermon},~\bfnm{Stefano}\binits{S.}} \AND
  \bauthor{\bsnm{Poole},~\bfnm{Ben}\binits{B.}}
(\byear{2021}).
\btitle{Score-based generative modeling through stochastic differential
  equations}.
In \bbooktitle{International Conference on Learning Representations}.
\end{binproceedings}
\endbibitem

\bibitem{srivastava2023beyond}
\begin{barticle}[author]
\bauthor{\bsnm{Srivastava},~\bfnm{Aarohi}\binits{A.}},
  \bauthor{\bsnm{Rastogi},~\bfnm{Abhinav}\binits{A.}},
  \bauthor{\bsnm{Rao},~\bfnm{Abhishek}\binits{A.}},
  \bauthor{\bsnm{Shoeb},~\bfnm{Abu Awal~Md}\binits{A.~A.~M.}},
  \bauthor{\bsnm{Abid},~\bfnm{Abubakar}\binits{A.}},
  \bauthor{\bsnm{Fisch},~\bfnm{Adam}\binits{A.}},
  \bauthor{\bsnm{Brown},~\bfnm{Adam~R}\binits{A.~R.}},
  \bauthor{\bsnm{Santoro},~\bfnm{Adam}\binits{A.}},
  \bauthor{\bsnm{Gupta},~\bfnm{Aditya}\binits{A.}},
  \bauthor{\bsnm{Garriga-Alonso},~\bfnm{Adri{\`a}}\binits{A.}} \betal{et~al.}
(\byear{2023}).
\btitle{Beyond the imitation game: Quantifying and extrapolating the
  capabilities of language models}.
\bjournal{Transactions on machine learning research}.
\end{barticle}
\endbibitem

\bibitem{stiennon2020learning}
\begin{barticle}[author]
\bauthor{\bsnm{Stiennon},~\bfnm{Nisan}\binits{N.}},
  \bauthor{\bsnm{Ouyang},~\bfnm{Long}\binits{L.}},
  \bauthor{\bsnm{Wu},~\bfnm{Jeffrey}\binits{J.}},
  \bauthor{\bsnm{Ziegler},~\bfnm{Daniel}\binits{D.}},
  \bauthor{\bsnm{Lowe},~\bfnm{Ryan}\binits{R.}},
  \bauthor{\bsnm{Voss},~\bfnm{Chelsea}\binits{C.}},
  \bauthor{\bsnm{Radford},~\bfnm{Alec}\binits{A.}},
  \bauthor{\bsnm{Amodei},~\bfnm{Dario}\binits{D.}} \AND
  \bauthor{\bsnm{Christiano},~\bfnm{Paul~F}\binits{P.~F.}}
(\byear{2020}).
\btitle{Learning to summarize with human feedback}.
\bjournal{Advances in neural information processing systems}
\bvolume{33}
\bpages{3008--3021}.
\end{barticle}
\endbibitem

\bibitem{su2024roformer}
\begin{barticle}[author]
\bauthor{\bsnm{Su},~\bfnm{Jianlin}\binits{J.}},
  \bauthor{\bsnm{Ahmed},~\bfnm{Murtadha}\binits{M.}},
  \bauthor{\bsnm{Lu},~\bfnm{Yu}\binits{Y.}},
  \bauthor{\bsnm{Pan},~\bfnm{Shengfeng}\binits{S.}},
  \bauthor{\bsnm{Bo},~\bfnm{Wen}\binits{W.}} \AND
  \bauthor{\bsnm{Liu},~\bfnm{Yunfeng}\binits{Y.}}
(\byear{2024}).
\btitle{Ro{F}ormer: Enhanced transformer with rotary position embedding}.
\bjournal{Neurocomputing}
\bvolume{568}
\bpages{127063}.
\end{barticle}
\endbibitem

\bibitem{sun2024generative}
\begin{binproceedings}[author]
\bauthor{\bsnm{Sun},~\bfnm{Quan}\binits{Q.}},
  \bauthor{\bsnm{Cui},~\bfnm{Yufeng}\binits{Y.}},
  \bauthor{\bsnm{Zhang},~\bfnm{Xiaosong}\binits{X.}},
  \bauthor{\bsnm{Zhang},~\bfnm{Fan}\binits{F.}},
  \bauthor{\bsnm{Yu},~\bfnm{Qiying}\binits{Q.}},
  \bauthor{\bsnm{Wang},~\bfnm{Yueze}\binits{Y.}},
  \bauthor{\bsnm{Rao},~\bfnm{Yongming}\binits{Y.}},
  \bauthor{\bsnm{Liu},~\bfnm{Jingjing}\binits{J.}},
  \bauthor{\bsnm{Huang},~\bfnm{Tiejun}\binits{T.}} \AND
  \bauthor{\bsnm{Wang},~\bfnm{Xinlong}\binits{X.}}
(\byear{2024}).
\btitle{Generative multimodal models are in-context learners}.
In \bbooktitle{Proceedings of the IEEE/CVF Conference on Computer Vision and
  Pattern Recognition}
\bpages{14398--14409}.
\end{binproceedings}
\endbibitem

\bibitem{suzgun2023challenging}
\begin{binproceedings}[author]
\bauthor{\bsnm{Suzgun},~\bfnm{Mirac}\binits{M.}},
  \bauthor{\bsnm{Scales},~\bfnm{Nathan}\binits{N.}},
  \bauthor{\bsnm{Sch{\"a}rli},~\bfnm{Nathanael}\binits{N.}},
  \bauthor{\bsnm{Gehrmann},~\bfnm{Sebastian}\binits{S.}},
  \bauthor{\bsnm{Tay},~\bfnm{Yi}\binits{Y.}},
  \bauthor{\bsnm{Chung},~\bfnm{Hyung~Won}\binits{H.~W.}},
  \bauthor{\bsnm{Chowdhery},~\bfnm{Aakanksha}\binits{A.}},
  \bauthor{\bsnm{Le},~\bfnm{Quoc}\binits{Q.}},
  \bauthor{\bsnm{Chi},~\bfnm{Ed}\binits{E.}},
  \bauthor{\bsnm{Zhou},~\bfnm{Denny}\binits{D.}} \betal{et~al.}
(\byear{2023}).
\btitle{Challenging big-bench tasks and whether chain-of-thought can solve
  them}.
In \bbooktitle{Findings of the Association for Computational Linguistics: ACL
  2023}
\bpages{13003--13051}.
\end{binproceedings}
\endbibitem

\bibitem{team2023gemini}
\begin{barticle}[author]
\bauthor{\bsnm{Team},~\bfnm{Gemini}\binits{G.}},
  \bauthor{\bsnm{Anil},~\bfnm{Rohan}\binits{R.}},
  \bauthor{\bsnm{Borgeaud},~\bfnm{Sebastian}\binits{S.}},
  \bauthor{\bsnm{Alayrac},~\bfnm{Jean-Baptiste}\binits{J.-B.}},
  \bauthor{\bsnm{Yu},~\bfnm{Jiahui}\binits{J.}},
  \bauthor{\bsnm{Soricut},~\bfnm{Radu}\binits{R.}},
  \bauthor{\bsnm{Schalkwyk},~\bfnm{Johan}\binits{J.}},
  \bauthor{\bsnm{Dai},~\bfnm{Andrew~M}\binits{A.~M.}},
  \bauthor{\bsnm{Hauth},~\bfnm{Anja}\binits{A.}},
  \bauthor{\bsnm{Millican},~\bfnm{Katie}\binits{K.}} \betal{et~al.}
(\byear{2023}).
\btitle{Gemini: a family of highly capable multimodal models}.
\bjournal{arXiv preprint arXiv:2312.11805}.
\end{barticle}
\endbibitem

\bibitem{team2025kimi}
\begin{barticle}[author]
\bauthor{\bsnm{Team},~\bfnm{Kimi}\binits{K.}},
  \bauthor{\bsnm{Zhang},~\bfnm{Yu}\binits{Y.}},
  \bauthor{\bsnm{Lin},~\bfnm{Zongyu}\binits{Z.}},
  \bauthor{\bsnm{Yao},~\bfnm{Xingcheng}\binits{X.}},
  \bauthor{\bsnm{Hu},~\bfnm{Jiaxi}\binits{J.}},
  \bauthor{\bsnm{Meng},~\bfnm{Fanqing}\binits{F.}},
  \bauthor{\bsnm{Liu},~\bfnm{Chengyin}\binits{C.}},
  \bauthor{\bsnm{Men},~\bfnm{Xin}\binits{X.}},
  \bauthor{\bsnm{Yang},~\bfnm{Songlin}\binits{S.}},
  \bauthor{\bsnm{Li},~\bfnm{Zhiyuan}\binits{Z.}} \betal{et~al.}
(\byear{2025}).
\btitle{Kimi Linear: An Expressive, Efficient Attention Architecture}.
\bjournal{arXiv preprint arXiv:2510.26692}.
\end{barticle}
\endbibitem

\bibitem{nllb2024scaling}
\begin{barticle}[author]
\bauthor{\bsnm{Team},~\bfnm{NLLB}\binits{N.}}
(\byear{2024}).
\btitle{Scaling neural machine translation to 200 languages}.
\bjournal{Nature}
\bvolume{630}
\bpages{841--846}.
\end{barticle}
\endbibitem

\bibitem{MathWorks_ChooseODE}
\begin{bmisc}[author]
\bauthor{\bsnm{{The MathWorks, Inc. }}}
(\byear{2025}).
\btitle{Choose an ODE Solver — MATLAB \& Simulink}.
\bhowpublished{\url{https://www.mathworks.com/help/matlab/math/choose-an-ode-solver.html}}.
\bnote{Accessed: 2025-11-05}.
\end{bmisc}
\endbibitem

\bibitem{touvron2023llama}
\begin{barticle}[author]
\bauthor{\bsnm{Touvron},~\bfnm{Hugo}\binits{H.}},
  \bauthor{\bsnm{Lavril},~\bfnm{Thibaut}\binits{T.}},
  \bauthor{\bsnm{Izacard},~\bfnm{Gautier}\binits{G.}},
  \bauthor{\bsnm{Martinet},~\bfnm{Xavier}\binits{X.}},
  \bauthor{\bsnm{Lachaux},~\bfnm{Marie-Anne}\binits{M.-A.}},
  \bauthor{\bsnm{Lacroix},~\bfnm{Timoth{\'e}e}\binits{T.}},
  \bauthor{\bsnm{Rozi{\`e}re},~\bfnm{Baptiste}\binits{B.}},
  \bauthor{\bsnm{Goyal},~\bfnm{Naman}\binits{N.}},
  \bauthor{\bsnm{Hambro},~\bfnm{Eric}\binits{E.}},
  \bauthor{\bsnm{Azhar},~\bfnm{Faisal}\binits{F.}} \betal{et~al.}
(\byear{2023}).
\btitle{{LLaMA}: Open and efficient foundation language models}.
\bjournal{arXiv preprint arXiv:2302.13971}.
\end{barticle}
\endbibitem

\bibitem{turner2023introduction}
\begin{barticle}[author]
\bauthor{\bsnm{Turner},~\bfnm{Richard~E}\binits{R.~E.}}
(\byear{2023}).
\btitle{An introduction to transformers}.
\bjournal{arXiv preprint arXiv:2304.10557}.
\end{barticle}
\endbibitem

\bibitem{vaswani2017attention}
\begin{binproceedings}[author]
\bauthor{\bsnm{Vaswani},~\bfnm{Ashish}\binits{A.}},
  \bauthor{\bsnm{Shazeer},~\bfnm{Noam}\binits{N.}},
  \bauthor{\bsnm{Parmar},~\bfnm{Niki}\binits{N.}},
  \bauthor{\bsnm{Uszkoreit},~\bfnm{Jakob}\binits{J.}},
  \bauthor{\bsnm{Jones},~\bfnm{Llion}\binits{L.}},
  \bauthor{\bsnm{Gomez},~\bfnm{Aidan~N}\binits{A.~N.}},
  \bauthor{\bsnm{Kaiser},~\bfnm{\L~ukasz}\binits{L.~u.}} \AND
  \bauthor{\bsnm{Polosukhin},~\bfnm{Illia}\binits{I.}}
(\byear{2017}).
\btitle{Attention is all you need}.
In \bbooktitle{Advances in Neural Information Processing Systems}
(\beditor{\bfnm{I.}\binits{I.}~\bsnm{Guyon}},
  \beditor{\bfnm{U.~Von}\binits{U.~V.}~\bsnm{Luxburg}},
  \beditor{\bfnm{S.}\binits{S.}~\bsnm{Bengio}},
  \beditor{\bfnm{H.}\binits{H.}~\bsnm{Wallach}},
  \beditor{\bfnm{R.}\binits{R.}~\bsnm{Fergus}},
  \beditor{\bfnm{S.}\binits{S.}~\bsnm{Vishwanathan}} \AND
  \beditor{\bfnm{R.}\binits{R.}~\bsnm{Garnett}}, eds.)
\bvolume{30}.
\bpublisher{Curran Associates, Inc.}
\end{binproceedings}
\endbibitem

\bibitem{villalobos2022will}
\begin{barticle}[author]
\bauthor{\bsnm{Villalobos},~\bfnm{Pablo}\binits{P.}},
  \bauthor{\bsnm{Ho},~\bfnm{Anson}\binits{A.}},
  \bauthor{\bsnm{Sevilla},~\bfnm{Jaime}\binits{J.}},
  \bauthor{\bsnm{Besiroglu},~\bfnm{Tamay}\binits{T.}},
  \bauthor{\bsnm{Heim},~\bfnm{Lennart}\binits{L.}} \AND
  \bauthor{\bsnm{Hobbhahn},~\bfnm{Marius}\binits{M.}}
(\byear{2022}).
\btitle{Will we run out of data? Limits of {LLM} scaling based on
  human-generated data}.
\bjournal{arXiv preprint arXiv:2211.04325}.
\end{barticle}
\endbibitem

\bibitem{waleffe2024empirical}
\begin{barticle}[author]
\bauthor{\bsnm{Waleffe},~\bfnm{Roger}\binits{R.}},
  \bauthor{\bsnm{Byeon},~\bfnm{Wonmin}\binits{W.}},
  \bauthor{\bsnm{Riach},~\bfnm{Duncan}\binits{D.}},
  \bauthor{\bsnm{Norick},~\bfnm{Brandon}\binits{B.}},
  \bauthor{\bsnm{Korthikanti},~\bfnm{Vijay}\binits{V.}},
  \bauthor{\bsnm{Dao},~\bfnm{Tri}\binits{T.}},
  \bauthor{\bsnm{Gu},~\bfnm{Albert}\binits{A.}},
  \bauthor{\bsnm{Hatamizadeh},~\bfnm{Ali}\binits{A.}},
  \bauthor{\bsnm{Singh},~\bfnm{Sudhakar}\binits{S.}},
  \bauthor{\bsnm{Narayanan},~\bfnm{Deepak}\binits{D.}} \betal{et~al.}
(\byear{2024}).
\btitle{An empirical study of mamba-based language models}.
\bjournal{arXiv preprint arXiv:2406.07887}.
\end{barticle}
\endbibitem

\bibitem{wan2023debias}
\begin{barticle}[author]
\bauthor{\bsnm{Wan},~\bfnm{Zhong~Yi}\binits{Z.~Y.}},
  \bauthor{\bsnm{Baptista},~\bfnm{Ricardo}\binits{R.}},
  \bauthor{\bsnm{Boral},~\bfnm{Anudhyan}\binits{A.}},
  \bauthor{\bsnm{Chen},~\bfnm{Yi-Fan}\binits{Y.-F.}},
  \bauthor{\bsnm{Anderson},~\bfnm{John}\binits{J.}},
  \bauthor{\bsnm{Sha},~\bfnm{Fei}\binits{F.}} \AND
  \bauthor{\bsnm{Zepeda-N{\'u}{\~n}ez},~\bfnm{Leonardo}\binits{L.}}
(\byear{2023}).
\btitle{Debias coarsely, sample conditionally: Statistical downscaling through
  optimal transport and probabilistic diffusion models}.
\bjournal{Advances in Neural Information Processing Systems}
\bvolume{36}
\bpages{47749--47763}.
\end{barticle}
\endbibitem

\bibitem{wang2022self}
\begin{barticle}[author]
\bauthor{\bsnm{Wang},~\bfnm{Xuezhi}\binits{X.}},
  \bauthor{\bsnm{Wei},~\bfnm{Jason}\binits{J.}},
  \bauthor{\bsnm{Schuurmans},~\bfnm{Dale}\binits{D.}},
  \bauthor{\bsnm{Le},~\bfnm{Quoc}\binits{Q.}},
  \bauthor{\bsnm{Chi},~\bfnm{Ed}\binits{E.}},
  \bauthor{\bsnm{Narang},~\bfnm{Sharan}\binits{S.}},
  \bauthor{\bsnm{Chowdhery},~\bfnm{Aakanksha}\binits{A.}} \AND
  \bauthor{\bsnm{Zhou},~\bfnm{Denny}\binits{D.}}
(\byear{2022}).
\btitle{Self-consistency improves chain of thought reasoning in language
  models}.
\bjournal{arXiv preprint arXiv:2203.11171}.
\end{barticle}
\endbibitem

\bibitem{wangself}
\begin{binproceedings}[author]
\bauthor{\bsnm{Wang},~\bfnm{Xuezhi}\binits{X.}},
  \bauthor{\bsnm{Wei},~\bfnm{Jason}\binits{J.}},
  \bauthor{\bsnm{Schuurmans},~\bfnm{Dale}\binits{D.}},
  \bauthor{\bsnm{Le},~\bfnm{Quoc~V}\binits{Q.~V.}},
  \bauthor{\bsnm{Chi},~\bfnm{Ed~H}\binits{E.~H.}},
  \bauthor{\bsnm{Narang},~\bfnm{Sharan}\binits{S.}},
  \bauthor{\bsnm{Chowdhery},~\bfnm{Aakanksha}\binits{A.}} \AND
  \bauthor{\bsnm{Zhou},~\bfnm{Denny}\binits{D.}}
(\byear{2023}).
\btitle{Self-Consistency Improves Chain of Thought Reasoning in Language
  Models}.
In \bbooktitle{The Eleventh International Conference on Learning
  Representations}.
\end{binproceedings}
\endbibitem

\bibitem{wang2024chain}
\begin{barticle}[author]
\bauthor{\bsnm{Wang},~\bfnm{Xuezhi}\binits{X.}} \AND
  \bauthor{\bsnm{Zhou},~\bfnm{Denny}\binits{D.}}
(\byear{2024}).
\btitle{Chain-of-thought reasoning without prompting}.
\bjournal{Advances in Neural Information Processing Systems}
\bvolume{37}
\bpages{66383--66409}.
\end{barticle}
\endbibitem

\bibitem{wei2022chain}
\begin{barticle}[author]
\bauthor{\bsnm{Wei},~\bfnm{Jason}\binits{J.}},
  \bauthor{\bsnm{Wang},~\bfnm{Xuezhi}\binits{X.}},
  \bauthor{\bsnm{Schuurmans},~\bfnm{Dale}\binits{D.}},
  \bauthor{\bsnm{Bosma},~\bfnm{Maarten}\binits{M.}},
  \bauthor{\bsnm{Xia},~\bfnm{Fei}\binits{F.}},
  \bauthor{\bsnm{Chi},~\bfnm{Ed}\binits{E.}},
  \bauthor{\bsnm{Le},~\bfnm{Quoc~V}\binits{Q.~V.}},
  \bauthor{\bsnm{Zhou},~\bfnm{Denny}\binits{D.}} \betal{et~al.}
(\byear{2022}).
\btitle{Chain-of-thought prompting elicits reasoning in large language models}.
\bjournal{Advances in neural information processing systems}
\bvolume{35}
\bpages{24824--24837}.
\end{barticle}
\endbibitem

\bibitem{wu2024next}
\begin{binproceedings}[author]
\bauthor{\bsnm{Wu},~\bfnm{Shengqiong}\binits{S.}},
  \bauthor{\bsnm{Fei},~\bfnm{Hao}\binits{H.}},
  \bauthor{\bsnm{Qu},~\bfnm{Leigang}\binits{L.}},
  \bauthor{\bsnm{Ji},~\bfnm{Wei}\binits{W.}} \AND
  \bauthor{\bsnm{Chua},~\bfnm{Tat-Seng}\binits{T.-S.}}
(\byear{2024}).
\btitle{Next-{GPT}: Any-to-any multimodal {LLM}}.
In \bbooktitle{Forty-first International Conference on Machine Learning}.
\end{binproceedings}
\endbibitem

\bibitem{xiong2020layer}
\begin{binproceedings}[author]
\bauthor{\bsnm{Xiong},~\bfnm{Ruibin}\binits{R.}},
  \bauthor{\bsnm{Yang},~\bfnm{Yunchang}\binits{Y.}},
  \bauthor{\bsnm{He},~\bfnm{Di}\binits{D.}},
  \bauthor{\bsnm{Zheng},~\bfnm{Kai}\binits{K.}},
  \bauthor{\bsnm{Zheng},~\bfnm{Shuxin}\binits{S.}},
  \bauthor{\bsnm{Xing},~\bfnm{Chen}\binits{C.}},
  \bauthor{\bsnm{Zhang},~\bfnm{Huishuai}\binits{H.}},
  \bauthor{\bsnm{Lan},~\bfnm{Yanyan}\binits{Y.}},
  \bauthor{\bsnm{Wang},~\bfnm{Liwei}\binits{L.}} \AND
  \bauthor{\bsnm{Liu},~\bfnm{Tieyan}\binits{T.}}
(\byear{2020}).
\btitle{On layer normalization in the transformer architecture}.
In \bbooktitle{International conference on machine learning}
\bpages{10524--10533}.
\bpublisher{PMLR}.
\end{binproceedings}
\endbibitem

\bibitem{xu2023llm}
\begin{barticle}[author]
\bauthor{\bsnm{Xu},~\bfnm{Xilie}\binits{X.}},
  \bauthor{\bsnm{Kong},~\bfnm{Keyi}\binits{K.}},
  \bauthor{\bsnm{Liu},~\bfnm{Ning}\binits{N.}},
  \bauthor{\bsnm{Cui},~\bfnm{Lizhen}\binits{L.}},
  \bauthor{\bsnm{Wang},~\bfnm{Di}\binits{D.}},
  \bauthor{\bsnm{Zhang},~\bfnm{Jingfeng}\binits{J.}} \AND
  \bauthor{\bsnm{Kankanhalli},~\bfnm{Mohan}\binits{M.}}
(\byear{2023}).
\btitle{An {LLM} can fool itself: A prompt-based adversarial attack}.
\bjournal{arXiv preprint arXiv:2310.13345}.
\end{barticle}
\endbibitem

\bibitem{yao2023tree}
\begin{barticle}[author]
\bauthor{\bsnm{Yao},~\bfnm{Shunyu}\binits{S.}},
  \bauthor{\bsnm{Yu},~\bfnm{Dian}\binits{D.}},
  \bauthor{\bsnm{Zhao},~\bfnm{Jeffrey}\binits{J.}},
  \bauthor{\bsnm{Shafran},~\bfnm{Izhak}\binits{I.}},
  \bauthor{\bsnm{Griffiths},~\bfnm{Thomas~L}\binits{T.~L.}},
  \bauthor{\bsnm{Cao},~\bfnm{Yuan}\binits{Y.}} \AND
  \bauthor{\bsnm{Narasimhan},~\bfnm{Karthik}\binits{K.}}
(\byear{2023}).
\btitle{Tree of thoughts: Deliberate problem solving with large language
  models, 2023}.
\bjournal{URL https://arxiv. org/abs/2305.10601}
\bvolume{3}
\bpages{1}.
\end{barticle}
\endbibitem

\bibitem{zehui2019dropattention}
\begin{barticle}[author]
\bauthor{\bsnm{Zehui},~\bfnm{Lin}\binits{L.}},
  \bauthor{\bsnm{Liu},~\bfnm{Pengfei}\binits{P.}},
  \bauthor{\bsnm{Huang},~\bfnm{Luyao}\binits{L.}},
  \bauthor{\bsnm{Chen},~\bfnm{Junkun}\binits{J.}},
  \bauthor{\bsnm{Qiu},~\bfnm{Xipeng}\binits{X.}} \AND
  \bauthor{\bsnm{Huang},~\bfnm{Xuanjing}\binits{X.}}
(\byear{2019}).
\btitle{DropAttention: A regularization method for fully-connected
  self-attention networks}.
\bjournal{arXiv preprint arXiv:1907.11065}.
\end{barticle}
\endbibitem

\bibitem{zhan2024anygpt}
\begin{barticle}[author]
\bauthor{\bsnm{Zhan},~\bfnm{Jun}\binits{J.}},
  \bauthor{\bsnm{Dai},~\bfnm{Junqi}\binits{J.}},
  \bauthor{\bsnm{Ye},~\bfnm{Jiasheng}\binits{J.}},
  \bauthor{\bsnm{Zhou},~\bfnm{Yunhua}\binits{Y.}},
  \bauthor{\bsnm{Zhang},~\bfnm{Dong}\binits{D.}},
  \bauthor{\bsnm{Liu},~\bfnm{Zhigeng}\binits{Z.}},
  \bauthor{\bsnm{Zhang},~\bfnm{Xin}\binits{X.}},
  \bauthor{\bsnm{Yuan},~\bfnm{Ruibin}\binits{R.}},
  \bauthor{\bsnm{Zhang},~\bfnm{Ge}\binits{G.}},
  \bauthor{\bsnm{Li},~\bfnm{Linyang}\binits{L.}} \betal{et~al.}
(\byear{2024}).
\btitle{Any{GPT}: Unified multimodal {LLM} with discrete sequence modeling}.
\bjournal{arXiv preprint arXiv:2402.12226}.
\end{barticle}
\endbibitem

\bibitem{zhang2019root}
\begin{barticle}[author]
\bauthor{\bsnm{Zhang},~\bfnm{Biao}\binits{B.}} \AND
  \bauthor{\bsnm{Sennrich},~\bfnm{Rico}\binits{R.}}
(\byear{2019}).
\btitle{Root mean square layer normalization}.
\bjournal{Advances in neural information processing systems}
\bvolume{32}.
\end{barticle}
\endbibitem

\bibitem{zhang2020pegasus}
\begin{binproceedings}[author]
\bauthor{\bsnm{Zhang},~\bfnm{Jingqing}\binits{J.}},
  \bauthor{\bsnm{Zhao},~\bfnm{Yao}\binits{Y.}},
  \bauthor{\bsnm{Saleh},~\bfnm{Mohammad}\binits{M.}} \AND
  \bauthor{\bsnm{Liu},~\bfnm{Peter}\binits{P.}}
(\byear{2020}).
\btitle{Pegasus: Pre-training with extracted gap-sentences for abstractive
  summarization}.
In \bbooktitle{International conference on machine learning}
\bpages{11328--11339}.
\bpublisher{PMLR}.
\end{binproceedings}
\endbibitem

\bibitem{zhang2025flexible}
\begin{barticle}[author]
\bauthor{\bsnm{Zhang},~\bfnm{Ruixiang}\binits{R.}},
  \bauthor{\bsnm{Zhai},~\bfnm{Shuangfei}\binits{S.}},
  \bauthor{\bsnm{Gu},~\bfnm{Jiatao}\binits{J.}},
  \bauthor{\bsnm{Zhang},~\bfnm{Yizhe}\binits{Y.}},
  \bauthor{\bsnm{Zheng},~\bfnm{Huangjie}\binits{H.}},
  \bauthor{\bsnm{Chen},~\bfnm{Tianrong}\binits{T.}},
  \bauthor{\bsnm{Bautista},~\bfnm{Miguel~Angel}\binits{M.~A.}},
  \bauthor{\bsnm{Susskind},~\bfnm{Josh}\binits{J.}} \AND
  \bauthor{\bsnm{Jaitly},~\bfnm{Navdeep}\binits{N.}}
(\byear{2025}).
\btitle{Flexible Language Modeling in Continuous Space with Transformer-based
  Autoregressive Flows}.
\bjournal{arXiv preprint arXiv:2507.00425}.
\end{barticle}
\endbibitem

\bibitem{zhang2024benchmarking}
\begin{barticle}[author]
\bauthor{\bsnm{Zhang},~\bfnm{Tianyi}\binits{T.}},
  \bauthor{\bsnm{Ladhak},~\bfnm{Faisal}\binits{F.}},
  \bauthor{\bsnm{Durmus},~\bfnm{Esin}\binits{E.}},
  \bauthor{\bsnm{Liang},~\bfnm{Percy}\binits{P.}},
  \bauthor{\bsnm{McKeown},~\bfnm{Kathleen}\binits{K.}} \AND
  \bauthor{\bsnm{Hashimoto},~\bfnm{Tatsunori~B}\binits{T.~B.}}
(\byear{2024}).
\btitle{Benchmarking large language models for news summarization}.
\bjournal{Transactions of the Association for Computational Linguistics}
\bvolume{12}
\bpages{39--57}.
\end{barticle}
\endbibitem

\bibitem{zhao2023survey}
\begin{barticle}[author]
\bauthor{\bsnm{Zhao},~\bfnm{Wayne~Xin}\binits{W.~X.}},
  \bauthor{\bsnm{Zhou},~\bfnm{Kun}\binits{K.}},
  \bauthor{\bsnm{Li},~\bfnm{Junyi}\binits{J.}},
  \bauthor{\bsnm{Tang},~\bfnm{Tianyi}\binits{T.}},
  \bauthor{\bsnm{Wang},~\bfnm{Xiaolei}\binits{X.}},
  \bauthor{\bsnm{Hou},~\bfnm{Yupeng}\binits{Y.}},
  \bauthor{\bsnm{Min},~\bfnm{Yingqian}\binits{Y.}},
  \bauthor{\bsnm{Zhang},~\bfnm{Beichen}\binits{B.}},
  \bauthor{\bsnm{Zhang},~\bfnm{Junjie}\binits{J.}},
  \bauthor{\bsnm{Dong},~\bfnm{Zican}\binits{Z.}} \betal{et~al.}
(\byear{2023}).
\btitle{A survey of large language models}.
\bjournal{arXiv preprint arXiv:2303.18223}.
\end{barticle}
\endbibitem

\bibitem{zhou2020scheduled}
\begin{binproceedings}[author]
\bauthor{\bsnm{Zhou},~\bfnm{Wangchunshu}\binits{W.}},
  \bauthor{\bsnm{Ge},~\bfnm{Tao}\binits{T.}},
  \bauthor{\bsnm{Wei},~\bfnm{Furu}\binits{F.}},
  \bauthor{\bsnm{Zhou},~\bfnm{Ming}\binits{M.}} \AND
  \bauthor{\bsnm{Xu},~\bfnm{Ke}\binits{K.}}
(\byear{2020}).
\btitle{Scheduled DropHead: A Regularization Method for Transformer Models}.
In \bbooktitle{Findings of the Association for Computational Linguistics: EMNLP
  2020}
\bpages{1971--1980}.
\end{binproceedings}
\endbibitem

\bibitem{zhu2025transformers}
\begin{binproceedings}[author]
\bauthor{\bsnm{Zhu},~\bfnm{Jiachen}\binits{J.}},
  \bauthor{\bsnm{Chen},~\bfnm{Xinlei}\binits{X.}},
  \bauthor{\bsnm{He},~\bfnm{Kaiming}\binits{K.}},
  \bauthor{\bsnm{LeCun},~\bfnm{Yann}\binits{Y.}} \AND
  \bauthor{\bsnm{Liu},~\bfnm{Zhuang}\binits{Z.}}
(\byear{2025}).
\btitle{Transformers without normalization}.
In \bbooktitle{Proceedings of the Computer Vision and Pattern Recognition
  Conference}
\bpages{14901--14911}.
\end{binproceedings}
\endbibitem

\bibitem{zhu2023promptrobust}
\begin{binproceedings}[author]
\bauthor{\bsnm{Zhu},~\bfnm{Kaijie}\binits{K.}},
  \bauthor{\bsnm{Wang},~\bfnm{Jindong}\binits{J.}},
  \bauthor{\bsnm{Zhou},~\bfnm{Jiaheng}\binits{J.}},
  \bauthor{\bsnm{Wang},~\bfnm{Zichen}\binits{Z.}},
  \bauthor{\bsnm{Chen},~\bfnm{Hao}\binits{H.}},
  \bauthor{\bsnm{Wang},~\bfnm{Yidong}\binits{Y.}},
  \bauthor{\bsnm{Yang},~\bfnm{Linyi}\binits{L.}},
  \bauthor{\bsnm{Ye},~\bfnm{Wei}\binits{W.}},
  \bauthor{\bsnm{Zhang},~\bfnm{Yue}\binits{Y.}},
  \bauthor{\bsnm{Gong},~\bfnm{Neil}\binits{N.}} \betal{et~al.}
(\byear{2023}).
\btitle{PromptRobust: Towards evaluating the robustness of large language
  models on adversarial prompts}.
In \bbooktitle{Proceedings of the 1st ACM workshop on large AI systems and
  models with privacy and safety analysis}
\bpages{57--68}.
\end{binproceedings}
\endbibitem

\bibitem{ziegler2019fine}
\begin{barticle}[author]
\bauthor{\bsnm{Ziegler},~\bfnm{Daniel~M}\binits{D.~M.}},
  \bauthor{\bsnm{Stiennon},~\bfnm{Nisan}\binits{N.}},
  \bauthor{\bsnm{Wu},~\bfnm{Jeffrey}\binits{J.}},
  \bauthor{\bsnm{Brown},~\bfnm{Tom~B}\binits{T.~B.}},
  \bauthor{\bsnm{Radford},~\bfnm{Alec}\binits{A.}},
  \bauthor{\bsnm{Amodei},~\bfnm{Dario}\binits{D.}},
  \bauthor{\bsnm{Christiano},~\bfnm{Paul}\binits{P.}} \AND
  \bauthor{\bsnm{Irving},~\bfnm{Geoffrey}\binits{G.}}
(\byear{2019}).
\btitle{Fine-tuning language models from human preferences}.
\bjournal{arXiv preprint arXiv:1909.08593}.
\end{barticle}
\endbibitem

\bibitem{ziv2003universal}
\begin{barticle}[author]
\bauthor{\bsnm{Ziv},~\bfnm{Jacob}\binits{J.}} \AND
  \bauthor{\bsnm{Lempel},~\bfnm{Abraham}\binits{A.}}
(\byear{2003}).
\btitle{A universal algorithm for sequential data compression}.
\bjournal{IEEE Transactions on information theory}
\bvolume{23}
\bpages{337--343}.
\end{barticle}
\endbibitem

\end{thebibliography}

\appendix

\section{Derivatives} \label{sec:deriv}

The next-token probability model defined in \eqref{eq:sumz} is parameterized by $\theta=(\varthetaa,\varthetae).$ The objective function $\Loss^N(\theta)$, given in \eqref{eq:lossN}, is minimized using gradient descent-based methods; \rev{thus we need to be able to differentiate the model with respect to its parameters. The following two subsections describe some details relating to the computation of the derivatives for $\varthetae$ and $\varthetaa$ in Subsections~\ref{subsec:varthetae_derivative} and~\ref{subsec:varthetaa_derivative} respectively.} These details are useful to understand how gradient-based training is implemented. The derivatives may also be useful beyond training, for example, to understand the sensitivity of the trained model to perturbations of its inputs.

\subsection{$\varthetae$ Derivative} \label{subsec:varthetae_derivative}

\rev{From  Definition \ref{def:thetam} we have that 
$\phi: \Voc \to \R^{\de}$ is defined by matrix
$$\varthetae \in \R^{\de \times |\Voc|}, \quad \varthetae =(\varthetae_1, \cdots, \varthetae_{|\Voc|}), \,\varthetae_m \in \R^\de$$
where $\varthetae_m \in \R^\de$ for all $m \in \Voc;$ then $\phi(v)=\varthetae_v$. Definition~\ref{def:phis} extends $\phi$ to act on sequences: for  sequence $f \in \F^M$, sequence $h=\phi \circ f \in \F^M_e$ is given by
$$h = (h_1,\dots,h_M) \in \F_e^M, \quad h_m = \phi(f_m) \in \R^{\de}.$$
We have chosen to label the elements in $\Voc$ by making the identification $\Voc=\{1, \cdots, |\Voc|\}=\Se^{|\Voc|}.$}

Recall the approximation \eqref{eq:sumz} for the learnable probability distribution on sequences in token-space, noting that it is built from $\qem\bigl(\cdot\,|\cdot\,; \varthetaa\bigr): \F^{(m-1)}_e \to \mathcal{P}(\Voc).$ Sequences in $\F^{(m-1)}_e$ (Euclidean space valued) are constructed from sequences in $\F^{(m-1)}$ (token space valued) by application of map $\phi$ defined by parameter $\varthetae \in \R^{\de \times |\Voc|}.$ In particular, $h \in \F^{(m-1)}_e$ is found from $f \in \F^{(m-1)}$ via $h=\phi \circ f.$  We describe how to differentiate $\qem$ with respect to this dependence on $\varthetae.$ Noting that $\mathcal{P}(\Voc) \subset \R^{|\Voc|}$, we abstract this question to consider differentiating maps from $\R^{\de \times |\Voc|} \times \F^M$ into $\R^{|\Voc|}$ with respect to their first argument.

Now assume that we are given function $\G: \F^M_e \to \R^{|\Voc|}$ and let $D_m \G(h)= \partial_{h_m} \G(h).$
Note that, for every $h \in \F^M_e$, $D_m \G(h) \in \cL(\R^{\de} ; \R^{|\Voc|})  \cong \R^{|\Voc| \times \de}.$ For any function $\He: \R^{\de \times |\Voc|} \to R_{\He}$, where $R_{\He}$ denotes the
range of $\He$, define $D^{(e)}\He\bigl(\varthetae) = \partial_{\varthetae}\He\bigl(\varthetae) \in \cL(\R^{\de \times |\Voc|},R_{\He}).$ In particular we define $\He: \R^{\de \times |\Voc|} \times \F^M \to \R^{|\Voc|}$
by $$\He\bigl(\varthetae;f\bigr):=\G\bigl(\phi \circ f).$$
This abstract construction mimics the dependence of $\qem$ on $\phi$, and hence on $\varthetae.$
Define also $D^{(e)}\He\bigl(\varthetae,f\bigr) = \partial_{\varthetae}\He\bigl(\varthetae,f\bigr).$
For every $\bigl(\varthetae,f\bigr) \in \R^{\de \times |\Voc|} \times \F^M$ we have
 $D^{(e)}\He\bigl(\varthetae,f\bigr) \ \in \cL(\R^{\de \times |\Voc|},\R^{|\Voc|})  \cong \R^{|\Voc| \times |\Voc|\de}.$ \rev{The following lemma provides an expression for the derivative of $\He$ with respect to $\varthetae$.}

 \begin{lemma} \label{lemma:Gmap_der} Viewing $D_m \G(h)$ as an element in $\R^{|\Voc| \times \de}$ and viewing
 $D^{(e)} \He\bigl(\varthetae;f\bigr)$ as an element in $\R^{|\Voc| \times |\Voc|\de}$, we obtain
\begin{align*}
         D^{(e)} \He\bigl(\varthetae;f\bigr) = \begin{pmatrix} \sum_{m=1}^M D_m \G(\phi \circ f.) \delta_{1,f_{m}}, & \cdots &, \sum_{m=1}^M D_m \G(\phi \circ f.) \delta_{|\Voc|,f_m} \end{pmatrix}.
     \end{align*}
 \end{lemma}

\begin{proof} Let $h=\phi \circ f.$ Note that $\sum_{m=1}^M D_m \G(h) \delta_{k,f_{m}}$ appearing in the $k^{th}$ position in the preceding identity is 
an element in $\R^{|\Voc| \times \de}.$ Since there are $|\Voc|$ such matrices stacked in a row we see that the preceding identity is
indeed an element in $\R^{|\Voc| \times |\Voc|\de}$ as required. 
Since $\phi(v)=\theta_v$ and  $h_m=\phi(f_m)$,
\begin{align*}
         D^{(e)} h_m = \begin{pmatrix} \partial_{\theta_1} h_m, & \cdots &, \partial_{\theta_{|\Voc|}} h_m \end{pmatrix} = \begin{pmatrix} \delta_{1,f_{m}} I, & \cdots &, \delta_{|\Voc|,f_{m}} I \end{pmatrix}.
     \end{align*}
By the chain rule, 
$$D^{(e)} \G(h_1,\dots,h_M) = \sum_{m=1}^M D_m \G(h_1,\dots,h_M) D^{(e)} h_m.$$
From a linear algebra perspective, we have the sum of products of matrices with sizes $\R^{|\Voc| \times \de}$ and $\R^{\de \times \de|\Voc|}$, resulting in a matrix of size $\R^{|\Voc| \times \de|\Voc|}$ with the desired form.
 \end{proof}

\rev{The preceding lemma may be used, with $\G(h_{1:\ell-1})=\qol\bigl(\cdot|h_{1:\ell-1};\varthetaa\bigr)$ and $h=\phi \circ f$, to compute the derivative of $\qem$ with respect to $\varthetae.$ We give a simple example to illustrate this.}

\begin{example} \rev{Let $\F_{e}^M \equiv (\R^\de)^M$ denote $\R^\de-$valued sequence of length $M=3$ and let $\Voc=\{a,b,c\}$ be a vocabulary of three distinct elements. 
In the abstract discussion we chose to label the elements in $\Voc$ by making the identification $\Voc=\{1, \cdots, |\Voc|\}=\Se^{|\Voc|}$; here we revert to working
directly with $\Voc=\{a,b,c\}$.
Consider the embedding matrix $\varthetae=(\theta_a, \theta_b, \theta_c) \in \R^{\de \times |\Voc|}$ with $\de = 2$.  For an input sequence $f=(a,b,a),$ of length $M$, the embedding map applied to the sequence is $h = \phi \circ f = (\theta_a,\theta_b,\theta_a) \in \F_e^M$; notice that the parameter $\theta_c$ does not appear in the embedded sequence.  Now consider an arbitrary function $\G:\F_{e}^M \to\R^{|\Voc|}$ and define $A_m=D_m\G(h)\in\R^{3\times 2}$ for $m = 1,2,3$. Then, Lemma~\ref{lemma:Gmap_der} gives, for $\He(\varthetae;f)=\G(h),$
$$
D^{(e)}\He(\varthetae;f) = \begin{pmatrix}
\sum_{m=1}^3 A_m\delta_{a,f_m}, 
&
\sum_{m=1}^3 A_m\delta_{b,f_m}, 
&
\sum_{m=1}^3 A_m\delta_{c,f_m}
\end{pmatrix} =
\begin{pmatrix}
A_1+A_3, & A_2, & 0
\end{pmatrix},$$
an element in $\R^{3 \times 6}.$}
\end{example}

\subsection{$\varthetaa$ Derivative} \label{subsec:varthetaa_derivative}
Here we discuss evaluation of the derivatives of the next-token probability model \eqref{eq:sumz} 
with respect to the model parameters $\varthetaa$: the
parameters defining the linear transformation $\Bl$ and the set transformer $\mathsf{T}$ in Subsection~\ref{ssec:sett2}. 
Recall that the conditional probabilities needed to define the model take the form
$$\qol\bigl(\cdot|h_{1:\ell-1};\varthetaa\bigr) = \sigma\bigl(\Bl \circ \mathsf{T}(h_{1:\ell-1})/\tau\bigr).$$
where we may write $\varthetaa=\Bigl(\theta^{(b)},\theta^{(t)}\Bigr)$ and $\theta^{(b)}$ (resp. $\theta^{(t)}$) are the
parameters entering $\Bl$ (resp. $\T.$) We recall that $\Bl \in \mathcal{L}(\R^{\de}; \R^{|\Voc|}) \cong \R^{|\Voc| \times \de}$
and we may simply view $\theta^{(b)}$ as an element of $\R^{|\Voc| \times \de} \equiv \R^{d_b}$, with $d_b=|\Voc|\de$ (although learning parameterization of $\Bl$ in lower dimension is also possible.) We have thus identified $\Bl$ with $\theta^{(b)}$ and its 
derivatives with respect to $\theta^{(b)}$ is simply the identity map. The parameterization of $\T$ is through the parameters of each attention map $\Ac_k$ appearing in \eqref{eq:settdef}, and lies in a Euclidean space isomorphic to $\R^{d_T}.$ The derivative of
$\T$ with respect to $\theta^{(t)}$ is thus an element of  $\mathcal{L}(\R^{d_T}; \R^{\de}).$ We also note that
$\sigma(z) \colon \R^{|\Voc|} \rightarrow \R^{|\Voc|}$ and $\sigma'(z) \in \mathcal{L}(\R^{|\Voc|};\R^{|\Voc|}) \cong \R^{|\Voc| \times |\Voc|}$.
Combining these derivatives we may use the chain rule to compute derivatives of $\qem\bigl(\cdot|h_{1:\ell-1};\varthetaa\bigr)$ with respect to
$\varthetaa.$

\section{Probabilistic Perspective on the Set Transformer Architecture} \label{sec:arch_probabilistic}

In this section, we interpret the set transformer architecture described in Section~\ref{sec:arch} through attention maps acting on probability measures on $\R^{\de}$, rather than on sequences of arbitrary length in $\R^{\de}$. This perspective is useful for applications of the set
transformer outside LLMs. It is also useful for the interpretation of LLMs and how text sequences
are represented. 

\paragraph{Encoding Sequences as Empirical Measures.} The first step in this interpretation of
the set transformer is to encode sequences in a manner that is independent of their length. 
To this end, we define an encoding operator
\begin{equation}
    \label{eq:emp}
\Em: \bigcup_{\ell=1}^L \F^{\ell}_e \to \prob(\R^{\de}), \quad \Em h_{1:\ell}=\frac{1}{\ell} \sum_{k=1}^{\ell} \delta_{h_k}.
\end{equation}
Now recall the notation for a sequence $h_{1:m-1}:=(h_1, \cdots, h_{m-1}) \in \F^{(m-1)}_e$. We process this sequence to obtain a probability measure in $\R^{\de}$ using the empirical measure
\begin{equation}
    \label{eq:emp2}
p^{(m-1)}:= \Em h_{1:m-1}=\frac{1}{(m-1)} \sum_{k=1}^{m-1} \delta_{h_k}.
\end{equation}

\paragraph{Set Transformer Acting on Probability Measures.}

The architecture is composed of parameterized \emph{attention} maps $\Asf_k \colon \mathcal{P}(\R^{\de}) \rightarrow \mathcal{P}(\R^{\de})$, seen as acting between spaces of probability measures, and a pooling layer $\mathfrak{P} \colon \mathcal{P}(\R^{\de}) \rightarrow \R^{\de}$, seen as mapping probability measures into Euclidean space. In particular, the set transformer $\mathsf{T}: \F^{(m-1)}_e \to \R^{\de}$ is defined by
\begin{equation}
\label{eq:settdef2}    
\mathsf{T}(h_{1:m-1}) = \mathfrak{P} \circ \Asf_K \circ \dots \Asf_1(p^{(m -1)}),
\end{equation}
where $p^{(m -1)}=\Em h_{1:m-1} $ is defined by \eqref{eq:emp2}. We describe the
attention $\Asf_k$ and pooling $\mathfrak{P}$ components of the transformer next.

\paragraph{Attention Maps.} Given a probability density\footnote{We allow empirical measures to be included in this definition of ``density''} $p \in \mathcal{P}(\R^{\de})$, we define the map $\Asf(p)  = \A(\cdot\,;p)_\sharp p$ by the pushforward of a transport map $\A(\cdot\,; p)  : \R^{\de} \to \R^{\de}$. To define the transport, we need the following definition of the $s-$parameterized density $\pi(\cdot;s)$ found by reweighting $p$:
\begin{align*}
\pi(u;s) &= \frac{1}{Z(s)}\exp\bigl(\langle Qs,Ku\rangle_{\R^{c}}\bigr)p(u),\\
Z(s) & = \int_{\R^{\de}} \exp\bigl(\langle Qs,Ku\rangle_{\R^{c}}\bigr)p(u)du.
\end{align*}
The map $\A$ is then defined by the following sequence of operations:
\begin{subequations} \label{eq:ow}
\begin{align}
\A(s;p) & \leftarrow s+ V\,\mathbb{E}_{u\sim \pi(\cdot;s)}[u], \quad u \in \R^{\de}, \label{eq:owa}\\
\A(s;p) & \leftarrow \NO\bigl(\A(s;p); \A(\cdot\,;p)_\sharp p\bigr),\\
\A(s;p) & \leftarrow s+\NN\bigl(\A(s;p)\bigr), \label{eq:owc} \\
\A(s;p) & \leftarrow \NO\bigl(\A(s;p); \A(\cdot\,;p)_\sharp p\bigr),
\end{align}
\end{subequations}
where $\NO:\R^{\de} \times \mathcal{P}^{(m-1)}_e \to \F^{(m-1)}_e$ and $\NN \colon \R^{\de} \rightarrow \R^{\de}$ denote normalization and neural network maps, respectively. These maps are defined in Section~\ref{ssec:mopm2} as acting elementwise, which here we extend to also depend on an input measure for the normalization map. Example~\ref{ex:batch_norm} provide an example of a normalization map that also depends on a probability distribution. 

We have defined $\A$ as acting on general probability densities, and note that the definition is readily interpreted to include action on empirical measures $p=p^{(m-1)}$. In particular, when the map $\A$ is both parameterized by an empirical measure and it is evaluated at each element in the empirical measure, i.e., $h_k \leftarrow \A(h_k\,; p^{(m-1)})$ for $k = 1,\dots,m-1$ and the four steps in~\eqref{eq:ow}, we arrive at the definition of the attention block as acting on sequences in~\eqref{eq:ow2}.

\begin{example}[Batch Normalization] \label{ex:batch_norm} This normalization depends on input measure $\pi$ on $\R^{\de}$. It is defined by computing the mean and variance of $\pi$ and applying the operation
$$\NO(u;\pi) =  D\frac{(u - \mathbb{E}_{u \sim \pi}[u]))}{\sqrt{\mathbb{V}_{u \sim \pi}[u] + \varepsilon}} + b,$$
where $D \in \R^{\de \times \de}$ and $b \in \R^{\de}$ are trainable parameters as in Example~\ref{ex:ln}; $\varepsilon > 0$ is a user-prescribed stabilization constant. 

This normalization is typically applied when evaluating the attention map on a batch of input sequences $h^{(1)}, \dots, h^{(N)}$ where each sequence $h^{(n)} = (h_1^{(n)}, \dots, h_\ell^{(n)})$ for a given length $\ell$. By assuming the sequence elements and batches are independent, $\pi$ can be defined as the empirical measure 
$$\pi = \frac{1}{\ell n} \sum_{k=1}^\ell \sum_{n=1}^N \delta_{h_k^{(n)}}.$$
By using larger batch sizes, we can expect to learn parameters that generalize well to input measures that are representative of the population-level limit over sequences defined in Data Assumption~\ref{das:sequences}.
\end{example}

\paragraph{Pooling Operator} The operator $\mathfrak{P} \colon \mathcal{P}(\R^{\de}) \rightarrow \R^{\de}$ is a function acting on probability measures and
returning a vector. This function includes the case of linear functionals on measures, written for $\mathfrak{f}: \R^{\de} \mapsto \R^{\de}$, as
\begin{equation} \label{eq:lf}
\mathfrak{P}(p) = \int \mathfrak{f}(h)dp(h).
\end{equation}
In the context of the set transformer, $\mathfrak{P}$ is defined so that, when it acts on an empirical measure defined from a sequence as
in \eqref{eq:emp}, the overall mapping is both  permutation-invariant and is well-defined for input sequences of arbitrary length. The following example contains the population-level variant of the widely used choices defined in Example~\ref{ex:pooling} that satisfy these properties. 

\begin{example} The mean pooling operation is given by choosing $\mathfrak{f}(h)=h,$ which yields the operator %
$$\mathfrak{P}(p) = \int h dp(h).$$
The max pooling operation is defined by choosing the elementwise operation $\mathfrak{f}(h)_j = \max((h)_j)$ for each $h \in \R^{\de}$ and $j = 1,\dots,\de$. The continuum analog underlying max pooling is to choose to compute the elementwise essential supremum, which can be defined as a limit:
$$\mathfrak{P}(p)_j = \lim_{\beta \rightarrow \infty} \frac{1}{\beta} \log\left(\int e^{\beta (h)_j} p(h) dh\right).$$
\end{example}

\section{Continuum State-Space Models} \label{sec:ssmcts}

In this section we provide a continuous time perspective on the state-space models introduced in Subsection~\ref{sssec:ssm}. We view the state-space models as time-discretizations defined through a continuously evolving hidden state, driven by a continuous-in-time input.

Structured State Space Models (SSMs) map an input $h \colon [0,T] \rightarrow \R^{\de}$ into output $u \colon [0,T] \rightarrow \R^{\de}.$ This is achieved by evolving a hidden state $v(t) \in \R^{s}$ that takes $h$ as input in a causal way, and then defining the output $u$ through a function, that takes both $v$ and $h$ as input, acting pointwise in time. The model is parameterized by $\vartheta=(\vartheta_0,\vartheta_{1},\vartheta_{2})$ and takes the form 
\begin{subequations}
\label{eq:state_space}
\begin{align}
\frac{dv}{dt} &= \mathcal{F}(v, h; \vartheta_1), \quad v(0) = \vartheta_0,\\
\quad u &= \mathcal{G}(v,h; \vartheta_2);
\end{align}
\end{subequations}
here the operators $\mathcal{F}(\cdot,\cdot;\vartheta_1) \colon \R^s \times \R^{\de} \rightarrow \R^s$ and $\mathcal{G}(\cdot,\cdot;\vartheta_2) \colon \R^s \times \R^{\de} \rightarrow \R^{\de}$ define the evolution of the state and the output starting from the initial condition $\vartheta_0 \in \R^s$. One natural choice is to define operators via the affine structure 
\begin{subequations}
\label{eq:linear_operators}
\begin{align}
\mathcal{F}(v,h;\vartheta) &= Av + Bh,\\ 
\mathcal{G}(v,h;\vartheta) &= Cv + Dh,
\end{align}
\end{subequations}
so that $\vartheta$ is defined by the matrices $A \in \R^{s \times s}, B \in \R^{s \times \de}, C \in \R^{\de \times s}, D \in \R^{\de \times \de}$. That is, the model parameters are given by $\vartheta = (\vartheta_0,A,B,C,D)$. 
The dimensionality of the hidden state is commonly chosen to be (considerably) larger than the embedding dimension of the sequence elements: $s \gg \de$. 

Solving \eqref{eq:state_space}, under \eqref{eq:linear_operators}, yields
\begin{subequations}
\begin{align} %
v(t) &= e^{At} \vartheta_0 + \int_0^t e^{At} B e^{-As} h(s) ds, \\
u(t) &= Cv(t) + Dh(t),
\end{align}
\end{subequations}
which defines a map from an input $h$ to the output $u$. 

\begin{remark}
This map, discretized in time. can be used in place of the attention mechanism in~\eqref{eq:owa}. 
In practice, the input sequence is represented at discrete-times $t_k \in [0,T]$ by considering a step-size $\Delta t \geq 0$. The evolution of the hidden state $v_k \coloneqq v(t_k)$ for the linear system as a function of $(h_k)_{k =1}^N$ and $h_k \coloneqq h(t_k)$ is then given by the discrete-time equations 

\begin{subequations}
\label{eq:state_space2}
\begin{align}
v_{k+1} &= e^{A \Delta t} v_{k} + e^{A_{t_{k+1}}}\Bigl(\int_0^{\Delta t} B e^{-Az}  dz\Bigr) e^{-At_k} h_k, \\
u_{k+1} &= Cv_{k+1} + D h_{k+1}.
\end{align}
\end{subequations}
\end{remark}

\section{Proof of Reversibility of Discrete Diffusion} \label{app:discrete_reversibility}

Recall $\mathsf{p}(\cdot,t)$ solving \eqref{eq:CTMC} and $\bar{\mathsf{p}}(\cdot,t)$ solving \eqref{eq:reverseCTMC}. We have the following theorem concerning their
relationship to one another:

\begin{theorem} Let $f(\cdot): [0,T] \to \Voc$ be a sequence that follows a CTMC according to the rate matrix $Q(t) \in \R^{|\Voc| \times |\Voc|}$ and has marginal probability $\mathsf{p}\colon \Voc \times [0,T] \rightarrow \mathcal{P}(\Voc)$. \rev{Assume $p(i,t) > 0$ for all $t \in [0,T]$ and $i=1,\dots,|\Voc|$.} Now consider a sequence $\overline{f}(\cdot)$ evolving in reverse-time that follows a CTMC according to the rate matrix $\overline{Q}(t)$ with entries 
\begin{equation} \label{eq:reverse_rate}
  \overline{Q}_{ij}(t) = Q_{ji}(t) \frac{\mathsf{p}(j,t)}{\mathsf{p}(i,t)},  
\end{equation}
for all $t \in [0,T]$ and $1 \leq i,j \leq |\Voc|$. Let $\overline{f}(t)$ have marginal probability $\overline{\mathsf{p}}(\cdot,t) \colon \Voc \times [0,T] \rightarrow \mathcal{P}(\Voc)$ and run the process backwards on $[0,T]$, starting from $\overline{f}(T) \sim \mathsf{p}(\cdot,T)$. Then, $\overline{f}(t) \overset{d}{=} f(t)$: $\bar{\mathsf{p}}(\cdot,t) = \mathsf{p}(\cdot,t)$ for all $t \in [0,T]$. 
\end{theorem}

\begin{proof} If $\overline{f}$ is a CTMC, then its marginal probability $\overline{\mathsf{p}}$ evolves according to the forward equation:
\begin{align*}
\frac{d\overline{\mathsf{p}}(i,t)}{dt} &= \sum_{k \in \Voc} \overline{Q}_{ki}(t)\overline{\mathsf{p}}(k,t) \\
&= \sum_{k \neq i} \overline{Q}_{ki}(t) \overline{\mathsf{p}}(k,t) + \overline{Q}_{ii}(t) \overline{\mathsf{p}}(i,t) \\
&= \sum_{k \neq i} \overline{Q}_{ki}(t) \overline{\mathsf{p}}(k,t) + \left(- \sum_{\ell \neq i} \overline{Q}_{i\ell}(t) \right) \overline{\mathsf{p}}(i,t),
\end{align*}
where in the last line we used the row-sum property for the rate matrix $\sum_\ell \overline{Q}_{i\ell}(t) = 0$ for all $i$. Using the definition for the reverse rate matrix in~\eqref{eq:reverse_rate} we have
\begin{align*}
\frac{d\overline{\mathsf{p}}(i,t)}{dt} &= \sum_{k \neq i} Q_{ik}(t) \frac{\mathsf{p}(i,t)}{\mathsf{p}(k,t)} \overline{\mathsf{p}}(k,t) - \sum_{\ell \neq i} Q_{\ell i}(t) \frac{\mathsf{p}(\ell, t)}{\mathsf{p}(i,t)} \overline{\mathsf{p}}(i, t).
\end{align*}
Observe that taking $\overline{\mathsf{p}}(k,t) = \mathsf{p}(k,t)$ for all $k = 1,\dots,|\Voc|$, then we have 
\begin{align*}
\frac{d\mathsf{p}(i,t)}{dt} = \frac{d\overline{\mathsf{p}}(i,t)}{dt} &= \sum_{k \neq i} Q_{ik}(t) \mathsf{p}(i,t) - \sum_{\ell \neq i} Q_{\ell i}(t) \mathsf{p}(\ell,t) \\
&= (-Q_{ii}(t))\mathsf{p}(i,t) - \sum_{\ell \neq i} Q_{\ell i}(t) \mathsf{p}(\ell,t) \\
&= -\sum_{\ell \neq i} Q_{\ell i}(t)\mathsf{p}(\ell,t),
\end{align*}
which is the forward equation for the marginal probability $\mathsf{p}$ after accounting for the change in the sign for the drift from evolving the equation backwards-in-time.
By uniqueness of $\mathsf{p}(\cdot,t)$ solving \eqref{eq:CTMC} and $\bar{\mathsf{p}}(\cdot,t)$ solving \eqref{eq:reverseCTMC} the result follows.
\end{proof}

\section{Details of Prompts and Answers} \label{sec:LLMresponses}

\subsection{Question Answering} \label{ssec:aqac}

The following example supplements the discussion in Subsection \ref{ssec:qac}.

\begin{lstlisting}[language=json, mathescape=true, escapeinside={(*}{*)}]
"prompt": A man named Frank travels from NYC to Philadelphia by road, on a Monday, leaving at 9:00am and arriving at 12noon. He then travels from Philadelphia to NYC the next day, Tuesday, also leaving at 9:00am and arriving at 12 noon. He takes the same route, but on Tuesday reverses the route taken on Monday. Is it necessarily the case that there is at least one point that he will be located at the same time, albeit on the two different days? Explain your answer
"output": Yes, it is necessarily the case that there is at least one point where Frank will be located at the same time on both days. This is due to the Intermediate Value Theorem, which states that if a continuous function takes values f(a) and f(b) at points a and b, respectively, then it also takes any value between f(a) and f(b) at some point between a and b.

In this scenario, we can consider Frank's position as a continuous function of time. On Monday, he travels from NYC to Philadelphia, and on Tuesday, he travels from Philadelphia back to NYC. Since he takes the same route on both days, his position as a function of time will be continuous.

Let's denote Frank's position on Monday as f(t) and on Tuesday as g(t), where t represents time. Since Frank arrives at 12 noon on both days, we have f(12) = g(12). Now, let's consider the function h(t) = f(t) - g(t). This function represents the difference in Frank's position between Monday and Tuesday.

Since Frank's position is continuous, h(t) is also continuous. Moreover, h(9) = f(9) - g(9) represents the difference in Frank's position at 9:00am on both days, and h(12) = f(12) - g(12) = 0, as Frank arrives at the same time on both days.

By the Intermediate Value Theorem, there must be some time t between 9:00am and 12:00pm where h(t) = 0, meaning that Frank's position on Monday and Tuesday is the same at that time. Therefore, it is necessarily the case that there is at least one point where Frank will be located at the same time on both days.<|endoftext|>
\end{lstlisting}

\subsection{Conversations}\label{ssec:aconv}

The example in this subsection supplements the discussion in Subsection~\ref{ssec:conv}, on question and answering  with LLMs, by extending the example on Bayes' theorem.

\begin{lstlisting}[language=json, mathescape=true, escapeinside={(*}{*)}]

"prompt": Now explain it more formally to a mathematician.
"output": Bayes' theorem is a fundamental concept in probability theory that allows us to calculate the conditional probability of an event given prior knowledge of other related events. Mathematically, it is expressed as follows:

P(A|B) = (P(B|A) * P(A)) / P(B)

Where:
- P(A|B) is the posterior probability of event A occurring given that event B has occurred.
- P(B|A) is the likelihood of event B occurring given that event A has occurred.
- P(A) is the prior probability of event A occurring.
- P(B) is the marginal probability of event B occurring.

In simpler terms, Bayes' theorem helps us update our beliefs or probabilities about an event based on new evidence provided by other related events. It is widely used in various fields such as statistics, machine learning, and artificial intelligence to make predictions and decisions under uncertainty.<|endoftext|>
\end{lstlisting}

\subsection{Mathematics}\label{ssec:acontext}

In this subsection we compare the outputs of the LLM, with and without context, for a multiplication task. We consider the multiplication of six different pairs of integers, where each pair contains a two digit and three digit number. Our
initial experiments will use deterministic greedy-decoding. We will then compare the responses with those obtained using a stochastic sampling procedure: positive temperature $\tau$ is used at inference time; see Remark~\ref{rem:temp}.

First, we demonstrate the deterministic algorithm for calculating $473 \times 45$; the correct answer is $21285$. When context is not provided in the prompt, an incorrect answer is given:

\begin{lstlisting}[language=json, mathescape=true, escapeinside={(*}{*)}]
"prompt":  Calculate 473*45
"output": To calculate \( 473 \times 45 \), we can use the standard multiplication method:

\[
\begin{array}{r}
     473 \\
\times  45 \\
\hline
   2365 \quad (\text{This is } 473 \times 5) \\
+ 18920 \quad (\text{This is } 473 \times 40, \text{ shifted one position to the left}) \\
\hline
  21385 \\
\end{array}
\]

So, \( 473 \times 45 = 21385 \).<|im_end|>
\end{lstlisting}
However, if we provide context in the prompt, then the correct answer is given. Here, the context comprises
the necessary skills, together with an example in which integer multiplication is built-up, step-by-step, using the provided skills.
\begin{lstlisting}[language=json, mathescape=true, escapeinside={(*}{*)}]
"context": Skill <extract_digits>: Extract the digits in a number to a list. 
For example, Extract digits in 123 to D=[1,2,3]. Extract digits in 7654 to D=[7,6,5,4]. \
Skill <list_length>: Get the number of elements in a list.
For example, D=[1,2,3], len(D)=3. A=[1,2,4,5,6], len(A)=5.

Skill <mul_two_single_digit_number>: Multiply two single-digit numbers.
0*1=0 0*2=0 0*3=0 0*4=0 0*5=0 0*6=0 0*7=0 0*8=0 0*9=0
1*1=1 1*2=2 1*3=3 1*4=4 1*5=5 1*6=6 1*7=7 1*8=8 1*9=9
2*1=2 2*2=4 2*3=6 2*4=8 2*5=10 2*6=12 2*7=14 2*8=16 2*9=18
3*1=3 3*2=6 3*3=9 3*4=12 3*5=15 3*6=18 3*7=21 3*8=24 3*9=27
4*1=4 4*2=8 4*3=12 4*4=16 4*5=20 4*6=24 4*7=28 4*8=32 4*9=36
5*1=5 5*2=10 5*3=15 5*4=20 5*5=25 5*6=30 5*7=35 5*8=40 5*9=45
6*1=6 6*2=12 6*3=18 6*4=24 6*5=30 6*6=36 6*7=42 6*8=48 6*9=54
7*1=7 7*2=14 7*3=21 7*4=28 7*5=35 7*6=42 7*7=49 7*8=56 7*9=63
8*1=8 8*2=16 8*3=24 8*4=32 8*5=40 8*6=48 8*7=56 8*8=64 8*9=72
9*1=9 9*2=18 9*3=27 9*4=36 9*5=45 9*6=54 9*7=63 9*8=72 9*9=81
     
Skill <add_multiple_numbers>: Add multiple numbers such as m+n+p:
1. Add the first two numbers m+n and get the result r1=m+n.
2. Add the third number p to r1 and get the result r2=r1+p.
For example, calculate 128+64+79:
1. Add the first two numbers 128+64 and get the result r1=128+64=192.
2. Add the third number 79 to r1 and get the result r2=192+79=271.
So 128+64+79=271.

Example: Calculate 184*67:
Answer:
1. Using Skill <extract_digits> and Skill <list_length>, extract the digits in 184 to DM=[1,8,4].
len(DM)=3. Extract the digits in 67 to DN=[6,7]. len(DN)=2.
2. Add 0,1,len(DM)-1=2 zeros to the end of every number in DM=[1,8,4] according to
the position of the number in DM: DMO=[1*100,8*10,4*1]=[100,80,4].
3. Add 0,len(DN)-1=1 zeros to the end of every number in DN=[6,7] according to the
position of the number in DN: DNO=[6*10,7*1]=[60,7].
4. Using Skill <mul_two_single_digit_number>, multiple every number in DMO=[100,80,4]
with every number in DNO=[60,7] and get R=[100*60,100*7,80*60,80*7,4*60,4*7]=
[6000,700,4800,560,240,28].
5. Using Skill <add_multiple_numbers>, add all the numbers in R=[6000,700,4800,560,240,28],
6000+700+4800+560+240+28:
i. Add the first two numbers: r1=6000+700=6700.
ii. Add the third number 4800 to r1=6700: r2=6700+4800=11500.
iii. Add the fourth number 560 to r2=11500: r3=11500+560=12060.
iv. Add the fifth number 240 to r3=12060: r4=12060+240=12300.
v. Add the sixth number 28 to r4=12300: r5=12300+28=12328.
6. So the answer is 12328
"prompt": Calculate 473*45
"output": Let's break down the calculation step by step:

1. **Extract Digits and Prepare for Multiplication:**
   - Extract digits from 473: DM = [4, 7, 3]. len(DM) = 3.
   - Extract digits from 45: DN = [4, 5]. len(DN) = 2.
   
   We need to add zeros to the end of each number in DM and DN to align their positions for multiplication:
   - DMO = [4*100, 7*10, 3*1] = [400, 70, 3].
   - DNO = [4*10, 5*1] = [40, 5].

2. **Multiply Each Pair of Numbers:**
   - Now we multiply each number in DMO with each number in DNO:
     - 400 * 40 = 16000
     - 400 * 5 = 2000
     - 70 * 40 = 2800
     - 70 * 5 = 350
     - 3 * 40 = 120
     - 3 * 5 = 15

   So, the result R is: R = [16000, 2000, 2800, 350, 120, 15].

3. **Add All Numbers in Result R:**
   - Add the first two numbers: r1 = 16000 + 2000 = 18000.
   - Add the third number to r1: r2 = 18000 + 2800 = 20800.
   - Add the fourth number to r2: r3 = 20800 + 350 = 21150.
   - Add the fifth number to r3: r4 = 21150 + 120 = 21270.
   - Add the sixth number to r4: r5 = 21270 + 15 = 21285.

Therefore, the final answer is 21285.<|im_end|>
\end{lstlisting}

We now repeat the comparison between multiplication, undertaken with and without context, for three other pairs of integers. For conciseness, we only include the final line of the LLM's output without and with in-context learning (ICL). We include the true answer for the task below each prompt. We observe that while the model without ICL only produces a correct solution in one out of three examples, ICL correctly solves all three of the multiplication tasks.

\begin{lstlisting}[language=json, mathescape=true, escapeinside={(*}{*)}]
"prompt": Calculate 299*89
"truth": 26611
"output (Without ICL)": So, \( 299 \times 89 = 26911 \).<|im_end|>
"output (With ICL)": Therefore, the final answer is 26611.<|im_end|>
\end{lstlisting}

\begin{lstlisting}[language=json, mathescape=true, escapeinside={(*}{*)}]
"prompt": Calculate 378*64
"truth": 24192
"output (Without ICL)": So, \( 378 \times 64 = 24072 \).<|im_end|>
"output (With ICL)": Therefore, the final answer is 24192.<|im_end|>
\end{lstlisting}

\begin{lstlisting}[language=json, mathescape=true, escapeinside={(*}{*)}]
"prompt": Calculate 952*76 
"truth": 72352
"output (Without ICL)": \[ 952 \times 76 = 72352 \]<|im_end|>
"output (With ICL)": So the answer is 72352.<|im_end|>
\end{lstlisting}

The reader might have the impression that use of ICL, then, solves the problem of multiplication accuracy. However, this is not the case as the next example shows: here the model without ICL is correct, but adding context leads to the wrong answer:
\begin{lstlisting}[language=json, mathescape=true, escapeinside={(*}{*)}]
"prompt": Calculate 874*25
"truth": 21850
"output (Without ICL)": So, \( 874 \times 25 = 21850 \).<|im_end|>
"output (With ICL)": Therefore, \(874 \times 25 = 6128\).<|im_end|>
\end{lstlisting}

Lastly, we present  model responses with stochasticity,
exemplifying the discussion in Remark~\ref{rem:temp}. Here we set the temperature parameter $\tau=0.5$. We recall that the choice $\tau=0$, deployed in all previous experiments, results in a deterministic output, but that any $\tau>0$ will result
in stochastic output. We present the output of the stochastic algorithm for the calculation of $378 \times 64$; the correct answer is $24192$. We include two outputs below from the stochastic algorithm without ICL to demonstrate when the model does and does not produce the correct answer. We observe that the model does produce the correct answer when it uses a reasoning process in the first output.

\begin{lstlisting}[language=json, mathescape=true, escapeinside={(*}{*)}]
"prompt": Calculate 378*64
"output": To calculate \( 378 \times 64 \), we can use the standard multiplication method or a calculator for an exact result.

Using the standard multiplication method:

\[
\begin{array}{r}
     378 \\
\times  64 \\
\hline
   1512 \quad (\text{378} \times 4) \\
 22680 \quad (\text{378} \times 60, \text{shift one position to the left}) \\
\hline
 24192 \\
\end{array}
\]

So, \( 378 \times 64 = 24192 \).<|im_end|>
\end{lstlisting}

\begin{lstlisting}[language=json, mathescape=true, escapeinside={(*}{*)}]
"prompt": Calculate 378*64
"output": To calculate \( 378 \times 64 \), we can use the standard multiplication method or a calculator for an exact result.

Using a calculator, the result is:

\[ 378 \times 64 = 24048 \]<|im_end|>
\end{lstlisting}
To further probe the role of stochastic next-token prediction we compute the output five times with and without ICL. We observe that the model returns the correct answer in four out of five instances without ICL, and in all five instances with ICL.

\subsection{Coding} \label{ssec:acode}

In this example, we present the output from requesting {\sc{Matlab}} code to numerically solve an ODE, contrasting
model response with and without fine tuning. Following the code, we include the plot produced by {\sc{Matlab}} for the code resulting from both models. We start without fine-tuning:
\begin{lstlisting}[language=json, mathescape=true, escapeinside={(*}{*)}]
"prompt": Write {\sc{Matlab}} code to numerically solve a nonlinear ODE with vector field f. Only include the code in the response.
"output": Sure, here's an example of how you can use the `ode45` solver in {\sc{Matlab}} to numerically solve a nonlinear ordinary differential equation (ODE) with a given vector field `f`:

```{\sc{Matlab}}
f = @(t, y) [y(2); -y(1) - y(2)^2];

y0 = [1; 0];

tspan = [0 10];

[t, y] = ode45(f, tspan, y0);

plot(t, y(:, 1), 'b', t, y(:, 2), 'r');
xlabel('Time');
ylabel('Solution');
legend('y1', 'y2');
title('Solution of the ODE');
```

This code defines the vector field `f` as a function of time `t` and the state vector `y`. It then uses the `ode45` solver to find the solution over the specified time span `[0 10]` with initial conditions `y0 = [1; 0]`. Finally, it plots the solutions for `y1` and `y2` over time.<|endoftext|>
\end{lstlisting}
With fine-tuning we obtain a different example:
\begin{lstlisting}[language=json, mathescape=true, escapeinside={(*}{*)}]
"prompt": Write {\sc{Matlab}} code to numerically solve a nonlinear ODE with vector field f. Provide only the code for solving.
"output": Sure, here's the code to numerically solve a nonlinear ODE with vector field f using the `ode45` solver in {\sc{Matlab}}:

```{\sc{Matlab}}
function df = myVectorField(t, y)
    df = [-y + sin(t)];
end

y0 = [1]; %
tspan = [0 10]; %

[t, y] = ode45(@myVectorField, tspan, y0);

plot(t, y);
xlabel('Time (t)');
ylabel('y(t)');
title('Solution of the ODE');
grid on;
```

This code defines a vector field function `myVectorField` that represents the right-hand side of your ODE. The `ode45` function is used to solve the ODE numerically over the specified time span with the given initial conditions. Finally, the solution is plotted using `plot`. Adjust the vector field function and initial conditions according to your specific problem.<|endoftext|>
\end{lstlisting}
The following figure demonstrates the fact that the Qwen based code generator \emph{without} fine-tuning delivers a piece of {\sc{Matlab}} software which integrates an ODE that blows up on the domain of integration; this is clearly undesirable.
\emph{With} fine-tuning this issue is not present.
\begin{figure}[!htb]
    \centering
    \begin{minipage}[t]{0.48\textwidth}
        \centering
        \includegraphics[width=0.95\textwidth]{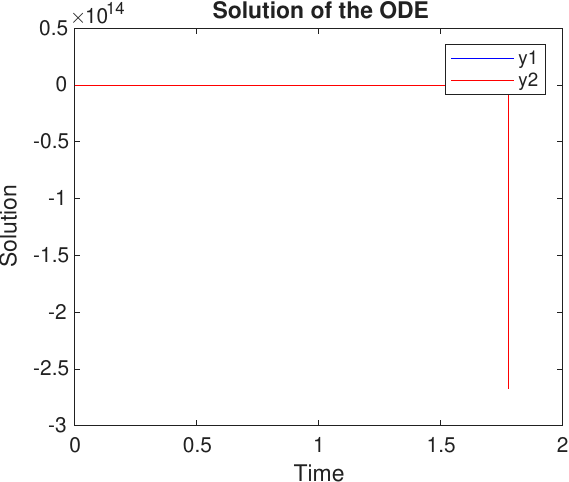}
        \caption{{\sc{Matlab}} output without fine-tuned model: Qwen-2.5 showing a solution which blows-up.}
        \label{fig:first}
    \end{minipage}
    \hfill
    \begin{minipage}[t]{0.48\textwidth}
        \centering
        \includegraphics[width=0.95\textwidth]{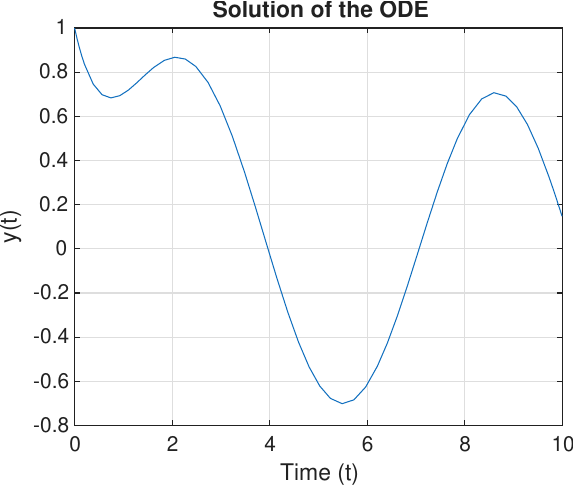}
        \caption{{\sc{Matlab}} output with fine-tuned model: Qwen-2.5-Coder showing a bounded solution.}
        \label{fig:second}
    \end{minipage}
\end{figure}

\end{document}